\documentclass[11pt, twoside]{amsart}
\calclayout
\allowdisplaybreaks

\title[Contragredient Lie algebras in symmetric categories]{Contragredient Lie algebras in symmetric categories}

\author{Iv\'an Angiono}
\address{FaMAF-CIEM (CONICET), Universidad Nacional de C\'ordoba, Medina Allende s/n, Ciudad Universitaria, 5000 C\'ordoba, Rep\'ublica Argentina}
\email{ivan.angiono@unc.edu.ar}
\urladdr{}

\author{Julia Plavnik}
\address{Department of Mathematics, Indiana University, Bloomington, IN 47405, USA.}
\email{jplavnik@iu.edu}
\urladdr{}

\author{Guillermo Sanmarco}
\address{Department of Mathematics, University of Washington, Seattle, WA 98195, USA.}
\email{sanmarco@uw.edu}
\urladdr{}

\makeatletter
\@namedef{subjclassname@2020}{%
\textup{2020} Mathematics Subject Classification}
\makeatother

\usepackage{mathabx}
\usepackage{import}
\usepackage{amsmath}
\usepackage{amsfonts}
\usepackage{amsthm}
\usepackage{amssymb,bbm}
\usepackage{mathrsfs}  
\usepackage{dsfont}
\usepackage[alphabetic, initials]{amsrefs}
\usepackage[english]{babel}
\usepackage{url}
\usepackage{fancyhdr}
\usepackage{graphicx}
\usepackage{verbatim}
\usepackage[inline]{enumitem}
\usepackage{mathtools}
\usepackage[
hypertexnames=false,
colorlinks=true,
linkcolor=black, 
anchorcolor=black,
citecolor=black,
urlcolor=black, 
]{hyperref}
\usepackage{microtype} 
\usepackage[dvipsnames]{xcolor}

\usepackage{tikz-cd}
\usepackage{todonotes}

\usepackage{cleveref}
\usepackage{adjustbox}
\usepackage[T1]{fontenc}
\usepackage{lmodern}
\usepackage{multicol}
\usepackage[all]{xy}
\SelectTips{cm}{}
\usepackage{tikz}
\usetikzlibrary{decorations.markings}
\usetikzlibrary{decorations.pathreplacing}
\usetikzlibrary{arrows,shapes,positioning}
\usepackage{tabularx}
\usepackage{caption}

\newcommand{\ad}{\operatorname{ad}}

\newcommand{\cha}{\operatorname{char}}
\newcommand{\coev}{\mathsf{coev}}

\newcommand{\ev}{\mathsf{ev}}

\newcommand{\GL}{\operatorname{GL}}
\newcommand{\gr}{\operatorname{gr}}
\newcommand{\he}{\mathrm{ht}}
\newcommand{\Hom}{\mathsf{Hom}}

\newcommand{\isomorph}{\stackrel{\sim}{\to}}

\newcommand{\one}{\mathds{1}}
\newcommand{\rk}{\operatorname{rank}}

\newcommand{\id}{\operatorname{id}}
\newcommand{\ima}{\operatorname{im}}

\newcommand{\opD}{\operatorname{D}}
\newcommand{\opRD}{\operatorname{RD}}
\newcommand{\opSD}{\operatorname{SD}}

\newcommand{\balp}{\boldsymbol\alpha}
\newcommand{\sfS}{\mathsf{S}}

\newcommand{\FLie}{\operatorname{FLie}}
\newcommand{\FOLie}{\operatorname{FOLie}}
\newcommand{\Fr}{\operatorname{Fr}}
\newcommand\restr[2]{{\left.\kern-\nulldelimiterspace #1 \vphantom{\big|} \right|_{#2}}} 

\newcommand{\Mod}{\mathsf{Mod}}
\newcommand{\Rep}{\mathsf{Rep}}

\renewcommand{\Vec}{\mathsf{Vec}_{\Bbbk}}
\newcommand{\sVec}{\mathsf{sVec}_\Bbbk}
\newcommand{\Ver}{\mathsf{Ver}}
\newcommand{\Verp}{\mathsf{Ver}_p}
\newcommand{\ot}{\otimes}
\newcommand{\Lie}{\mathsf{Lie}}
\newcommand{\OLie}{\mathsf{OLie}}
\newcommand{\gMod}[2]{\mathsf{Mod}_{#1}(#2)}
\newcommand{\indcat}[1]{#1^{\op{ind}}}

\providecommand{\op}[1]{\operatorname{#1}}
\newcommand{\ov}[1]{\overline{#1}}

\newcommand{\bI}{\mathbb{I}}

\newcommand{\bZ}{\mathbb{Z}}
\newcommand{\bN}{\mathbb{N}}
\newcommand{\bk}{\Bbbk}

\newcommand{\cA}{\mathcal{A}}
\newcommand{\cC}{\mathcal{C}}
\newcommand{\cD}{\mathcal{D}}

\newcommand{\cI}{\mathcal{I}}

\newcommand{\fg}{\mathfrak{g}}
\newcommand{\fh}{\mathfrak{h}}
\newcommand{\fk}{\mathfrak{k}}

\newcommand{\fsl}{\mathfrak{sl}}
\newcommand{\fgl}{\mathfrak{gl}}
\newcommand{\fn}{\mathfrak{n}}
\newcommand{\fm}{\mathfrak{m}}

\newcommand{\fH}{\mathfrak{H}}
\newcommand{\ft}{\mathfrak{t}}
\newcommand{\ff}{\mathfrak{f}}

\newcommand{\tfg}{\widetilde{\mathfrak{g}}}
\newcommand{\tfn}{\widetilde{\mathfrak{n}}}

\newcommand{\ufg}{\underline{\mathfrak{g}}}

\newcommand{\ati}{\widetilde{a}}
\newcommand{\bti}{\widetilde{b}}

\newcommand{\ttb}{\mathtt{b}}
\newcommand{\ttd}{\mathtt{d}}
\newcommand{\tte}{\mathtt{e}}
\newcommand{\ttf}{\mathtt{f}}
\newcommand{\ttg}{\mathtt{g}}

\newcommand{\ttq}{\mathtt{q}}
\newcommand{\ttn}{\mathtt{n}}
\newcommand{\ttm}{\mathtt{m}}

\newcommand{\ttI}{\mathtt{I}}
\newcommand{\ttK}{\mathtt{K}}
\newcommand{\ttL}{\mathtt{L}}
\newcommand{\ttV}{\mathtt{V}}
\newcommand{\ttX}{\mathtt{X}}
\newcommand{\ttB}{\mathtt{B}}
\newcommand{\ttW}{\mathtt{W}}
\newcommand{\tttg}{\widetilde{\ttg}}
\newcommand{\tttn}{\widetilde{\ttn}}

\newcommand{\tttB}{\widetilde{\ttB}}

\newcommand{\bp}{\mathbf{p}}

\newcommand{\osp}{\mathfrak{osp}}
\newcommand{\supersl}[2]{{\mathfrak{sl} }(#1|#2)}

\newcommand{\bx}{\mathbf{x}}
\newcommand{\by}{\mathbf{y}}
\newcommand{\bz}{\mathbf{z}}

\numberwithin{equation}{section}

\newtheorem*{rep@theorem}{\rep@title}
\newcommand{\newreptheorem}[2]{%
\newenvironment{rep#1}[1]{%
 \def\rep@title{#2 \ref{##1}}%
 \begin{rep@theorem}}%
 {\end{rep@theorem}}}
\makeatother

\newtheorem{theorem}{Theorem}[section]
\newtheorem{proposition}[theorem]{Proposition}
\newreptheorem{proposition}{Proposition}
\newtheorem{corollary}[theorem]{Corollary}
\newreptheorem{corollary}{Corollary}
\newtheorem{lemma}[theorem]{Lemma}

\newtheorem{theorem*}{Theorem}
\newreptheorem{theorem}{Theorem}
\newtheorem{claim}{Claim}
\usepackage{etoolbox}
\AtEndEnvironment{proof}{\setcounter{claim}{0}}

\theoremstyle{definition}
\newtheorem{definition}[theorem]{Definition}
\newtheorem{notation}[theorem]{Notation}

\newtheorem{example}[theorem]{Example}
\newtheorem{remark}[theorem]{Remark}

\makeatletter            
\let\c@equation\c@theorem  
\makeatother
\numberwithin{equation}{section}

\begin{document}

\begin{abstract}
We define contragredient Lie algebras in symmetric categories, generalizing the construction of Lie algebras of the form $\mathfrak{g}(A)$ for a Cartan matrix $A$ from the category of vector spaces to an arbitrary symmetric tensor category. The main complication resides in the fact that, in contrast to the classical case, a general symmetric tensor category can admit tori (playing the role of Cartan subalgebras) which are non-abelian and have a sophisticated representation theory.
Using this construction, we obtain and describe new examples of Lie algebras in the universal Verlinde category in characteristic $p\geq5$. We also show that some previously known examples  can be obtained with our construction.
\end{abstract}

\maketitle

\tableofcontents


\section{Introduction}

\subsection{Analogs of Deligne's theorem in positive characteristic}
By classical results of Deligne \cites{Del90,Del02}, pre-Tannakian categories (symmetric tensor categories wherein all objects have finite length) of moderate growth over a field $\bk$ of characteristic zero (always assumed algebraically closed in this work) are well understood. It turns out that any such category $\cC$ admits a (unique) symmetric tensor functor $\cC \to \sVec$ to the category of finite-dimensional super vector spaces, and can thus be reconstructed from group scheme theory in $\sVec$. If we regard pre-Tannakian categories of moderate growth as potential habitats to develop  commutative algebra or algebraic geometry, Deligne's results essentially say that, in characteristic zero, we should content ourselves with the categories of ordinary and super vector spaces, plus some (super) group equivariant structure.

It is also classically known \cites{GelKaz,GeoMat} that if we take $\bk$ of characteristic $p>0$, the theory of pre-Tannakian categories of moderate growth over $\bk$ becomes richer. However, the subject only started to enjoy a more systematic development after Ostrik's recent outbreak \cite{Ost2}, which sparked the deployment of intensive research efforts such as \cites{BEO,Cou-tan,Cou-comm,CEO,CEO-incomp,CEOP,Czenky,Etingof,Etingof-Gelaki-symmetric,EO-frob,Ven1,Ven2,Ven-GLVerp}. 

\medbreak
Assume from now on that $\bk$ has positive characteristic $p$. Even though still out of reach, a path towards an understanding of pre-Tannakian categories of moderate growth over $\bk$ emerges from some of the developments cited above. The main organizational tool behind said path is the notion of incompressible symmetric category. This is a pre-Tannakian category $\cD$ for which any symmetric tensor functor $\cD\to \cC$ is an embedding; informally, this means that $\cD$ can not be constructed from group scheme theory in a smaller symmetric tensor category. Thanks to \cite{CEO-incomp}, every pre-Tannakian category of moderate growth admits a symmetric functor to some incompressible category of moderate growth. Thus, to understand all pre-Tannakian categories of moderate growth over $\bk$, a possible route goes as follows:

\begin{enumerate}[leftmargin=*]
\item\label{item:incompres} Find all incompressible tensor categories $\cD$ of moderate growth.
\item \label{item:fiber} For each $\cD$ as above, characterize intrinsically which pre-Tannakian categories $\cC$ admit a symmetric tensor functor $\cC\to\cD$. 
\end{enumerate}

Conjecturally, incompressible tensor categories (which are expected to necessarily have moderate growth) are classified \cites{BEO,CEO-incomp}. However, for the smallest incompressible categories known (even for $\Vec$ and $\sVec$), a characterization of all $\cC$ as in part \eqref{item:fiber} is a hard problem.
To illustrate our point, consider  the universal Verlinde category $\Verp$, defined as the semisimplification of $\Rep_{\bk} (\bZ/p)$. This is the first term in a conjecturally exhaustive nested sequence of incompressible categories of moderate growth 
\[\Verp \subset \Ver_{p^2}\subset \cdots  \subset \Ver_{p^n}\subset \cdots \]
Now, given a pre-Tannakian category $\cC$, the problem of deciphering intrinsically whether it admits a symmetric functor $\cC\to\Verp$ was the motivating problem of the foundational work \cite{Ost2}. It was finally solved in \cite{CEO} using \cites{Ost2,Cou-tan,EO-frob}, among other novel tools and ideas; the problem turns out to be governed by the exactness of the \emph{Frobenius functor}, see Section \ref{subsec:Frob-exact} for some details. 

\subsection{Affine group schemes in pre-Tannakian categories}
By Tannaka-Krein duality \cites{Del90,CEO-incomp}, for each pre-Tannakian target $\cD$, any symmetric tensor category $\cC$ endowed with a symmetric functor $\cC\to\cD$ comes from group scheme theory in $\cD$. More precisely, any such $\cC$ can be uniquely obtained as $\Rep (G,\phi)$ where $G$ is an affine group scheme in $\cD$ equipped with a group map $\phi\colon\pi(\cD)\to G$ that induces the canonical action of the fundamental group $\pi(\cD)$ on $G$, and  $\Rep (G,\phi)$ denotes the representations of $G$ where $\pi(\cD)$ acts canonically after restriction along $\phi$.

A theory of affine group schemes (and their representations) in pre-Tannakian categories is starting to emerge. Precise results for group schemes in $\Verp$ have been established by \cites{Ven1,Ven2}, with a wider (yet detailed) study of the theory over more general pre-Tannakian categories by \cite{Cou-comm}. 
However, group schemes in general symmetric categories are still not well understood through examples, and in this work we establish a first step towards a steady supply. 

To the best of our knowledge, the only concrete examples constructed thus far are the general and special linear groups $\GL(X)$ and $\op{SL}(X)$ over any object $X$ of a symmetric category.\footnote{It has recently been brought to our attention by Arun Kannan that orthogonal, symplectic and type-Q groups can also be defined in $\Verp$.} Furthermore, when working on the category $\Verp$, the irreducible representations of these groups are understood; all this progress was achieved in \cites{Ven2,Ven-GLVerp}, relying on a Harish-Chandra equivalence. This is where Lie algebras enter the picture. 

The main result of \cite{Ven2} establishes an equivalence between affine group schemes of finite type in $\Verp$ and Harish-Chandra pairs in $\Verp$.
Here a group $G$ corresponds to the pair $(G_0, \fg)$, where $G_0\in \Vec\subseteq\Verp$ is an ordinary affine group underlying $G$ and $\fg\in \Verp$ is the Lie algebra of $G$. A possible interpretation of this result is that the main breach from group schemes in $\Vec$ to group schemes in $\Verp$ is governed by Lie algebras in $\Verp$. 

\subsection{Contragredient Lie superalgebras}
On a more classical note, the predominant tool and source of examples towards a (still incomplete) classification of simple Lie algebras in $\Vec$ or $\sVec$ comes from what is known as the contragredient construction \cite{Kac-book}. This process takes a square matrix $A$ and a parity vector $\bp$ and defines, by generators and relations, the \emph{contragredient} Lie superalgebra $\fg(A,\bp)$, which possess many desired properties, such as a lattice grading and a generalized root system.\footnote{In the literature, contragredient algebras are also called algebras with a Cartan matrix, or Kac-Moody algebras. The latter is sometimes reserved for matrices that produce locally nilpotent generators.}

The main contribution of this work is the construction of contragredient Lie algebras over arbitrary symmetric tensor categories. The first obvious obstacle when generalizing from $\Vec$ or $\sVec$ to a general symmetric category is the lack of elements. A second obstacle, that is not observed when upgrading the construction from $\Vec$ to $\sVec$, is the potential existence of new algebraic tori. To illustrate, recall that the construction of the contragredient superalgebra $\fg(A,\bp)$ starts with (and depends on) a \emph{realization} of the matrix $A$. Essentially, this is an abelian Lie algebra $\fh$ in $\Vec$ playing the role of a Cartan subalgebra. The adjoint action of $\fh$ on $\fg(A,\bp)$ is then completely described by the matrix $A$. In contrast, already in $\Verp$, one encounters algebraic tori with a more sophisticated representation theory, see \cite{Ven-GLVerp}. 

\subsection{A summary of our construction and main results}
Fix a symmetric tensor category $\cC$. Our construction takes as input a \emph{contragredient datum} in $\cC$ as introduced in \Cref{def:contragredient-data}. This is a triple $(\ttX, \rho, \ttd)$ where $\ttX\in \cC$ is a Lie algebra, $\rho \colon \ttX \ot \ttV \to \ttV$ is an $\ttX$-module in $\cC$, and $\ttd \colon \ttV \ot \ttV^* \to \ttX$ is a map of $\ttX$-modules.\footnote{In reminiscence to the classical case, $\ttX$ should be thought as a placeholder for the Cartan subalgebra, $\ttV$ is a substitute for the positive generators $e_i$'s and their parity, while $\rho$ (the action of $\ttX$ on $\ttV$) was formerly encoded by the matrix $A$; finally $\ttd$ describes the bracket between $\ttV$ and $\ttV^*$.} 
From this data we define first an auxiliary Lie algebra $\tttg(\rho,\ttd)$ as the minimal quotient of the semidirect product $\FLie(\ttV\oplus \ttV^* )\rtimes_\rho \ttX$ where the bracket between $\ttV$ and $\ttV^*$ coincides with $\ttd$, see \Cref{def:gtilde}. Our first main result is concerned with the structure of this algebra.

\begin{reptheorem}{thm:gtilde-verp}
Let $(\ttX, \rho, \ttd)$ as in \Cref{def:contragredient-data}. Assume that $\ttX$ has enough modules in the sense of \Cref{def:enough-modules}. Then, for the operadic Lie algebra $\tttg\coloneq\tttg(\rho, \ttd)$, the following hold:
\begin{enumerate}[leftmargin=*,label=\rm{(\alph*)}]
\item There is a $\bZ$-grading
\begin{align*}
\tttg&=\bigoplus_{k\in\bZ} \tttg_k, & \text{such that } \deg \ttV^*&=-1, & \deg \ttX&=0, & \deg \ttV&=1.
\end{align*}
\item If $\tttn_+$ and $\tttn_-$ denote the subalgebras generated by $\ttV$ and $\ttV^*$, respectively, then the natural maps $\ttV, \ttV^* \to \tttg$ induce Lie algebras isomorphisms $\FLie(\ttV)\cong \ttn_+$ and $\FLie(\ttV^*)\cong\ttn_-$.
\item  We have $\tttg=\tttn_-\oplus \ttX \oplus\tttn_+$ as objects of $\indcat{\cC}$. Moreover,
\begin{align*}
\tttn_- &= \bigoplus_{k<0} \tttg_k, & \ttX &= \tttg_0, & \tttn_+ &= \bigoplus_{k>0} \tttg_k. 
\end{align*} 
\item Among all $\bZ$-homogeneous ideals of $\tttg$ trivially intersecting $\ttX$, there exists a unique maximal one $\ttm$. Moreover, $\ttm=\ttm_+\oplus \ttm_-$, where $\ttm_{\pm}=\ttm\cap\tttn_{\pm}$.
\end{enumerate}
\end{reptheorem}

Now we arrive at our main definition: the contragredient Lie algebra $\ttg(\rho,\ttd)$ associated to $(\ttX,\rho,\ttd)$ is the quotient $\tttg(\rho,\ttd)/\ttm$, see \Cref{def:contragredient}. Under a mild assumption on the pair $(\rho,\ttd)$, we obtain in \Cref{thm:g-contragredient-verp} the first structural properties for $\ttg(\rho,\ttd)$. In particular, it inherits from $\tttg(\rho,\ttd)$ a $\bZ$-grading and a triangular decomposition $\ttg(\rho,\ttd)=\ttn_-\oplus \ttX \oplus\ttn_+$, where $\ttn_+$ and $\ttn_-$ are the subalgebras generated by $\ttV$ and $\ttV^*$. 

To obtain further properties, we restrict to a special source of contragredient data, that we call symmetrizable and introduce in \Cref{def:symmetrizable-data}. The main difference here is the assumption that $\ttX$ comes with an invariant non-degenerate form $\ttK\colon\ttX\ot\ttX\to \one$. Then, for any module $\rho \colon \ttX\ot\ttV\to \ttV$, one can produce a map $\ttd_{\rho,\ttK} \colon \ttV \ot \ttV^* \to \ttX$ of $\ttX$-modules. 

Our first main result show that, from any $\ttX$-module decomposition $\ttV=\oplus_{1\leq i \leq r} \ttV_i$, one obtains for $\ttg(\rho,\ttd_{\rho,\ttK})$ a grading by the lattice $Q\coloneq \oplus _{1\leq i \leq r} \alpha_i \bZ$.

\begin{reptheorem}{thm:contragredient-symmetrizable-grading}
The Lie algebra $\ttg\coloneq\ttg(\rho,\ttd_{\rho,\ttK})$ admits a $Q$-grading such that
\begin{align*}
\ttg&=\bigoplus_{\alpha \in Q} \ttg_\alpha, &\deg \ttX&=0, &\deg \ttV_i&=\alpha_i, &\deg \ttV_i^*&=-\alpha_i, & 1&\leq i\leq r.
\end{align*}
Moreover, under this grading we have
\begin{align*}
\ttg_0&=\ttX, &\ttn_+&=\bigoplus_{\alpha \in Q^+, \alpha \neq 0} \ttg_\alpha,  &\ttn_-&=\bigoplus_{\alpha \in Q^-, \alpha \neq 0} \ttg_\alpha.
\end{align*}
\end{reptheorem}

In our second main theorem we show that any decomposition for $\ttV$ as above yields a non-degenerate invariant form on $\ttg(\rho,\ttd_{\rho,\ttK})$.
\begin{reptheorem}{thm:contragredient-bilinear}
The algebra $\ttg$ admits  an invariant symmetric form $\ttB\colon\ttg\ot\ttg\to\one$ such that
\begin{enumerate}[leftmargin=*,label=\rm{(\alph*)}]
\item $\restr{\ttB}{\ttX\ot\ttX}=\ttK$,
\item $\restr{\ttB}{\ttg_\alpha\ot\ttg_\beta}=0$ if $\alpha+\beta\ne0$, and
\item $\restr{\ttB}{\ttg_\alpha\ot\ttg_{-\alpha}}$ is non-degenerate for all $\alpha \in Q$.
\end{enumerate}
\end{reptheorem}

We end the paper with examples of contragredient Lie algebras, mostly in $\Verp$, coming from three sources. First, we recover general and special Lie algebras $\fgl(V)$, $\fsl(V)$ over an object $V\in\Ver_p$ as contragredient Lie algebras. Then we go back to our previous work \cite{APS-ssLieAlg}, where we studied Lie algebras in $\Ver_p$ obtained by semisimplification (following \cite{Kan}); we show in \Cref{subsection:semisimplification-examples} that these examples are also contragredient. 

Finally, in \Cref{subsec:examples-over-gl(L2)} we describe all contragredient data over the torus $\fgl(\ttL_2)\in\Lie(\Verp)$ and its subalgebras $\one$ and $\fsl(\ttL_2)$, assuming that the positive part is a simple object of  $\Ver_p$. For the torus $\ttX=\one$ we are able to compute the components in degree $\pm 2$ of the corresponding contragredient Lie algebras in \Cref{lem:contragredient-Li-one}, 
and we show that in $\cha 5$ the associated Lie algebras do not have finite length. These examples can be thought of as generalizations of  the rank-one Lie superalgebras with non-trivial action of $\one$, namely $\fsl_2$ and $\osp(2|1)$, see \Cref{rem:sl2-osp21-generalization}, even though they are infinite.

\bigbreak 
\noindent\textbf{Acknowledgments.} 
We are deeply grateful to Pavel Etingof and Arun Kannan for many helpful discussions. 

I. A. was partially supported by Conicet, SeCyT (UNC) and MinCyT. 
The research of J.P. was partially supported by NSF grant DMS-2146392 and by Simons Foundation Award 889000 as part of the Simons Collaboration on Global Categorical Symmetries. J.P. would like to thank the hospitality
and excellent working conditions at the Department of Mathematics at the University of Hamburg, where J. P. has carried out part of this research as an Experienced Fellow of the Alexander von Humboldt Foundation.
The work of G.S was partially supported by an NSF Simons travel grant. 

\section{Preliminaries}\label{sec:preliminaries}
\noindent\textbf{Conventions.}
We denote by $\bN$ the set of positive integers and $\bN_0=\bN\cup \{0\}$. 

Fix an algebraically closed field $\bk$ of characteristic $p\geq 0$. 
We denote by $\Vec$ and $\sVec$ the symmetric fusion categories of finite dimensional vector spaces and super vector spaces over $\bk$, respectively.

In this work, all categories are assumed essentially small. When working on an additive category $\cC$, we denote by $\oplus \colon \cC \times \cC \to \cC$ the biproduct bifunctor. Consider arrows $f_1\colon X_1 \to Y_1$ and $f_2\colon X_2 \to Y_2$ in $\cC$, then 
\begin{itemize}[leftmargin=*]
\item we reserve $f_1\oplus f_2\colon X_1\oplus X_2 \to Y_1 \oplus Y_2$ to denote the biproduct arrow;
\item if $X_1=X_2=X$ we denote the product arrow by $(f_1, f_2)\colon X \to Y_1 \oplus Y_2$;
\item if $Y_1=Y_2=Y$, we use $f_1 \sqcup f_2 \colon X_1\oplus X_2 \to Y$ to denote the coproduct arrow;
\item if $X_1=X_2=X$ and  $Y_1=Y_2=Y$, we use $f_1 \pm f_2$ to denote the arrows obtained from the group structure of $\Hom_\cC (X,Y)$.
\end{itemize}

\subsection{Symmetric tensor categories and Deligne's Theorem}
A \emph{(symmetric) tensor category} is an abelian $\bk$-linear category  which admits a rigid (symmetric) monoidal structure in a compatible way: the endomorphism space of the unit object is one-dimensional and the tensor product functor is $\bk$-bilinear on morphisms (thus exact by rigidity \cite{EGNO}*{Proposition 4.2.1}). Notice that we adopt the conventions of \cite{CEO} which differ from those of \cites{Del90, Del02, EGNO}.

Given an object $X$ in an abelian category $\cC$, the \emph{length} $\ell(X)$ of $X$ is the supremum of the set of integers $m$ for which $X$ admits an strictly increasing family of subobjects $0\subsetneq X_1 \subsetneq\dots\subsetneq X_m=X$ in $\cC$. 

A \emph{pre-Tannakian} category is a symmetric tensor category wherein all objects have finite length. Under this assumption all Hom spaces turn out to be finite dimensional. 

\smallbreak
Throughout this work, tensor functors are assumed (or required) to be exact. More precisely, a \emph{(symmetric) tensor functor}  is an exact $\bk$-linear (symmetric) monoidal functor between (symmetric) tensor categories. Notice that, for a monoidal $\bk$-linear functor between tensor categories, exactness is equivalent to faithfulness thanks to \cite{Del90}*{Corollaire 2.10 (i)} and \cite{CEOP}*{Theorem 2.4.1}.

We sometimes say that a pre-Tannakian category $\cC$ \emph{fibers over} some other pre-Tannakian category $\cD$ if there is a symmetric tensor functor $\cC\to\cD$. This terminology fits well with the classical one (e.g. that from \cite{Saa}) thanks to \cite{CEO-incomp}*{\S 4}, where Tannaka-Krein duality over a pre-Tannakian target is established. Roughly speaking, this means that both $\cC$ and the fiber functor can be uniquely recovered from group-scheme theory in $\cD$.

\subsubsection{Semisimplification of symmetric tensor categories}
Consider a $\bk$-linear monoidal category $\cC$. A tensor ideal $\cI$ in $\cC$ is a collection of subspaces $\cI=\{\cI(X,Y)\subseteq \Hom_\cC(X,Y)\}$, where $X, Y$ are objects in $\cC$, which is compatible with compositions and tensor products. This enables us to form the \emph{quotient category} of $\cC$ by $\cI$, which has the same objects as $\cC$, but the new spaces of morphisms are the quotients $ \Hom_\cC(X,Y)/ \cI(X,Y)$. This quotient procedure is compatible with the $\bk$-linear monoidal structure (among many other properties and structures) of the original category.

An special instance of this constructions starts with an spherical category $\cC$ \cite{EGNO}*{\S 4.7} and considers the tensor ideal of \emph{negligible} objects. In this case the quotient, known as the \emph{semisimplification} of $\cC$, turns out to be semisimple, where the simples are the indecomposables in $\cC$ with non-zero categorical dimension. See \cite{EO-ss} and references therein.

\subsubsection{Ind-completions of symmetric tensor categories}
For a symmetric tensor category $\cC$, we denote by $\indcat{\cC}$ its \emph{ind-completion}, which is defined to be the closure of $\cC$ under filtered colimits. 
Since the tensor product of $\cC$ is exact, it preserves colimits and induces a tensor product on $\indcat{\cC}$. By naturality, $\indcat{\cC}$ inherits the braiding from $\cC$. Hence $\indcat{\cC}$ is a $\bk$-linear abelian category with an exact and symmetric tensor product, and there is an exact symmetric embedding $\cC \hookrightarrow \indcat{\cC}$ that is universal in some sense, see \cite{Kash-Shap}.
We shall refer to the objects of $\indcat{\cC}$ as ind-objects of $\cC$. The genuine objects of $\cC$ are then recovered as ind-objects of finite length. If its clear from the context that $X$ and $Y$ are ind-objects, we write $\Hom_{\cC} (X,Y)$ instead of $\Hom_{\indcat{\cC}} (X,Y)$. 
As an example, if $\cC$ is moreover fusion, that is, finite and semisimple, then the objects of $\indcat{\cC}$ are (possibly infinite) direct sums of simple objects in $\cC$.

\subsubsection{Vanishing of alternating powers} Given a positive integer $n$ and an object $X$ in a symmetric tensor category, the $n$-th alternating power $\op{A}^n(X)$ is defined as the image of $\sum_{\sigma\in S_n}(-1)^{\op{sign}(\sigma)} \sigma \colon X^{\ot n} \to X^{\ot n}$. Here we follow the terminology of \cite{Del90}, and point out that there are other possible definitions of alternating powers of $X$ in $\cC$, which are equivalent in $\Vec$ but differ already in $\sVec$ if $p>2$, see \cite{EHO}*{\S 2.1}. 
A first version of Deligne's Theorem states that, in characteristic $0$, vanishing of alternating powers detects Tannakian categories (those of the form $\Rep(G)$ for an affine group scheme $G$) among all pre-Tannakian categories.

\begin{theorem}{\cites{Del90, Del02}} \label{thm:Deligne-Tannakian} Assume $\op{char} \bk=0$ and let $\cC$ be a pre-Tannakian category over $\bk$. Then $\cC$ admits a symmetric tensor functor $\cC\to\Vec$ if and only if for each $X$ in $\cC$ there is some $n\ge 1$ such that $\op{A}^n(X)=0$.
\end{theorem}

\subsubsection{Moderate growth} 
Consider a tensor category $\cC$ where all objects have finite length. Then $\cC$ has \emph{moderate growth} if for each object $X$ there is some real number $a_X$ so that $\ell(X^{\ot n})\leq a_X^n$ for all $n\geq 1$. 
Famously, if $\op{char} \bk=0$, this growth condition detects categories of finite-dimensional representations of affine proalgebraic supergroups among all pre-Tannakian categories. More precisely:

\begin{theorem}{\cite{Del02}}  \label{thm:Deligne-super-Tannakian} Assume $\op{char} \bk=0$ and let $\cC$ be a pre-Tannakian category over $\bk$. Then $\cC$ admits a symmetric tensor functor $\cC\to\sVec$ if and only if it has moderate growth.
\end{theorem}

It is well known (see e.g.~\cites{GelKaz, GeoMat, Ven1}) that this characterization fails to hold over fields of positive characteristic, where the problem of obtaining analogue statements remained somewhat unexplored until Ostrik's recent breakthrough \cite{Ost2}. From the moment of its appearance, \cite{Ost2} motivated and inspired many other remarkable works concerning symmetric tensor categories of moderate growth over fields of positive characteristic. In the next Sections, we discuss this still unveiling story.

\subsection{The Verlinde category and Ostrik's Theorem}\label{subsec:verp}

The first Frobenius kernel $\balp_p$ of the additive group scheme over a field of characteristic $p>0$ is represented by the Hopf algebra $\bk[t]/(t^p)$ where $t$ is primitive. By self-duality, indecomposables in $\Rep(\balp_p)$ are parametrized, up to isomorphism, by Jordan blocks $L_i$ of size $i$ with $1\leq i \leq p$. 

By definition, the Verlinde category $\Verp$ is the semisimplification of the symmetric tensor category $\Rep(\balp_p)$, see \cites{GelKaz, GeoMat, Ost2}. Hence, $\Ver_p$ is a symmetric fusion category equipped with a $\bk$-linear symmetric monoidal functor $\sfS \colon \Rep(\balp_p) \to \Verp$. Simple objects of $\Verp$ are given, up to isomorphism, by the images $\ttL_i$ of the indecomposables $L_i$ with $1\leq i \leq p-1$ of $\Rep (\balp_p)$, and the monoidal unit is $\one=\ttL_1$. Non-faithfulness of the semisimplification functor becomes evident by looking at the indecomposable $L_p$, which is zero-quantum-dimensional and thus annihilated. In $\Verp$ all simple objects are self-dual, and the fusion rules are given by the truncated Clebsch-Gordan formula
\begin{align}\label{eq:verp-fusionrules}
\ttL_i \ot \ttL_j &= \bigoplus_{k=1}^{\min\{i, j, p-i, p-j\}} \ttL_{\vert i-j \vert + 2k-1},& 1&\leq i, j \leq p-1.
\end{align}
Notice that $\Ver_2=\Vec$ and $\Ver_3=\sVec$. Furthermore, for any $p\geq 3$, the abelian subcategory of $\Verp$ generated by $\ttL_1$ and $\ttL_{p-1}$ is a fusion category that turns out to be equivalent, as a symmetric tensor category, to $\sVec$, see \cites{Ost2,Kan}. 

On the other hand, when $p\geq 5$, the fusion rules \eqref{eq:verp-fusionrules} imply that the abelian subcategory of $\Verp$ spanned by simples $\ttL_i$ with $i$ odd is again a fusion subcategory, denoted by $\Verp^+$. By \cite{Ost2}, the subcategory $\Verp^+$ is tensor generated  by $\ttL_3$, and we have $\Verp \cong \sVec\boxtimes \Verp^+$ as symmetric fusion categories. 

By a dimension argument, as long as $p\geq 5$, neither $\Verp$ nor $\Verp^+$ fiber over $\sVec$. For instance, when $p=5$, the object $\ttL_3$ satisfies $\ttL_3^{\ot 2} = \one \oplus \ttL_3$. The next result is in some sense the starting point of the modern theory of symmetric tensor categories of moderate growth.

\begin{theorem}{\cite{Ost2}} Assume $\op{char} \bk=p>0$. Then any symmetric fusion category $\cC$ over $\bk$ admits a symmetric tensor functor $\cC\to\Verp$.
\end{theorem}

We compute for later use the self braiding of the simple objects $\ttL_i$ in $\Verp$. 
\begin{lemma}\label{lem:symmetry-Li-ot-Li}
Let $1\le i<p$. With respect to the fixed decomposition \eqref{eq:verp-fusionrules}, the symmetry $c:=c_{\ttL_i\ot\ttL_i}$ is given by the diagonal matrix $\mathrm{diag}\left((-1)^{i-k}\right)_{1\le k\le \min\{i, p-i\}}$. In other words,
\begin{align*}
\restr{c}{\ttL_{2k-1}}&=(-1)^{i-k}\id_{\ttL_{2k-1}}, & &1\le k\le \min\{i, p-i\}.
\end{align*}
\end{lemma}
\begin{proof}
Take first $i\le \frac{p-1}{2}$. Fix a cyclic basis $(v_j)_{1\leq j \leq i}$ of the indecomposable $\balp_p$-module $L_i$, which means that $t\cdot v_j=v_{j+1}$ if $j<i$ and $t\cdot v_i=0$. On the other hand, consider $M\coloneq \Bbbk[x,y]/(x^i,y^i)$; since $ \overline{(x+y)^p}=\overline{x^p+y^p}=0$, the rule $t\cdot f=(x+y)f$ defines an $\balp_p$-module structure on $M$. Furthermore, the map $L_i\ot L_i\to M$, $v_j\ot v_k\mapsto x^{j-1}y^{k-1}$, is an isomorphism of $\balp_p$-modules. We decompose $M$ to compute the self braiding of $L_i$.

Let $M_S$ and $M_A$ denote the subspaces of symmetric and antisymmetric polynomials modulo $(x^i,y^i)$, respectively. Thus $M=M_S\oplus M_A$ as $\balp_p$-modules. We decompose further
\begin{align*}
M_S&=L_{j_1}\oplus\cdots\oplus L_{j_r}, & M_A&=L_{j_{r+1}}\oplus\cdots\oplus L_{j_i}, &j_1>\cdots >j_r, & \, j_{r+1}>\cdots>j_i,
\end{align*}
as sums of indecomposable submodules, with each $L_{j}$ generated by $v_{j}$ as $\balp_p$-module, so $\{j_1,\dots,j_i \}=\{2k-1\colon 1\leq k\leq i\}$. Polynomial degree in $\Bbbk[x,y]$ induces an $\bN_0$-graded algebra structure on $M$, say
$M=\oplus_{0\le k\le 2i-2} M_k$, which satisfies  $t\cdot M_k\subseteq M_{k+1}$. Let $\pi_k\colon M\twoheadrightarrow M_k$ be the canonical projection. As $M_A\subseteq \oplus_{1\le k\le 2i-2} M_k$, we have that $t^{2i-2}\cdot M_A=0$. Thus $j_1=2i-1$, with $\pi_0(v_{j_1})\ne 0$. As $M_S\cap M_1$ and $M_S\cap M_{2i-2}$ are one-dimensional, since $0\ne t\cdot v_{j_1}\in M_{2t-2}$, we have that $L_{j_2}\oplus\cdots\oplus L_{j_k} \subseteq \oplus_{2\le k\le 2i-3} M_k$, so $t^{2i-4}$ annihilates these summands, which implies that $j_{r+1}=2i-3$, and $\pi_1(v_{j_{r+1}})\ne 0$. Similarly, $L_{j_{r+2}}\oplus\cdots\oplus L_{j_i} \subseteq \oplus_{3\le k\le 2i-4} M_k$, so $j_2=2i-5$. A similar argument shows recursively that $r=\tfrac{i}{2}$ if $i$ is even and $r=\tfrac{i+1}{2}$ if $i$ is odd, with 
\begin{align*}
j_1& =2i-1, j_2=2i-5, \dots, j_r=\begin{cases} 3 & i\text{ even}, \\ 1 & i\text{ odd},\end{cases}  & 
j_{r+1}&=2i-3, \dots, j_i =\begin{cases} 1 & i\text{ even}, \\ 3 & i\text{ odd},\end{cases}
\end{align*}
and the statement follows by applying the semisimplification functor $\sfS$.

\smallbreak
Consider now $i\ge \frac{p-1}{2}$. The isomorphism $\ttL_i\simeq \ttL_{p-1}\otimes \ttL_{p-i}$ induces an isomorphism
\begin{align*}
\ttL_i \ot \ttL_i\simeq \ttL_{p-1}\otimes \ttL_{p-i} \ot \ttL_{p-1}\otimes \ttL_{p-i}
\end{align*}
Using the naturality of the symmetry, $\restr{c}{\ttL_{2k-1}}$ is given by 
the restriction of $c_{\ttL_{p-1}\otimes \ttL_{p-i},\ttL_{p-1}\otimes \ttL_{p-i}}$ composed with $c_{\ttL_{p-1},\ttL_{p-1}}=-\id_{L_1}$, so the statement follows from the first case.
\end{proof}

\subsection{Frobenius exact categories and Coulembier-Etingof-Ostrik's Theorem}\label{subsec:Frob-exact} 
One of the main contributions of \cite{Ost2} is the construction of the \emph{Frobenius functor} $\cC\to\cC^{(1)}\boxtimes \Verp$, which comes as a generalization of the pullback $\Rep(G)\to \Rep(G)^{(1)}$ along the Frobenius map $\lambda\mapsto\lambda^p$ for an affine group scheme $G$. A priori, the Frobenius functor in its original form \cite{Ost2} was only available for \emph{semisimple} pre-Tannakian categories. The task of making it available in a more general setting immediately garnered relevance, and was independently achieved in \cites{Cou-tan, EO-frob}, producing slightly different versions of a $\bk$-linear functor $\Fr \colon \cC\to\cC^{(1)}\boxtimes \Verp$ for any pre-Tannakian category $\cC$. These approaches were unified in \cite{CEO}. It turns out that \emph{Frobenius exactness} (i.e. exactness of $\Fr$) is the unique novel obstruction for a pre-Tannakian category to fiber over $\Verp$. 

\begin{theorem}{\cite{CEO}}
Assume $\op{char} \bk=p>0$. Then, for a pre-Tannakian category $\cC$, the following are equivalent:
\begin{enumerate}[leftmargin=*, label=(\arabic*)]
\item $\cC$ admits a symmetric tensor functor $\cC\to\Verp$;
\item $\cC$ is Frobenius exact and of moderate growth;
\item $\cC$ is Frobenius exact and for each object $X$ in $\cC$ there is some $n\geq1$ so that $\op{A}^n(X)=0$.
\end{enumerate}
\end{theorem}

A curious byproduct of this result is that Deligne's conditions for Tannakian and super-Tannakian categories in characteristic $0$  (see  \Cref{thm:Deligne-Tannakian,thm:Deligne-super-Tannakian}) are actually equivalent for Frobenius exact categories in positive characteristic.


\subsection{Operadic, PBW, and genuine Lie algebras}\label{subsec:operadic} 
Following \cite{Etingof} an \emph{(ind-)operadic Lie algebra} over a (strict) symmetric tensor category $\cC$ is an (ind-)object $\fg$ in $\cC$ equipped with a morphism $b\colon \fg\otimes\fg \to \fg$ that satisfies the usual anti-symmetric and Jacobi conditions: 
\begin{align}\label{eq:operadic}
b\circ(\id_{\fg\ot \fg} + c_{\fg, \fg}) = 0,  \qquad b\circ (b\ot \id_{\fg})\circ(\id_{\fg^{\ot 3}} + (123)_{\fg^{\ot 3}} +(132)_{\fg^{\ot 3}}) = 0, 
\end{align}
where $(123)_{\fg^{\ot 3}}, (132)_{\fg^{\ot 3}}\colon\fg^{\ot 3}\to \fg^{\ot 3}$ are obtained from the action of the symmetric group. We use $\OLie(\cC)$ to denote the category of operadic Lie algebras in $\cC$. When no confusion is possible, we use the notation $\OLie(\cC)$ also for $\OLie(\indcat{\cC})$.  Note that in previous bibliography, such as \cite{Rumynin}, operadic Lie algebras are just called Lie algebras; however, a distinction needs to be made,  especially when working in positive characteristic.

Recall from \cite{Etingof} that the \emph{universal enveloping algebra} $\op{U}(\fg)$ is the quotient of the tensor algebra $\op{T}(\fg)$ by the two-sided ideal generated by the image of the map
\begin{align}\label{eq:defn-enveloping-algebra}
a\coloneq(-b,\id_{\fg\ot\fg}- c_{\fg,\fg}): \fg\ot\fg\to \fg\oplus\fg\ot\fg\subseteq \op{T}(\fg).
\end{align}
Similarly, the symmetric algebra $\op{S}(\fg)$ is the quotient of $\op{T}(\fg)$ by the image of $\id_{\fg\ot\fg}- c_{\fg,\fg}$.
As usual, there is a natural map $\eta \colon \op{S}(\fg) \to \gr \op{U}(\fg)$ of ind-algebras in $\cC$. However, it can fail to be an isomorphism, which is the main reason behind the adjective \emph{operadic} in the definition above. 
This somewhat undesired behavior was already exhibited by $\cC=\Vec$ when $p=2$: for $\eta$ to be an isomorphism one needs to impose the additional condition $b(x\ot x)=0$ for all $x$. Similarly, in $\cC=\sVec$ over $p=3$, one needs to assume $b(b(x\ot x)\ot x)=0$ for all odd $x$. 

\begin{definition}{\cite{Etingof}}\label{def:PBW}
We say that an operadic Lie algebra $\fg$ is \emph{PBW} if the natural map $\eta \colon \op{S}(\fg) \to \gr \op{U}(\fg)$ is an isomorphism.
\end{definition}

\subsubsection{Free operadic Lie algebras} 
By \cite{Etingof}*{Section 4.2}, one can build a free operadic Lie algebra over any object in a symmetric tensor category $\cC$. More precisely, there is a functor $\FOLie$ from $\cC$ to $\OLie(\indcat{\cC}$), the category of ind-operadic Lie algebras in $\cC$, which is left adjoint to the forgetful functor. That is, we have a natural isomorphism 
\begin{align*}
\Hom_{\cC}(V, L)&\simeq\Hom_{\OLie(\indcat{\cC})}(\FOLie(V), L),& V&\in \cC, \ L\in \OLie(\cC). 
\end{align*}

This notion plays a fundamental role in the theory of Lie algebras in positive characteristic. For instance, it is used in \cite{Etingof}*{Proposition 6.10} to construct in $\cC=\Ver_p$ (for any characteristic $p\geq5$), an operadic Lie algebra (of finite length) that is not PBW.
More importantly, the notion of genuine Lie algebra in a symmetric tensor category $\cC$ introduced in \cite{Etingof} relies heavily on this construction, as we briefly summarize next.

Given any object $V$ in $\cC$, the inclusion of $V$ into $\op{T}(V)$ --which becomes an operadic Lie algebra via the commutator-- gives rise to a map $\phi^V \in \Hom_{\OLie(\cC)}( \FOLie(V), \op{T}(V))$.
As in \cite{Etingof}*{\S 4.2} we write $\phi^V=\oplus_{n\ge 1}\phi^V_n$, where $\phi^V_n:\FOLie_n(V)\to V^{\otimes n}$, and denote 
\begin{align}\label{eq:defn-E(V)}
E_n(V) &\coloneqq \ker \phi^V_n, \quad n\ge 1, &
E(V) & \coloneqq \ker \phi^V= \oplus_{n\ge 1} E_n(V).
\end{align}
On the other hand, for an operadic Lie algebra $L$ in $\cC$ let \[\beta^L\in \Hom_{\OLie(\cC)}( \FOLie(L), L)\] denote the map corresponding to the identity map $\id_L$ in $\cC$ under the adjunction. 

\begin{definition}{\cite{Etingof}}\label{def:Lie-algebra}
An operadic Lie algebra $L$ in a symmetric tensor category $\cC$ is a \emph{Lie algebra} if $\beta^L$ annihilates $E(L)$.
We denote by $\Lie(\cC)$  the category of Lie algebras in $\cC$.
\end{definition}

Verifying if a given operadic Lie algebra in a tensor category is either PBW or a genuine Lie algebra is rather involved. Fortunately, in \cite{Etingof} it is shown that the former implies the latter, and more generally, that any associative algebra and its subobjects closed under the commutator are Lie algebras. Preservation of these properties under quotients and subobjects is also studied.

\smallbreak
The situation is better understood for $\cC=\Ver_p$, and hence for any Frobenius exact pre-Tannakian category of moderate growth. Indeed, \cite{Etingof} introduces for any $p\geq5$ the \emph{$p$-Jacobi identity}, a rather explicit relation of degree $p$ that generalizes the conditions $b(x\ot x)=0$ and $b(b(x\ot x)\ot x)=0$ needed for characteristics $2$ and $3$, respectively. Namely,

\begin{definition}{\cite{Etingof}}\label{def:p-Jacobi} We say
an operadic Lie algebra $L$ in a symmetric tensor category satisfies the $p$-Jacobi identity if $\restr{\beta^L}{E_p(L)}=0$.
\end{definition}

\begin{theorem}{\cite{Etingof}*{Theorem 6.6, Corollary 6.7}, \cite{CEO}*{Theorem 1.1}} \label{thm:Lie-alg-equivalences}
Let $\cC$ be a Frobenius exact pre-Tannakian category of moderate growth. An operadic Lie algebra in $\cC$ is PBW if and only if it is a Lie algebra, if and only if it satisfies the $p$-Jacobi identity. 
\end{theorem}

\subsubsection{Free Lie algebras} 
Using free operadic Lie algebras, one can construct the functor
\begin{align*}
\FLie &\colon \cC \to \Lie (\cC),& \FLie(V)&\coloneq\FOLie(V)/\ker\left(\phi^V\right), 
\end{align*}
which is left adjoint to the forgetful functor by \cite{Etingof}*{Corollary 4.9}. Thus for any object $V$ we have the free Lie algebra $\FLie(V)$, which is an ind-Lie subalgebra of $\op{T}(V)$.

\begin{proposition}\label{prop:FLie-subalg-TV}
The free Lie algebra $\FLie(V)$ is naturally isomorphic to the Lie subalgebra $\ff$ of $\op{T}(V)$ generated by $V$. The inclusion $\ff\hookrightarrow  \op{T}(V)$ extends to an isomorphism of algebras $\op{U}(\ff) \isomorph \op{T}(V)$. In particular, $\FLie(V)$ is PBW.
\end{proposition}

\begin{proof}
Note first that $\ff$ is PBW by \cite{Etingof}*{Proposition 4.7}. Let $\iota_{\ff} \colon V \hookrightarrow \ff$ denote the inclusion. We show that for any $L\in\Lie(\cC)$ and $\psi \in \Hom_{\cC}(V,L)$, there is a unique $\ov{\psi}  \in \Hom_{\Lie(\cC)}(\ff, L)$ such that $\ov{\psi}\iota_{\ff}=\psi$. 
By the universal property of $\op{T}(V)$ there is a unique algebra map $\ov{\psi}\colon \op{T}(V) \to \op{U}(L)$ such that the following diagram commutes
\[\begin{tikzcd}
V\arrow[hook]{rr}{\iota_V} \arrow{d}{\psi}  && \op{T}(V)  \arrow{d}{\overline{\psi}}\\
L \arrow[hook]{rr}{\iota_L}&& \op{U}(L).
\end{tikzcd}\]
Thus the image of $\restr{\ov \psi}{V}$ is contained in $L$. Since $\ov{\psi}$ is an algebra map and $\ff$ is generated by $V$, this implies that the image of $\restr{\ov \psi}{\ff}$ is also contained in $L$. Hence we have a Lie algebra map $\restr{\ov \psi}{\ff} \colon \ff \to L$ such that $\ov{\psi}\iota_{\ff}=\psi$, as desired. Uniqueness is clear, so the first statement follows. 

For the second statement, given an associative algebra $A$ and $\phi\in\Hom_{\Lie(\cC)}(\ff, A)$, the restriction $\restr{\phi}{V}$ can be extended uniquely to a map of associative algebras $\ov{\phi}\colon\op{T}(V) \to A$. Then one shows as before that $\restr{\ov{\phi}}{\ff} = \phi$, and the claim follows.
\end{proof}

\begin{remark}\label{rem:subalg-generated}
Let $\fg$ be a Lie algebra in $\cC$, and $V$ a subobject of $\fg$. The subalgebra generated by $V$ is $\langle V \rangle=\sum_{n\ge 1} \ima \varphi_n$, where $\varphi_1:V\hookrightarrow \fg$ is just the inclusion and $\varphi_{n+1}\coloneq \restr{b}{\ima\varphi_n\otimes V}$.

This description applies to the pair $(V, \fg)=(V, \op{T}(V))$, and shows that $\ff\simeq\FLie(V)$ is a graded Lie algebra $\ff=\oplus_{n\ge 1} \ff_n$, as $\ima \varphi_n\subseteq V^{\ot n}$. Here, $\ff_1=V$, $\ff_2=\ima(\id_{V\ot V}-c_{V,V})$.
\end{remark}

\begin{remark}\label{rem:FOLie-Li}
Consider $\cC=\Ver_p$ and a simple object $V=\ttL_i$. We have a decomposition $\ttL_i\ot \ttL_i=\bigoplus_{k=1}^{\min\{i, p-i\}} \ttL_{2k-1}$, and by \Cref{lem:symmetry-Li-ot-Li}, we get
\begin{align*}
\FLie(\ttL_i)_2 &= \bigoplus_{k \ \text{antisymmetric}} \ttL_{2k-1}^{(2)},
 \end{align*}
where the sum runs over indexes $1\leq k\leq \min\{i, p-i\}$ such that the self-braiding $c_{\ttL_i, \ttL_i}$ acts as $-1$ on  $\ttL_{2k-1}$.
In particular, $\FLie(\ttL_i)_2$ contains a copy of $\one=\ttL_{1}$ if and only if $i$ is even.
\end{remark}

\subsection{Contragredient Lie (super)algebras}\label{subsec:semisimplification}
Here we recall the basics of contragredient Lie superalgebras; we refer to  \cites{BGL, Kac-book} for details.
The starting input data is:
\begin{enumerate}[leftmargin=*,label=\rm{(D\arabic*)}]
\item\label{item:contragredient-1} a square matrix $A=(a_{ij})_{1\leq i,j\leq r}$ with entries in $\bk$;
\item \label{item:contragredient-2} a \emph{parity vector} $\bp=(p_i)\in (\bZ/2\bZ)^r$.
\end{enumerate}
Up to isomorphism, there is a unique \emph{realization} of $A$, which consists on
\begin{enumerate}[leftmargin=*, resume, label=\rm{(D\arabic*)}]
\item \label{item:contragredient-3} a $\Bbbk$-vector space $\fh$ of dimension $2r-\rk A$;
\item \label{item:contragredient-4} linearly independent subsets $(\xi_i)_{1\leq i \leq r} \subset \fh^*$ and $(h_i)_{1\leq i \leq r} \subset \fh$,
such that $\xi_j(h_i)=a_{ij}$ for all $1\leq i,j\leq r$.
\end{enumerate}

From this data, we introduce first the auxiliary Lie superalgebra $\tfg\coloneq\tfg(A,\bp)$  generated by elements $e_i$, $f_i$ for $1\leq i, j \leq r$, and $\fh$,
with parity given by
\begin{align*}
|e_i|&=|f_i|= p_i,\quad  1\leq i\leq r, & |h|&=0, \mbox{ for all }h\in\fh,
\end{align*}
modulo the relations 
\begin{align}
\label{eq:relaciones gtilde}
[h,h']&=0, &[h,e_i] &= \xi_i(h)e_i, & [h,f_i] &= -\xi_i(h)f_i, & [e_i,f_j]&=\delta_{ij}h_i
\end{align}
for all $1\leq i, j \leq r$, and $h, h'\in \fh$.

The Lie superalgebra $\tfg$ admits a $\bZ$-grading with $\deg f_i=-1$, $\deg \fh=0$ and $\deg e_i =1$. Moreover, $\tfg$ contains a unique maximal graded ideal $\fm$ among those that meet $\fh$ trivially. 
\smallbreak
The \emph{contragredient Lie superalgebra} associated to the pair $(A, \bp)$ is the Lie superalgebra quotient $\fg(A,\bp)\coloneq\tfg(A,\bp) / \fm$. 
This superalgebra canonically contains the $e_i$'s, $f_i$'s, and $\fh$. It also inherits a $\bZ$-grading $\fg = \mathop{\oplus}\limits_{k\in \bZ}\fg_k$ and a triangular decomposition
$\fg=\fn_-\oplus\fh\oplus\fn_+$ where $\fn_-$ and $\fn_+$ are the Lie subsuperalgebras generated by the $e_i$'s and the $f_i$'s, respectively.

The $\bZ$-grading above can be refined to a $\bZ^r$-grading. Indeed, $\tfg$ is $\bZ^r$-graded with
\begin{align*}
\deg f_i&=-\alpha_i, & \deg h&=0, & \deg e_i&=\alpha_i,& \mbox{ for }1\leq i\leq r& \mbox{ and }  h\in \fh,
\end{align*}
where $(\alpha_i)$ denotes the canonical basis of $\bZ^r$.  The defining ideal $\fm$ happens to be $\bZ^r$-homogeneous, so $\fg=  \fh \oplus \mathop{\oplus}\limits_{\alpha\in\bZ^r, \alpha\neq 0} \fg_{\alpha}$.
We define the set of roots of $(A, \bp)$ by
$$ \nabla^{(A,\bp)}:=\{\alpha\in\bZ^r - 0\colon \fg_\alpha\neq0\},$$
 and consider as usual the subsets of positive and negative roots: $\nabla_{\hspace{2pt} \pm}^{(A,\bp)}=\nabla^{(A,\bp)}\cap(\pm\bN^{\bI}_0)$.

\section{Constructions for operadic Lie algebras in symmetric tensor categories}
Fix a strict symmetric tensor category $\cC$. Essentially all constructions that are known for (operadic) Lie algebras in $\Vec$ or $\sVec$ should have a counterpart in our setting. Here we include some of these for completeness.

\subsection{Modules over operadic Lie algebras}
With the goal of laying the foundations needed for our construction in \Cref{sec:contragredient-verp}, we develop some abstract nonsense representation theory for operadic Lie algebras in symmetric categories. A more in-depth analysis for the representation theory of affine group schemes in pre-Tannakian categories, towards the development of commutative algebra over incompressible categories, is available in \cites{Ven2, Cou-comm}.

Following e.g. \cite{Rumynin}*{\S 2.2}, a module over $\fg$ is an object $V$ in $\cC$ equipped with a map $\eta_V \colon \fg \ot V \to V$ for which the following diagram commutes
\begin{equation}\label{eq:module-Lie-alg-defn}
\begin{tikzcd}[row sep=tiny]
\fg \ot \fg \ot V 
\arrow{rr}{b\ot \id_V} 
\arrow{dd}[swap]{(\id_{\fg\ot \fg} - c_{\fg, \fg})\ot \id_V} 
&&
\fg \ot V \arrow{rrd}{\eta_V} &&
\\
&& && 
V.
\\
\fg \ot \fg \ot V 
\arrow{rr}{\id_\fg \ot \eta_V} 
&&
\fg \ot V \arrow{rru}{\eta_V} &&
\end{tikzcd}
\end{equation}

\begin{remark}\label{rem:module-representation}
Following \cite{Ven-GLVerp}*{\S 3}, the \emph{general Lie algebra} over an object $V\in\cC$ is the Lie algebra $\fgl(V)$ attached to the associative algebra $V\ot V^*$ with multiplication $\id_V\ot \ev_{V}\ot\id_{V^*}$. By \cite{Ven-GLVerp}*{Proposition 3.12}, Lie algebra maps $\fg\to\fgl(V)$ are in correspondence with $\fg$-module structures on $V$.
\end{remark}

Let $W$ be another $\fg$-module. A morphism of $\fg$-modules from $V$ to $W$ is a map $f\in\Hom_{\cC}(V,W)$ that commutes with the actions of $\fg$. The set of such morphisms is denoted by $\Hom_{\fg}(V,W)$ as usual. The tensor product $V\ot W\in\cC$ is also a $\fg$-module with action
\begin{align}\label{eq:action-tensor-prod}
\eta_{V\ot W} \coloneqq \eta_V\ot \id_W+(\id_V\ot \eta_W)(c_{\fg,V}\ot\id_W).
\end{align}
The dual $V^*$ is also a $\fg$-module with action $\eta_{V^*}:\fg\ot V^*\to V^*$,
\begin{align}\label{eq:action-dual}
\eta_{V^*} &= -(\ev\ot \id_{V^*})(\id_{V^*}\ot \eta \ot\id_{V^*})(\id_{\fg\ot V^*}\ot \coev)c_{\fg, V^*}.
\end{align}
Notice that $\fg$ and $\one$ are $\fg$-modules with $\eta_{\fg}=b$, $\eta_{\one}=0$; and combining this with \eqref{eq:action-dual}, $\fg^*$ is a $\fg$-module. The maps $\ev$ and $\coev$ are morphisms. Moreover, by \cite{Rumynin}*{Proposition 2} the family of all $\fg$-modules is a tensor category with unit $(\one,\eta_{\one})$. We denote it by $\gMod{\cC}{\fg}$. 

\begin{remark}\label{rem:g-mod-Ug-mod} As usually, $\Mod_{\cC}(\fg)$ is canonically equivalent, as an abelian category, to $\Mod_{\cC}(\op{U}(\fg))$, the category of modules in $\cC$ for the associative algebra $\op{U}(\fg) \in \indcat{\cC}$.

Indeed, any $\fg$-action $\eta \colon V\ot\fg\to V$ as above, induces a map $\eta\colon \op{T}(\fg)\otimes V\to V$ endowing $V$ with the structure of a module over the associative algebra $\op{T}(\fg)$. Here $\eta_n\coloneq \restr{\eta}{\fg^{\ot n}\ot V}$ is defined recursively by taking $\eta_0\colon\one\ot V\to V$ as the unit isomorphism, and $\eta_n=\eta_{n-1}(\id_{\fg^{\ot n-1}}\ot\eta)$  for $n>0$. 
Since $V$ is a $\fg$-module, this map descends to the quotient $\op{U}(\fg)\ot V\to V$, and endows $V$ with a canonical $\op{U}(\fg)$-module structure.

Conversely, a $\op{U}(\fg)$-module becomes a $\fg$-module via pullback along the natural map $\fg\to\op{U}(\fg)$.
\end{remark}

By \cite{Maj92}*{Proposition 2.5} together with Remark \ref{rem:g-mod-Ug-mod}, we can conclude that $\gMod{\cC}{\fg}$ is a tensor category by proving that $\op{U}(\fg)$ is a Hopf algebra.
Let $\iota\colon \one\to \op{T}(\fg)$ be the unit map. As in \cite{Ven2}*{Definition 5.27}, $\op{T}(\fg)$ is a Hopf algebra, where the comultiplication $\Delta\colon \op{T}(\fg)\to\op{T}(\fg)\ot\op{T}(\fg)$ is the unique algebra map such that 
$$\restr{\Delta}{\fg}=\id_{\fg}\ot\iota+\iota\ot\id_{\fg}. $$
Clearly $\op{T}(\fg)$ is a cocommutative Hopf algebra.

\begin{lemma}\label{lem:Ug-Hopf-algebra}
The map $\Delta$ descends to an algebra map $\Delta \colon \op{U}(\fg)\to\op{U}(\fg)\ot\op{U}(\fg)$, and this makes $\op{U}(\fg)$ a cocommutative Hopf algebra.
\end{lemma}
\begin{proof}
As for usual Lie algebras,
\[\restr{\Delta}{\fg^{\ot 2}}=\id_{\fg^{\ot 2}}\ot\iota+(\id_{\fg^{\ot 2}}+c)\iota\ot\id_{\fg^{\ot 2}}:\fg^{\ot 2}\to \fg^{\ot 2}\ot\one\oplus \fg\ot \fg \oplus \one \otimes \fg^{\ot 2},\]
and by direct computation, $\restr{\Delta}{\fg^{\ot 2}}\circ a =a\ot \iota +\iota \otimes a$. Thus the ideal generated by the image of $a$ is a bi-ideal. The antipode on $T(V)$ descends to an antipode on $\op{U}(\fg)$.
\end{proof}

\begin{remark}\label{rem:modules-over-free-algebras}
Let $V,W\in\cC$ and $\overline{\eta}:V\ot W\to W$ a map.
\begin{enumerate}[leftmargin=*,label=\rm{(\roman*)}]
\item $W$ is a $\op{T}(V)$-module, with action $\eta=\oplus_{n\ge 0}\eta_n\colon \op{T}(V)\ot W\to W$ given by
\begin{align*}
\eta_n &:V^{\ot n}\ot W\to W, & \eta_1&=\overline{\eta}, & \eta_{n+1} &=\eta_n (\id_{V^{\ot n}}\ot \overline{\eta})
\end{align*}
\item By \Cref{prop:FLie-subalg-TV}, $\FLie(V)$ can be viewed as a Lie subalgebra of $\op{T}(V)$. The restriction of $\eta$ to $\FLie(V)$ makes $W$ a module over $\FLie(V)$.
\end{enumerate}
\end{remark}

\subsubsection{Actions by derivations}
Let $(\ff,m)\in\cC$ be an operadic Lie (respectively, an associative) algebra with bracket (respectively, multiplication) $m\colon \ff\ot\ff\to\ff$ and let $\rho \colon \fg \otimes \ff \to \ff$ denote a $\fg$-module structure on the underlying object.
We say that $\fg$ \emph{acts by derivations} on $(\ff,m)$ if $m$ is a $\fg$-module morphism:
\[\begin{tikzcd}
\fg \ot \ff \ot \ff 
\arrow{rr}{\id_{\fg} \ot m} 
\arrow{d}[swap]{\rho_{\ff\ot\ff}} 
&& 
\fg \ot \ff 
\arrow{d}{\rho}
\\
\ff \ot \ff  
\arrow{rr}{m}
&&\ff.
\end{tikzcd}\] 
Notice the presence of the braiding $c_{\fg,\ff}$ in the definition of $\rho_{\ff\ot\ff}$. 

This generalizes the notion of an  action by derivations for (operadic) Lie algebras in $\sVec$, where an \emph{even} action $\rho \colon \fg \otimes \ff \to \ff$ is by derivations if 
\begin{align*}
\rho(x)[s, t]= [\rho (x) s, t] + (-1)^{\vert x \vert \vert s \vert} [s, \rho(x)t]. 
\end{align*}

\begin{lemma}\label{lem:action-deriv-assoc-to Lie}
Let $\fg$ an operadic Lie algebra with an action by derivations $\rho_A$ on an associative algebra $(A,m)$ in $\cC$. Then $\fg$ acts by derivations on the Lie algebra $(A,[\cdot,\cdot])$.
\end{lemma}
\begin{proof}
By hypothesis, $\rho_A(\id_{\fg}\ot m)=m\rho_{A\ot A}$. Using this identity,
\begin{align*}
\rho_A(\id_{\fg}\ot [\cdot,\cdot]) &= \rho_A\left(\id_{\fg}\ot (m-mc_{A,A})\right)
= \rho_A\left(\id_{\fg}\ot m\right)- \rho_A\left(\id_{\fg}\ot m\right)\left(\id_{\fg}\ot c_{A,A}\right)
\\
& = m\rho_{A\ot A}- m\rho_{A\ot A} \left(\id_{\fg}\ot c_{A,A}\right),
\\
[\cdot,\cdot] \rho_{A\ot A} &= m\rho_{A\ot A}- m c_{A,A}\rho_{A\ot A},
\end{align*}
so it is enough to prove that $\rho_{A\ot A} \left(\id_{\fg}\ot c_{A,A}\right)=c_{A,A}\rho_{A\ot A}$, or equivalently,
\begin{align*} 
&(\rho_A\ot \id_A)\left(\id_{\fg}\ot c_{A,A}\right)+(\id_A\ot \rho_A)(c_{\fg,A}\ot\id_A)\left(\id_{\fg}\ot c_{A,A}\right)
\\ & \qquad =c_{A,A}(\rho_A\ot \id_A)+c_{A,A}(\id_A\ot \rho_A)(c_{\fg,A}\ot\id_A).
\end{align*}
The naturality of the braiding implies that 
\begin{align*}
c_{A,A}(\rho_A\ot\id_A)&=(\id_A\ot\rho_A)c_{\fg\ot A,A}, &
c_{A,A}(\id_A\ot\rho_A)&=(\rho_A\ot\id_A)c_{A,\fg\ot A}.
\end{align*}
From these two identities and using that $c$ satisfies the braid equation, we get
\begin{align*}
&c_{A,A}(\rho_A\ot \id_A)+c_{A,A}(\id_A\ot \rho_A)(c_{\fg,A}\ot\id_A) 
\\
& \quad = (\id_A\ot\rho_A)c_{\fg\ot A,A}+(\rho_A\ot\id_A)c_{A,\fg\ot A}(c_{\fg,A}\ot\id_A)
\\
& \quad = (\id_A\ot\rho_A)(c_{\fg,A}\ot\id_A)(\id_{\fg}\ot c_{A,A})+(\rho_A\ot\id_A)
(\id_{\fg}\ot c_{A,A})(c_{A,\fg}\ot\id_A)(c_{\fg,A}\ot\id_A),
\end{align*}
and the claim follows.
\end{proof}

\begin{remark}\label{rem:action-deriv-tensor-algebra}
Given $V\in\gMod{\cC}{\fg}$, the $\fg$-action on $V$ extends uniquely to an $\fg$-action by derivations on the tensor algebra $\op{T}(V)\in\indcat{\cC}$, both as an a associative and as a Lie algebra. Indeed the associative case is straightforward while the Lie case follows from Lemma \ref{lem:action-deriv-assoc-to Lie}.
\end{remark}

\begin{lemma}\label{lem:extended-action}
Any action of $\fg$ on an object $V$ extends uniquely to an action of $\fg$ on $\FLie(V)$ by derivations.
\end{lemma}
\begin{proof}
By Proposition \ref{prop:FLie-subalg-TV} we may identify $\FLie(V)$ with the Lie subalgebra $\ff$ of $\op{T}(V)$ generated by $V$. The $\fg$-action on $V$ extends uniquely to an action on $\op{T}(V)$ by derivations, so we have to prove that $\ff$ is a $\fg$-submodule.
By Remark \ref{rem:subalg-generated}, $\ff=\oplus_{n\ge 1} \ima \varphi_n$, so it suffices to check that each $\ima \varphi_n$ is a $\fg$-submodule. This last statement follows by induction on $n$, using that the action on $\op{T}(V)$ is by derivations.
\end{proof}

\subsection{Semidirect product of operadic Lie algebras}
Now we are ready to introduce the notion of semidirect product.

\begin{definition}
Let $\fg$ and $\ff$ be operadic Lie algebras in $\cC$ with an action $\rho \colon \fg \ot \ff \to \ff$ by derivations. The \emph{Lie algebra semidirect product} $\ff \rtimes_\rho \fg$ is the operadic Lie algebra with underlying space $\ff\oplus\fg$ and Lie bracket $b$ defined by the composition 
\[\begin{tikzcd}
(\ff \oplus \fg) \otimes (\ff \oplus \fg) 
\arrow[equal]{d}
\arrow{rrrr}{b}
&&&&
\ff \oplus \fg
\\ 
\ff \otimes \ff \, \oplus \, \fg \otimes \ff  \, \oplus \, \ff \otimes \fg  \, \oplus \, \fg \otimes \fg  
\arrow{rrrr}{b_{\ff} \oplus \rho \oplus( -\rho c_{\ff,\fg} )\oplus b_{\fg}}
&&&&
\ff \oplus \ff \oplus \ff \oplus \fg.
\arrow{u}[swap]{(\id_\ff \sqcup \id_\ff \sqcup \id_\ff) \oplus \id_\fg}
\end{tikzcd}\] 
\end{definition}

\begin{remark}
Let $\iota_1:\ff\to\ff\oplus\fg$, $\pi_1:\ff\oplus\fg\to\ff$, $\iota_2:\fg\to\ff\oplus\fg$, $\pi_2:\ff\oplus\fg\to\fg$ be the canonical maps. 
Identifying $\ff \otimes \ff \, \oplus \, \fg \otimes \ff  \, \oplus \, \ff \otimes \fg  \, \oplus \, \fg \otimes \fg $ with $(\ff \oplus \fg) \otimes (\ff \oplus \fg)$ via the maps $\iota_i\otimes\iota_j$, $i,j\in\{1,2\}$, the bracket of $\ff \rtimes_\rho \fg$ is the map $b$
such that
\begin{align*}
\restr{b}{\ff\ot\ff} &= \iota_1 \circ b_{\ff}, &
\restr{b}{\ff\ot\fg} &= -\iota_1 \circ \rho\circ c_{\ff,\fg}, \\
\restr{b}{\fg\ot\ff} &= \iota_1 \circ \rho, &
\restr{b}{\fg\ot\fg} &= \iota_2 \circ b_{\fg}.
\end{align*}
\end{remark}

In the case of operadic Lie algebras in $\sVec$ with an action by derivations $\rho \colon \fg \ot \ff \to \ff$ in $\sVec$ (i.e., an even map), we recover the usual definition of the semidirect product $\ff \rtimes_\rho \fg$:
\begin{align*}
\big[(t_1, x_1), (t_2, x_2)\big] = \left[[t_1,t_2] + \rho(x_1)t_2 - (-1)^{\vert t_1 \vert \vert x_2 \vert} \rho(x_2)t_1 , [x_1,x_2] \right].
\end{align*}

The next crucial example gives an element-free description of contragradient Lie superalgebras, and will serve as a guide throughout \Cref{sec:contragredient-verp}, where we construct contragredient Lie algebras in more abstract symmetric tensor categories. In particular, it is the main reason why we are interested in semidirect products of Lie algebras.

\begin{example}\label{exa:classical-semidirect}
Consider data as in \ref{item:contragredient-1}--\ref{item:contragredient-4}, and view $\fh$ as a purely even abelian Lie algebra in $\sVec$. 
From these, we obtain an $\fh$-module $\rho\colon \fh\ot V \to V$ in $\sVec$ and a compatible map $d\colon V\ot V^*\to \fh$ as follows.

As an object, $V \in \sVec$ has a basis $(e_i)_{1\leq i \leq r}$ according to the given parity vector; we obtain an action $\rho$ on $V$ by letting $\fh$ act each on each $e_i$ by the given character $\xi_i$. On the dual module $V^*$, with dual basis $(f_i)_{1\leq i \leq r}$, the action of $\fh$ is given by the characters $(-\xi_i)$. Extend these to an action by derivations $\rho$ of $\fh$ on the free Lie superalgebra $\FLie(V\oplus V^*)$. 

In this case $\FLie(V\oplus V^*) \rtimes_\rho \fh$ is the Lie superalgebra defined by the relations \eqref{eq:relaciones gtilde} except for $[e_i, f_j]=\delta_{ij} h_i$. To realize the missing relation, let $d$ be the map
\[\begin{tikzcd}[row sep=0.1 em]
d\colon V\otimes V^* \arrow{r} & \fh, && e_i\ot f_j \arrow[r, mapsto] & \delta_{ij}h_i.
\end{tikzcd}\]
Then the Lie superalgebra $\tfg(A,\bp)$ associated to \ref{item:contragredient-1}--\ref{item:contragredient-4} is just the quotient of the semidirect product $\FLie(V\oplus V^*) \rtimes_\rho \fh$ by the ideal generated by the image of
\[
([\cdot,\cdot],-d) \colon V\otimes V^* \longrightarrow \FLie(V\oplus V^*)\rtimes_\rho \fh,
\]
where $[\cdot,\cdot]$ denotes the restriction of the Lie bracket of $ \FLie(V\oplus V^*)$ to $V\otimes V^*$. Notice that $d$ is a map of $\fh$-modules, i.e.
\begin{align*}
[h, d(e_i\otimes f_j)] &=0 = d\left( h\cdot e_i \otimes f_j\right)+ d\left(e_i \otimes h\cdot f_j\right), &h&\in\fh, &1&\leq i,j\leq r.
\end{align*}
Moreover, this compatibility between the bracket in $\fh$ and the map $d$ is a necessary consequence of the Jacobi identity in $\tfg(A,\bp)$.
\end{example}

The semidirect product in $\OLie(\cC)$ satisfies the following expected property.

\begin{lemma}\label{lem:semidirect-split}
Let $\fg$ and $\ff$ be operadic Lie algebras in $\cC$ with an action $\rho \colon \fg \ot \ff \to \ff$ by derivations. 
Then $\iota_1\colon\ff\to\ff\rtimes_\rho\fg$, $\iota_2\colon\fg\to\ff\rtimes_\rho\fg$, $\pi_2\colon\ff\rtimes_\rho\fg\to\fg$ are morphisms of operadic Lie algebras.
Thus the semidirect product $\ff \rtimes_\rho \fg$ fits in a short exact sequence 
\begin{align*}
0\to \ff \to \ff \rtimes_\rho \fg \to \fg \to 0
\end{align*}
of operadic Lie algebras in $\cC$, which splits under $\iota_2$.\qed
\end{lemma}

\begin{lemma}\label{lem:semidirect-p-Jac}
Let $\fg$ and $\ff$ be operadic Lie algebras with an action by derivations $\rho \colon \fg \ot \ff \to \ff$ in a symmetric tensor category. If $\fg$ and $\ff$ satisfy the $p$-Jacobi identity, then so does $\ff \rtimes_\rho \fg$.
\end{lemma}
\begin{proof}
We have to prove that $\beta^{\ff \rtimes_\rho \fg}$  annihilates $E_p(\ff \rtimes_\rho \fg)$. By \cite{Etingof}*{Proposition 4.5},
\begin{align*}
E_p(\ff \rtimes_\rho \fg) &= \ker \phi^{\ff \rtimes_\rho \fg} _p = \ker \phi^{\ff \oplus \fg}_p =
E_p(\ff \oplus \fg) = E_p(\ff) \oplus E_p(\fg).
\end{align*}
On the other hand, by naturality of $\beta$, and since the canonical inclusions $\ff, \fg\hookrightarrow \ff \rtimes_\rho \fg$ are Lie algebra maps, we have $\restr{\beta^{\ff \rtimes_\rho \fg}}{\ff}=\beta^{\ff}$ and $\restr{\beta^{\ff \rtimes_\rho \fg}}{\fg}=\beta^{\fg}$. By hypothesis, $\restr{\beta^{\ff}}{E_p(\ff)}=\restr{\beta^{\fg}}{E_p(\fg)}=0$, so $\beta^{\ff \rtimes_\rho \fg}$ annihilates $E_p(\ff) \oplus E_p(\fg)=E_p(\ff \rtimes_\rho \fg)$.
\end{proof}

\begin{corollary}\label{cor:Fr-exact-semidirect-closed}
If $\cC$ is a Frobenius exact pre-Tannakian category of moderate growth, then the subcategory $\Lie(\indcat{\cC})$ of $\OLie(\indcat{\cC})$ is closed under semidirect products.
\end{corollary}

\begin{proof}
Follows immediately from \Cref{thm:Lie-alg-equivalences} and \Cref{lem:semidirect-p-Jac}
\end{proof}

\begin{remark}
The semidirect product is universal. Indeed, we can imitate the element-free universal property for semidirect products of finite groups to describe the assignment $(\fg, \ff, \rho) \mapsto (\fg \hookrightarrow \ff \rtimes_\rho \fg)$ as a left adjoint. 

Consider on one hand $\mathsf{OLieActions}(\cC)$, the category with objects $(\fg, \ff, \rho)$, where $\rho \colon \fg \ot \ff \to \ff$ is an action of operadic Lie algebras by derivations, and arrows $(\fg, \ff, \rho)\to (\fg', \ff', \rho')$ given by pairs $(f_1\colon \fg \to \fg', f_2 \colon \ff \to \ff')$ satisfying the obvious compatibility.

On the other hand, let $\mathsf{Arr}\OLie(\cC)$ denote the category whose objects are the arrows in $\OLie(\cC)$, and where morphisms are commutative diagrams.

We introduce the \emph{pullback forgetful functor} $\mathsf{Arr}\OLie(\cC) \to \mathsf{OLieActions}(\cC)$, which sends an arrow $\phi \colon \fg \to \ff$ in $\OLie(\cC)$ to the pullback of the adjoint action of $\ff$ along $\phi$, that is,
\[\begin{tikzcd}
\fg \otimes \ff \arrow{r}{\phi \ot \id} &\ff \ot \ff \arrow{r}{b_\ff} & \ff.
\end{tikzcd}\]
Then the assignment $(\fg, \ff, \rho) \mapsto (\fg \hookrightarrow \ff \rtimes_\rho \fg)$ serves as a left adjoint for the pullback.
\end{remark}

\begin{remark}
There is a close relation between the semidirect product introduced above and the smash product from \cite{Ven2}*{Definition 5.30}. 
Namely, we can see $\ff$ as a $\op{U}(\fg)$-module due to Remark \ref{rem:g-mod-Ug-mod}, and then construct the smash product
$\op{T}(\ff)\rtimes\op{U}(\fg)$. As the $\fg$-action is by derivations, the defining ideal of $\op{U}(\ff)$ is $\fg$-stable (hence also $\op{U}(\fg)$-stable), so we can quotient and obtain a smash product Hopf algebra $H=\op{U}(\ff)\rtimes\op{U}(\fg)$. 

Assume further that $\ff$ and $\fg$ are PBW, so we can view $\ff$ and $\fg$ as subobjects of $\op{U}(\ff)$ and $\op{U}(\fg)$, respectively. 
Since the commutator of $H$ restricted to $\ff\oplus \fg$ coincides with that of $\ff \rtimes_\rho \fg$, we may identify $\ff\oplus \fg\hookrightarrow \ff\ot\one\oplus \one \otimes\fg \subset H$. From here, $H\simeq \op{U}(\ff \rtimes_\rho \fg)$.
\end{remark}

Next we describe the obstruction governing whether a pair morphisms from the factors can be extended to a morphisms from their semidirect product.

\begin{proposition}\label{prop:morphism-from-semidirect-product}
Let $\ff, \fg, \fh\in\cC$ be operadic Lie algebras. Assume $\fg$ acts by derivations on $\ff$ via $\rho \colon \fg \ot \ff \to \ff$, and let $\phi_{\ff}\colon \ff\to\fh$, $\phi_{\fg}\colon\fg\to\fh$ be Lie algebra morphisms. Set $\phi\coloneqq\phi_{\ff}\sqcup\phi_{\fg}:\ff \rtimes_\rho \fg\to\fh$. 
Then $\phi$ is a Lie algebra morphism if and only if
\begin{align}\label{eq:morphism-from-semidirect-product-compatibility}
\phi_{\ff} \circ \rho&= b_{\fh}\circ (\phi_{\fg}\ot\phi_{\ff})\colon \fg\ot\ff\to\fh.
\end{align}
\end{proposition}
\begin{proof}
Assume first that $\phi$ is a Lie algebra map, i.e. $\phi\circ  b=b_{\fh}\circ (\phi\ot\phi)$. Then
\begin{align*}
\phi_{\ff}\circ \rho &= \phi \circ \iota_1 \circ \rho =\phi \circ b \circ (\iota_1\ot \iota_2)
= b_{\fh}\circ (\phi\ot\phi) \circ (\iota_1\ot \iota_2) = b_{\fh}\circ (\phi_{\fg}\ot\phi_{\ff}),
\end{align*}
so \eqref{eq:morphism-from-semidirect-product-compatibility} holds. 

Reciprocally, assume that \eqref{eq:morphism-from-semidirect-product-compatibility} holds. The equality $\phi_{\fh}\circ  b=b\circ (\phi\ot\phi)$
holds if and only if $\phi\circ  b \circ (\iota_i\ot \iota_j)=b_{\fh}\circ (\phi\ot\phi)\circ (\iota_i\ot \iota_j)$ holds for all $i,j\in\{1,2\}$. The case
$i=1$, $j=2$ is exactly \eqref{eq:morphism-from-semidirect-product-compatibility} while $i=2$, $j=1$ follows by composing with $c_{\fg,\ff}$, and the cases
$i=j=1$ and $i=j=2$ follow since $\phi_{\ff}$, $\phi_{\fg}$ are Lie algebra morphisms, respectively. Thus $\phi$ is a Lie algebra morphism.
\end{proof}

\subsubsection{Modules over a semidirect product} Finally we construct modules over the semidirect product starting from an object which is simultaneously a module over both Lie algebras.

\begin{proposition}\label{prop:module-over-semidirect-product}
Assume we have an action by derivations $\rho \colon \fg \ot \ff \to \ff$ of operadic Lie algebras. Let $V\in\cC$ be simultaneously a $\fg$-module and an $\ff$-module, with actions $\eta_{\ff}\colon \ff\ot V\to V$, $\eta_{\fg}\colon\fg\ot V\to V$. Set $\eta\coloneqq\eta_{\ff}\sqcup\eta_{\fg}\colon \ff \rtimes_\rho \fg\ot V\to V$
\footnote{Here we identify $\ff \rtimes_\rho \fg\ot V = \ff \ot V\oplus \fg\ot V$ as objects in $\cC$.}.

Then $\eta$ defines an  $\ff \rtimes_\rho \fg$-module structure on $V$ if and only if
\begin{align}\label{eq:module-over-semidirect-product-compatibility}
\eta_{\ff}(\rho\ot \id_{V}) &= \eta_{\fg}(\id_{\fg}\ot \eta_{\ff})-\eta_{\ff}(\id_{\ff}\ot \eta_{\fg})(c_{\fg,\ff}\ot \id_{V}).
\end{align}
\end{proposition}
\begin{proof}
If $\eta$ gives an action, then we get \eqref{eq:module-over-semidirect-product-compatibility} by restricting \eqref{eq:module-Lie-alg-defn} to $\fg\ot\ff \ot V \subset (\ff \rtimes_\rho \fg)^{\ot2}\ot V$. 

On the other hand, assume that $\eta_{\ff}$ and $\eta_{\fg}$ satisfy \eqref{eq:module-over-semidirect-product-compatibility}. We have to check that $\eta$ makes the diagram \eqref{eq:module-Lie-alg-defn} commutative, i.e, the two obvious maps $(\ff \rtimes_\rho\fg)^{\ot 2} \ot V \to V$ are actually equal. It suffices to show so when restricted to $\ff\ot\ff \ot V$, $\fg\ot\ff \ot V$, $\ff\ot\fg \ot V$, and $\fg\ot\fg \ot V$. The equality of these restrictions to the first and the fourth summands follow since $\eta_{\ff}$ and $\eta_{\fg}$ are actions. The second one is exactly \eqref{eq:module-over-semidirect-product-compatibility}, and the third one is just \eqref{eq:module-over-semidirect-product-compatibility} pre-composed with $-c_{\ff,\fg}\ot\id_V$.
\end{proof}

A particularly relevant example of semidirect product can be obtained by taking $\ff=\FLie(W)$ for a $\fg$-module $W$.  We show next that it suffices to check the compatibilities in Propositions \ref{prop:morphism-from-semidirect-product} and \ref{prop:module-over-semidirect-product} just over $W$.

\begin{lemma}\label{lem:semidirect-compatibility-free-algebra}
Let $W\in\cC$ be a $\fg$-module, and denote by $\rho$ the corresponding action of $\fg$ on $\ff=\FLie(W)$ by derivations as in \Cref{lem:extended-action}.
\begin{enumerate}[leftmargin=*,label=\rm{(\roman*)}]
\item\label{item:semidirect-compatibility-free-algebra-morphism} Let $\fh\in\cC$, $\phi_{\ff}\colon \ff\to\fh$, $\phi_{\fg}\colon\fg\to\fh$ be as in \Cref{prop:morphism-from-semidirect-product}. 
Then $\phi\coloneqq\phi_{\ff}\sqcup\phi_{\fg}:\ff \rtimes_\rho \fg\to\fh$.  is a Lie algebra map if and only if \eqref{eq:morphism-from-semidirect-product-compatibility} holds on $\fg\ot W$.

\item\label{item:semidirect-compatibility-free-algebra-module} 
Let $V\in\cC$, $\eta_{\ff}\colon \ff\ot V\to V$, $\eta_{\fg}\colon\fg\ot V\to V$ be as in \Cref{prop:module-over-semidirect-product}. 
Then $V$ is an $\ff \rtimes_\rho \fg$-module via $\eta\coloneqq\eta_{\ff}\sqcup\eta_{\fg}\colon \ff \rtimes_\rho \fg\ot V\to V$ if and only if \eqref{eq:module-over-semidirect-product-compatibility} holds on $\fg\ot W\ot V$.
\end{enumerate}
\end{lemma}
\begin{proof}
As in Remark \ref{rem:subalg-generated}, let $\varphi_1:W\hookrightarrow \ff$ be the inclusion and $\varphi_{n+1}\coloneq \restr{b_{\ff}}{\ima\varphi_n\otimes W}$. 

\smallbreak
The forward implication in \ref{item:semidirect-compatibility-free-algebra-morphism} follows directly from \Cref{prop:morphism-from-semidirect-product}. Conversely, assume that \eqref{eq:morphism-from-semidirect-product-compatibility} holds on $\fg\ot W$. As $\ff=\oplus_{n\ge 1} \ima\varphi_n$, we have to check that
\begin{align*}
 \phi_{\ff} \circ \rho \circ(\id_{\fg}\ot \varphi_n) &= b_{\fh}\circ (\phi_{\fg}\ot\phi_{\ff})\circ(\id_{\fg}\ot \varphi_n) & & \text{for all }n\in\bN.
\end{align*}
The proof is by induction on $n$; the case $n=1$ is exactly our assumption. Now assume that the equality above holds for $n$. Using that $\rho$ acts by derivations, $\phi_{\ff}$ is a Lie algebra map, inductive hypothesis and the Jacobi identity
\begin{align*}
\phi_{\ff} \circ & \rho \circ(\id_{\fg}\ot \varphi_{n+1}) = \phi_{\ff} \circ \rho \circ \restr{b_{\ff}}{\fg\ot \ima\varphi_n\otimes W}
\\
&= \phi_{\ff} \circ b_{\ff} \circ \restr{\big(\rho \ot \id_{W} + (\id\ot \rho)(c_{\fg,\ima\varphi_n}\ot\id_{W})\big)}{\fg\ot \ima\varphi_n\otimes W}
\\ 
&= b_{\fh} \circ \restr{\big(\phi_{\ff}\circ \rho \ot \phi_{\ff} 
+ (\phi_{\ff}\ot \phi_{\ff} \circ \rho)(c_{\fg,\ima\varphi_n}\ot\id_{W})\big)}{\fg\ot \ima\varphi_n\otimes W}
\\ 
&= b_{\fh} \circ \restr{\big(b_{\fh}\circ (\phi_{\fg}\ot\phi_{\ff}) \ot \phi_{\ff} 
+ (\phi_{\ff}\ot b_{\fh}\circ (\phi_{\fg}\ot\phi_{\ff}))(c_{\fg,\ima\varphi_n}\ot\id_{W})\big)}{\fg\ot \ima\varphi_n\otimes W}
\\ 
&= b_{\fh} \circ (\id_{\fh}\ot b_{\fh}) \circ \restr{\big(\phi_{\fg}\ot\phi_{\ff}\ot \phi_{\ff}\big)}{\fg\ot \ima\varphi_n\otimes W}
\\ 
&= b_{\fh} \circ \restr{\big(\phi_{\fg}\ot\phi_{\ff}\circ b_{\fh}\big)}{\fg\ot \ima\varphi_n\otimes W}
= b_{\fh}\circ (\phi_{\fg}\ot\phi_{\ff})\circ(\id_{\fg}\ot \varphi_{n+1}),
\end{align*}
thus the inductive step follows.

\smallbreak
The proof of \ref{item:semidirect-compatibility-free-algebra-module} follows similarly.
\end{proof}

\section{Contragredient Lie algebras}\label{sec:contragredient-verp}

This Section contains the main constructions and results of this work. The goal  is to construct a notion of contragredient Lie algebras in a symmetric tensor category $\cC$. The input necessary for the construction will be called contragredient data. Actually, we will introduce three hierarchically ordered notions of such data, starting with a rather relaxed one that will progressively get richer, thus allowing proofs fore deeper structural results. 

\subsection{Contragredient data}\label{subsec:contragredient-data} We begin with the following data:
\begin{definition}\label{def:contragredient-data}
Consider a Lie algebra $(\ttX,\ttb_{\ttX})$ in a strict symmetric tensor category $\cC$. A \emph{contragredient datum} over $\ttX$ is a pair $(\rho, \ttd)$ where:
\begin{enumerate}[leftmargin=*,label=\rm{(D\arabic*)}]
\item\label{item:Ver-contragredient-1} $\rho\colon \ttX \ot \ttV \to \ttV$ is an $\ttX$-module in $\cC$, and  
\item \label{item:Ver-contragredient-2} $\ttd \colon \ttV \ot \ttV^*\to \ttX$ is a morphism in $\cC$, 
\end{enumerate}
which are compatible in the following sense:
\begin{align}\label{eq:contragredient-compatibility}
\ttb_{\ttX}(\id_{\ttX} \ot \ttd) = \ttd(\rho \ot \id_{\ttV^*})+\ttd(\id_{\ttV} \ot \rho^{\vee})(c_{\ttX, \ttV} \ot \id_{\ttV^*}).
\end{align}
Here, as in \eqref{eq:action-dual}, $\rho^{\vee}\colon \ttX \ot \ttV^* \to \ttV^*$ is the dual module defined by the composition 
\begin{equation}\label{eq:rho-vee}
\begin{tikzcd}
\rho^{\vee} \colon \ttX\ot\ttV^* \arrow{rr}{ c_{\ttX, \ttV^*}\ot \coev_{\ttV}}  && \ttV^*\ot\ttX \ot \ttV \ot \ttV^* \arrow{rr} {\ttV^* \ot (-\rho)\ot \ttV^*}&&\ttV^*\ot \ttV \ot \ttV^* \arrow{rr} {\ev_{\ttV} \ot \ttV^*}&&\ttV^*,
\end{tikzcd}
\end{equation}
and \eqref{eq:contragredient-compatibility} means that $\ttd$ is a map of $\ttX$-modules.

We denote by $\opD(\ttX)$ the set of pairs $(\rho, \ttd)$ as above. If we want to emphasize $\ttV$, we say that $(\rho,\ttd)$ is an $\ttX$-contragredient structure on $\ttV$.
\end{definition}

\begin{remark}\label{rem:contragredient-dual-pair}
Let $(\rho, \ttd)\in \opD(\ttX)$. As $\cC$ is symmetric, $\ttV$ has a natural structure of dual object to $\ttV^*$ and $(\rho^{\vee})^{\vee}$ identifies with $\rho$. That is, there exists a natural isomorphism $\vartheta_{\ttV}\colon\ttV\overset{\sim}{\to} \ttV^{**}$ which is a map of $\ttX$-modules. Define $\ttd^{\vee} \colon \ttV^* \ot \ttV^{**}\to \ttX$ by
\begin{equation}
\begin{tikzcd}
\ttd^{\vee} \colon\ttV^* \ot \ttV^{**} \arrow{rr}{\ttV^*\ot \vartheta_{\ttV}^{-1}}  && \ttV^*\ot \ttV \arrow{rr} {c_{\ttV^*,\ttV}}&& \ttV \ot \ttV^* \arrow{rr} {-\ttd}&&\ttX.
\end{tikzcd}
\end{equation}
Then one can see that the pair $(\rho^{\vee},\ttd^{\vee})$ satisfies the compatibility \eqref{eq:contragredient-compatibility}, i.e. $(\rho^{\vee},\ttd^{\vee})\in \opD(\ttX)$.
\end{remark}

Given a contragredient datum, and employing the usual diagrammatic calculus in strict symmetric tensor categories\footnote{We read diagrams from top to bottom.}, we introduce the following:

\begin{notation} \label{notation:contragredient-setup}
\begin{enumerate}[leftmargin=*]
\item \label{item:contragredient-diagrams} The morphisms $\ttb_{\ttX}$, $\id_{\ttV}$, $\id_{\ttV^*}$, $\rho$,  and $\ttd$ are denoted by the diagrams 
\begin{equation*}
\ttb_{\ttX}=\xy
(0,0)*{
\begin{tikzpicture}[scale=.3]
\draw [very thick, blue, bend right=40] (-2,0) to (0,-2);
\draw [very thick, blue, bend left=40] (2,0) to (0,-2);
\draw [very thick, blue] (0,-2) to (0,-4);
\end{tikzpicture} 
};
\endxy, 
\qquad
\id_{\ttV}=\xy
(0,0)*{
\begin{tikzpicture}[scale=.3]
\draw [very thick, ->] (0,0) to (0,-4);
\end{tikzpicture} 
};
\endxy, 
\qquad
\id_{\ttV^*}=\xy
(0,0)*{
\begin{tikzpicture}[scale=.3]
\draw [very thick, ->] (0,-4) to (0,0);
\end{tikzpicture} 
};
\endxy, 
\qquad
\rho=\xy
(0,0)*{
\begin{tikzpicture}[scale=.3]
\draw [very thick, blue, bend right=40] (-2,0) to (0,-2);
\draw [very thick, ->] (0,0) to (0,-4);
\draw [blue, fill=blue] (0,-2) circle (0.3cm and 0.3cm);
\end{tikzpicture} 
};
\endxy, 
\qquad
\ttd=\xy
(0,0)*{
\begin{tikzpicture}[scale=.3]
\draw [very thick, bend right=40, ->] (-2,0) to (0,-2);
\draw [very thick, bend left=40, <-] (2,0) to (0,-2);
\draw [very thick, blue] (0,-2) to (0,-4);
\end{tikzpicture} 
};
\endxy;
\end{equation*}
we introduce the following shortcut notation for $\rho^{\vee}$:
\begin{equation*}
\rho^\vee=\xy
(0,0)*{
\begin{tikzpicture}[scale=.3]
\draw [very thick, blue, bend right=40] (-2,0) to (0,-2);
\draw [very thick, <-] (0,0) to (0,-4);
\draw [blue, fill=blue] (0,-2) circle (0.3cm and 0.3cm);
\end{tikzpicture} 
};
\endxy
\coloneqq
\xy
(0,0)*{
\begin{tikzpicture}[scale=.3]
\draw [very thick, blue, bend right=40] (-2,0) to (2,-3);
\draw [very thick, <-] (0,0) to (0,-4);
\draw [very thick] (0,-4) to [out=-90, looseness=2, in=-90] (2,-4);
\draw [very thick, ->] (2,-2) to (2, -4);
\draw [very thick] (2,-2) to [out=90, looseness=2, in=90] (4,-2);
\draw [very thick, ->]  (4,-5.5) to (4,-2);
\draw [blue, fill=blue] (2,-3) circle (0.3cm and 0.3cm);
\end{tikzpicture} 
};
\endxy
\end{equation*}
Notice that, since we are working over a symmetric category, we can just use crossing lines to encode braidings.

\item With the notation above, the compatibility \eqref{eq:contragredient-compatibility} becomes 
\begin{equation*}
\xy
(0,0)*{
\begin{tikzpicture}[scale=.3]
\draw [very thick, blue] (-2,-4) to [out=180, looseness=1, in=-90] (-4,0);
\draw [very thick, bend right=40, ->] (-2,0) to (0,-2);
\draw [very thick, bend left=40, <-] (2,0) to (0,-2);
\draw [very thick, blue] (-2,-4) to [out=0, looseness=1, in=-90] (0,-2);
\draw [very thick, blue] (-2,-4) to (-2,-7);
\end{tikzpicture} 
};
\endxy
=
\xy
(0,0)*{
\begin{tikzpicture}[scale=.3]
\draw [very thick, blue] (-4,2) to [out=-90, looseness=1, in=180] (-2,0);
\draw [very thick] (-2,2) to (-2,0);
\draw [very thick, bend right=40, ->] (-2,0) to (0,-2);
\draw [very thick, bend left=40] (2,0) to (0,-2);
\draw [very thick, ->] (2,0) to (2,2);
\draw [very thick, blue] (0,-2) to (0,-5);
\draw [blue, fill=blue] (-2,0) circle (0.3cm and 0.3cm);
\end{tikzpicture} 
}
\endxy
+
\xy
(0,0)*{
\begin{tikzpicture}[scale=.3]
\draw [very thick, blue] (-4,2) to [out=-90, looseness=1, in=180] (2,0);
\draw [very thick] (-2,2) to (-2,0);
\draw [very thick, bend right=40, ->] (-2,0) to (0,-2);
\draw [very thick, bend left=40] (2,0) to (0,-2);
\draw [very thick, ->] (2,0) to (2,2);
\draw [very thick, blue] (0,-2) to (0,-5);
\draw [blue, fill=blue] (2,0) circle (0.3cm and 0.3cm);
\end{tikzpicture} 
}
\endxy.
\end{equation*}
This implies that for any $\ttX$-module $\lambda\colon \ttX \ot \ttW \to \ttW$, denoted by $\xy
(0,0)*{
\begin{tikzpicture}[scale=.3]
\draw [very thick, blue, bend right=40] (-2.5,0) to (0,-2);
\draw [very thick,red] (0,0) to (0,-4);
\draw [fill=white] (.75,-1.5) rectangle (-.75,-2.5);
\node at (0,-2) {\tiny $\lambda$};
\end{tikzpicture} 
};
\endxy$, we have
\begin{equation}\label{eq:contragradient-compatibility-module}
\xy
(0,0)*{
\begin{tikzpicture}[scale=.3]
\draw [very thick, blue] (-4,0) to [out=-90, looseness=1, in=180]  (4,-6);
\draw [very thick, bend right=40, ->] (-2,0) to (0,-2);
\draw [very thick, bend left=40, <-] (2,0) to (0,-2);
\draw [very thick, blue] (0,-2) to [out=-90, looseness=1, in=180] (4,-4);
\draw [very thick, red] (4,0) to (4,-7);
\draw [fill=white] (3.25,-3.5) rectangle (4.75,-4.5);
\node at (4,-4) {\tiny $\lambda$};
\draw [fill=white] (3.25,-5.5) rectangle (4.75,-6.5);
\node at (4,-6) {\tiny $\lambda$};
\end{tikzpicture} 
};
\endxy
\, \,- \, \, 
\xy
(0,0)*{
\begin{tikzpicture}[scale=.3]
\draw [very thick, blue] (-4,0) to [out=-90, looseness=.5, in=180]  (4,-2);
\draw [very thick] (-2,0) to (-2,-2);
\draw [very thick, <-] (2,0) to (2,-2);
\draw [very thick, bend right=40, ->] (-2,-2) to (0,-4);
\draw [very thick, bend left=40] (2,-2) to (0,-4);
\draw [very thick, blue] (0,-4) to [out=-90, looseness=1, in=180] (4,-6);
\draw [very thick, red] (4,0) to (4,-7);
\draw [fill=white] (3.25,-5.5) rectangle (4.75,-6.5);
\node at (4,-6) {\tiny $\lambda$};
\draw [fill=white] (3.25,-1.5) rectangle (4.75,-2.5);
\node at (4,-2) {\tiny $\lambda$};
\end{tikzpicture} 
};
\endxy
\, \, =\, \, 
\xy
(0,0)*{
\begin{tikzpicture}[scale=.3]
\draw [very thick, blue] (-4,2) to [out=-90, looseness=1, in=180] (-2,0);
\draw [very thick] (-2,2) to (-2,0);
\draw [very thick, bend right=40, ->] (-2,0) to (0,-2);
\draw [very thick, bend left=40] (2,0) to (0,-2);
\draw [very thick, ->] (2,0) to (2,2);
\draw [very thick, blue] (0,-2) to [out=-90, looseness=1, in=180] (4,-4);
\draw [blue, fill=blue] (-2,0) circle (0.3cm and 0.3cm);
\draw [very thick, red] (4,2) to (4,-5);
\draw [fill=white] (3.25,-3.5) rectangle (4.75,-4.5);
\node at (4,-4) {\tiny $\lambda$};
\end{tikzpicture} 
}
\endxy
\, \, + \, \, 
\xy
(0,0)*{
\begin{tikzpicture}[scale=.3]
\draw [very thick, blue] (-4,2) to [out=-90, looseness=1, in=180] (2,0);
\draw [very thick] (-2,2) to (-2,0);
\draw [very thick, bend right=40, ->] (-2,0) to (0,-2);
\draw [very thick, bend left=40] (2,0) to (0,-2);
\draw [very thick, ->] (2,0) to (2,2);
\draw [very thick, blue] (0,-2) to [out=-90, looseness=1, in=180] (4,-4);
\draw [blue, fill=blue] (2,0) circle (0.3cm and 0.3cm);
\draw [very thick, red] (4,2) to (4,-5);
\draw [fill=white] (3.25,-3.5) rectangle (4.75,-4.5);
\node at (4,-4) {\tiny $\lambda$};
\end{tikzpicture} 
}
\endxy.
\end{equation}

\item By \Cref{lem:extended-action}, $\rho$ and $\rho^{\vee}$ induce an action by derivations 
\[\rho\colon \ttX \ot \FLie(\ttV\oplus \ttV^* ) \longrightarrow  \FLie(\ttV\oplus \ttV^* ).\]
\end{enumerate}
\end{notation}

\subsection{The auxiliary operadic Lie algebra \texorpdfstring{$\tttg(\rho, \ttd)$}{g}}\label{subsec:contragredient-gtilde}
\begin{definition}\label{def:gtilde}
Let $(\ttX, \rho, \ttd)$ as in \Cref{def:contragredient-data}. We define the operadic Lie algebra $\tttg(\rho, \ttd)$ as the quotient of $\FLie(\ttV\oplus \ttV^* )\rtimes_\rho \ttX$ by the ideal generated by the image of 
\[(\ttb_{\ttV \ot \ttV^*},-\ttd) \colon \ttV\ot \ttV^* \longrightarrow \FLie(\ttV\oplus \ttV^* )\rtimes_\rho \ttX,
\]
where $\ttb_{\ttV \ot \ttV^*}$ denotes the restriction of the Lie bracket $\ttb$ of $\FLie(\ttV\oplus \ttV^* )$ to $\ttV\ot   \ttV^*$.
\end{definition}

\begin{remark}\label{rem:contragredient-gtilde}
\begin{itemize}[leftmargin=*]
\item Assume that $\cC$ is such that $\FLie(\indcat{\cC})$ is closed under semidirect products. This holds for example if $\cC$ is a Frobenius exact pre-Tannakian category of moderate growth, thanks to  \Cref{cor:Fr-exact-semidirect-closed}. Since $\tttg(\rho, \ttd)$ is a quotient of $\FLie(\ttV\oplus \ttV^* )\rtimes_\rho \ttX$, it follows from \cite{Etingof}*{Lemma 7.6}  that $\tttg(\rho, \ttd)$ is a Lie algebra in $\indcat{\cC}$, rather than an operadic one.

\item There are natural maps $\ttX, \ttV, \ttV^*\to \tttg(\rho,\ttd)$; we will prove these are injective in \Cref{thm:gtilde-verp}.
\item The map $\ttV \to \tttg(\rho, \ttd)$ extends to a Lie algebra map $\iota_+\colon\FLie(\ttV)\to\tttg(\rho, \ttd)$.\footnote{This is not a direct consequence of the definition of $\FLie(\ttV)$ because $\tttg(\rho,\ttd)$ is an operadic Lie algebra.} Indeed, one constructs $\iota_+$ as the composition
\[\iota_{+}\colon\FLie(\ttV)\to \FLie(\ttV \oplus \ttV^*)\hookrightarrow \FLie(\ttV \oplus \ttV^*)\rtimes_\rho \ttX \twoheadrightarrow \tttg(\rho, \ttd).\]
Similarly, $\ttV^* \to \tttg(\rho, \ttd)$ extends to a Lie algebra map $\iota_-\colon\FLie(\ttV^*)\to\tttg(\rho, \ttd)$.
\item We denote by $\tttn_+$ and $\tttn_-$ the subalgebras of $\tttg(\rho, \ttd)$ generated by $\iota_+(\ttV)$ and $\iota_-(\ttV^*)$, resp.

\item There is a $\bZ$-grading on the ind-Lie algebra $\FLie(\ttV\oplus \ttV^* )\rtimes_\rho \ttX$ determined by
\begin{align*}
\deg \ttV^*&=-1, & \deg \ttX&=0, & \deg \ttV&=1.
\end{align*}
Since the defining ideal of $\tttg(\rho, \ttd)$ is homogeneous, we get a grading
\begin{align*}
\tttg(\rho, \ttd)&=\bigoplus_{k\in\bZ} \tttg_k(\rho, \ttd).
\end{align*}
\end{itemize}
\end{remark}

We obtain first a sort of Cartan involution for the algebra $\tttg(\rho, \ttd)$, which relies on the natural built-in duality of $\opD(\ttX)$ described in \Cref{rem:contragredient-dual-pair}.

\begin{lemma}\label{lem:Cartan-involution}
Let $(\rho, \ttd)\in \opD(\ttX)$. There exists a unique operadic Lie algebra isomorphism 
\begin{align*}
\widehat{\omega}\colon\FLie(\ttV\oplus \ttV^* )\rtimes_\rho \ttX\to \FLie(\ttV^*\oplus \ttV^{**} )\rtimes_{\rho^{\vee}} \ttX & &\text{with}  &&\restr{\widehat{\omega}}{\ttX}=\id_{\ttX}, 
&&\restr{\widehat{\omega}}{\ttV}=\vartheta_{\ttV}, &&\restr{\widehat{\omega}}{\ttV^*}=\id_{\ttV^*}.
\end{align*}
Moreover, $\widehat{\omega}$ descends to the quotient and induces an operadic Lie algebras isomorphism
\begin{align*}
\widetilde{\omega}\colon\tttg(\rho,\ttd)\to \tttg(\rho^\vee, \ttd^\vee).
\end{align*}
\end{lemma}

\begin{proof}
First, we prove the existence of $\widehat{\omega}$.
We apply \Cref{lem:semidirect-compatibility-free-algebra} for $W=\ttV\oplus\ttV^*$, $\fg=\ttX$, $\fh=\FLie(\ttV^*\oplus \ttV^{**} )\rtimes_{\rho^{\vee}} \ttX$. Here $\ff=\FLie(\ttV\oplus \ttV^* )$, we take $\phi_{\ff}\colon \ff\to \fh$ as the unique Lie algebra map whose restriction to $\ttV\oplus \ttV^*$ is $\vartheta_{\ttV}\sqcup\id_{\ttV^*}$, and $\phi_{\ttX}:\ttX\to\fh$ is the inclusion. Now the compatibility  \eqref{eq:morphism-from-semidirect-product-compatibility} holds trivially on $\ttX\ot\ttV^*$, and holds on $\ttX\ot\ttV$ since $\vartheta_{\ttV}$ is a map of $\ttX$-modules. Thus there exists a Lie algebra map $\widehat{\omega}$ as stated.

By the same argument, there is a map $\varpi\colon  \FLie(\ttV^*\oplus \ttV^{**} )\rtimes_{\rho^{\vee}} \ttX \to  \FLie(\ttV\oplus \ttV^* )\rtimes_\rho \ttX$ such that 
$\restr{\varpi}{\ttV^{**}}=\vartheta_{\ttV}^{-1}$, $\restr{\varpi}{\ttV^*}=\id_{\ttV^*}$, $\restr{\varpi}{\ttX}=\id_{\ttX}$. Since $\ttV^*$, $\ttX$ and $\ttV$ generate $\FLie(\ttV\oplus \ttV^* )\rtimes_\rho \ttX$, it is clear that $\varpi$ is the inverse of $\widehat{\omega}$, which is thus an isomorphism.
\smallbreak

The next step is to prove that $\widehat{\omega}$ factors through the quotient and hence yields the desired Lie algebra map $\widetilde{\omega}$. Let $\pi\colon\FLie(\ttV^*\oplus \ttV^{**} )\rtimes_{\rho^{\vee}} \ttX\twoheadrightarrow \tttg(\rho^{\vee}, \ttd^{\vee})$ be the projection. We show that $\pi\widehat{\omega}$ annihilates the image of $(\ttb_{\ttV \ot \ttV^*},-\ttd)$.
By definition of $\widehat{\omega}$ and since the bracket is antisymmetric, we have
$$ \widehat{\omega}(\ttb_{\ttV \ot \ttV^*},0)
= (\ttb_{\ttV^{**} \ot \ttV^*},0)(\vartheta_{\ttV}\ot\id_{\ttV^*}) 
= (\ttb_{\ttV^{*} \ot \ttV^{**}},0)(\id_{\ttV^*}\ot\vartheta_{\ttV})c_{\ttV,\ttV^*}. $$
From here, and using the definition of $\pi$ and $\ttd^{\vee}$ we obtain
\begin{align*}
\pi\widehat{\omega}(\ttb_{\ttV \ot \ttV^*},0) &= -\pi(\ttb_{\ttV^{*} \ot \ttV^{**}},0)(\id_{\ttV^*}\ot\vartheta_{\ttV})c_{\ttV,\ttV^*}
= -\pi(0,\ttd^{\vee})(\id_{\ttV^*}\ot\vartheta_{\ttV})c_{\ttV,\ttV^*}
= \pi\widehat{\omega}(0,\ttd);
\end{align*}
thus $\pi\widehat{\omega}(\ttb_{\ttV \ot \ttV^*},-\ttd)=0$, as desired, which implies the existence of $\widetilde{\omega}$.

Similarly, the inverse $\varpi$ for $\widehat{\omega}$ introduced above descends to the quotient thus producing an inverse for $\widetilde{\omega}$.
\end{proof}

Next we develop some structural results for $\tttg(\rho, \ttd)$. Namely, we want to obtain a triangular decomposition, prove that the subalgebras $\tttn_+$ and $\tttn_-$ from \Cref{rem:contragredient-gtilde} are the free Lie algebras on $\ttV$ and $\ttV^*$, respectively, and finally verify that there exists a maximal ideal trivially intersecting $\ttX$. Essentially, we follow the strategy from \cite{Kac-book}*{Theorem 1.2}.  In particular, our proofs of existence and several structural properties of contragredient Lie algebras will rely heavily on a family of modules.

\begin{lemma}\label{lem:gtilde-representation-on-tensor-algebra}
Let $(\rho, \ttd)\in \opD(\ttX)$ and $(\ttW,\lambda) \in \Mod_{\cC}(\ttX)$. Consider the following maps:
\begin{enumerate}[leftmargin=*,label=\rm{(\arabic*)}]
\item\label{item:Kac-module-0} $\eta_{\ttX}\colon \ttX\ot\op{T}(\ttV)\ot \ttW \to \op{T}(\ttV)\ot \ttW$ is the tensor product of the $\ttX$-modules $(\op{T}(\ttV), \rho)$ (obtained from $(\ttV,\rho)$ as in \Cref{rem:action-deriv-tensor-algebra}) and $(\ttW,\lambda)$.
\item\label{item:Kac-module-positive}  $\eta_{+}\colon \ttV\ot \op{T}(\ttV)\ot \ttW\to \op{T}(\ttV)\ot \ttW$ is the multiplication on $\op{T}(\ttV)$ tensored with $\id_{\ttW}$.
\item \label{item:Kac-module-negative} $\eta_{-}\colon \ttV^*\ot \op{T}(\ttV)\ot \ttW\to \op{T}(\ttV)\ot \ttW$ is defined recursively $\eta_{-}^{n}\colon \ttV^* \ot \ttV^{\ot n}\ot \ttW\to\ttV^{\ot (n-1)}\ot \ttW$ by
\begin{align*}
\eta_{-}^{0} &=0, & 
\eta_{-}^{n} &\coloneqq -\eta_{\ttX}(\ttd\circ c_{\ttV^*,\ttV}\ot \id_{\ttV^{\ot (n-1)}\ot \ttW}) +(\id_{\ttV}\ot\eta_{-}^{n-1}) (c_{\ttV^*,\ttV}\ot\id_{\ttV^{\ot (n-1)}\ot \ttW}) \quad n\ge 1.
\end{align*}
\end{enumerate}
Then these three maps extend uniquely to an action 
\[\eta_\rtimes \colon \FLie(\ttV\oplus \ttV^* )\rtimes_\rho \ttX \ot \op{T}(\ttV)\ot \ttW \to \op{T}(\ttV)\ot \ttW.\]
Moreover, $\eta_\rtimes$ descends to the quotient thus producing an action of $\tttg(\rho, \ttd)$ on $\op{T}(\ttV) \ot \ttW$.
\end{lemma}
\begin{proof} 
Note first that, in diagrammatic terms, the map $\eta_{-}$ is defined recursively by 
\begin{equation*}
\eta_{-}^{n}  \, \,= \, \, - \, \, 
\xy(0,0)*{
\begin{tikzpicture}[scale=.3]
\draw [very thick, <-] (0,0) to  [out=-90, looseness=1, in=90] (4,-4);
\draw [very thick, ->] (4,0) to  [out=-90, looseness=1, in=90] (0,-4);
\draw [very thick] (0,-4) to [out=-90, looseness=1, in=180] (2,-5);
\draw [very thick] (4,-4) to [out=-90, looseness=1, in=0] (2,-5);
\draw [very thick, blue] (2,-5) to [out=-90, looseness=1, in=180] (6,-7);
\draw [very thick, ->] (6,0) to (6,-6.25);
\draw [very thick, ->] (6,-7.75) to (6,-9);
\draw [very thick, ->] (10,0) to (10,-6.25);
\draw [very thick, ->] (10,-7.75) to (10,-9);
\draw [very thick, red] (12,0) to (12,-9);
\draw [fill=white] (5.25,-6.25) rectangle (12.75,-7.75);
\node at (9,-7) {$\eta_{\ttX}$};
\node at (8,-1) {\tiny $n-1$};
\node at (8,-2) {$\dots$};
\node at (8,-8.25) {\tiny $n-1$};
\node at (8,-9) {$\dots$};
\end{tikzpicture}
}; 
\endxy
\, \,+ \, \, 
\xy(0,0)*{
\begin{tikzpicture}[scale=.3]
\draw [very thick, <-] (0,0) to  [out=-90, looseness=1, in=90] (4,-4);
\draw [very thick] (4,0) to  [out=-90, looseness=1, in=90] (0,-4);
\draw [very thick, ->] (0,-4) to (0,-9);
\draw [very thick] (4,-4) to [out=-90, looseness=1, in=180] (6,-6);
\draw [very thick, ->] (6,0) to (6,-5.25);
\draw [very thick, ->] (10,0) to (10,-5.25);
\draw [very thick, red] (12,0) to (12,-9);
\draw [fill=white] (5.25,-5.25) rectangle (12.75,-6.85);
\draw [very thick, ->] (6.5,-6.85) to (6.5,-9);
\draw [very thick, ->] (9.5,-6.85) to (9.5,-9);
\node at (9,-6.2) {$\eta_{-}^{n-1}$};
\node at (8,-1) {\tiny $n-1$};
\node at (8,-2) {$\dots$};
\node at (8,-8.25) {\tiny $n-2$};
\node at (8,-9) {$\dots$};
\end{tikzpicture}
}; 
\endxy \, \, .
\end{equation*}
\begin{claim}\label{claim:Kac-module-FLie}
There is a unique action $\eta_{\op F}$ of $\FLie(\ttV\oplus \ttV^* )$ on $\op{T}(\ttV) \ot \ttW$ extending $\eta_{+}$ and $\eta_{-}$.
\end{claim}
Indeed, by \Cref{rem:modules-over-free-algebras}, the map $(\ttV\oplus \ttV^* )\ot \op{T}(\ttV)\ot \ttW\to \op{T}(\ttV)\ot \ttW$ defined as $\eta_+$ on $\ttV$ and as $\eta_{-}$ on $\ttV^*$ extends uniquely to an action $\eta_{\op F} \colon \FLie(\ttV\oplus \ttV^* )\ot \op{T}(\ttV)\ot \ttW\to \op{T}(\ttV)\ot \ttW$.

\begin{claim}\label{claim:Kac-module-semidirect}
The actions $\eta_{\ttX}$ from \eqref{item:Kac-module-0} and $\eta_{\op F}$ from \Cref{claim:Kac-module-FLie} extend uniquely to an action $\eta_\rtimes$ of $\FLie(\ttV\oplus \ttV^* )\rtimes_\rho \ttX$ on $\op{T}(\ttV)\ot \ttW$.
\end{claim}

By \Cref{prop:module-over-semidirect-product}, we just need to check that $\eta_{\op F}$ and $\eta_{\ttX}$ satisfy the compatibility  \eqref{eq:module-over-semidirect-product-compatibility}. That is, we must verify the following identity on $\Hom_{\cC}\big(\ttX\ot \FLie(\ttV\oplus \ttV^* )\ot \op{T}(\ttV)\ot \ttW,\op{T}(\ttV)\ot \ttW\big)$:
\begin{align}\label{eq:Kac-module-compatibility}
\eta_{\op F}(\rho\ot \id_{\op{T}(\ttV)\ot \ttW}) = \eta_{\ttX}(\id_{\ttX}\ot \eta_{\op F})-\eta_{\op F}(\id_{\FLie(\ttV\oplus \ttV^* )}\ot \eta_{\ttX})(c_{\ttX,\FLie(\ttV\oplus \ttV^* )}\ot \id_{\op{T}(\ttV)\ot \ttW}).
\end{align}
By \Cref{lem:semidirect-compatibility-free-algebra} \ref{item:semidirect-compatibility-free-algebra-module}, it suffices to check this equality on $\ttX\ot\ttV\ot \op{T}(\ttV)\ot \ttW$ and $\ttX\ot\ttV^*\ot \op{T}(\ttV)\ot \ttW$. 

To verify the identity on the first term, we restrict to $\ttX\ot\ttV\ot \ttV^{\ot n}\ot \ttW$, where it becomes
\begin{equation*}
\xy
(0,0)*{
\begin{tikzpicture}[scale=.3]
\draw [very thick, blue] (0,0) to [out=-90, looseness=1, in=180] (2,-4);
\draw [very thick, ->] (2,0) to (2,-6);
\draw [blue, fill=blue] (2,-4) circle (0.3cm and 0.3cm);
\draw [very thick, ->] (4,0) to (4,-6);
\draw [very thick, ->] (8,0) to (8,-6);
\node at (6,-1) {\tiny $n$};
\node at (6,-2) {$\dots$};
\draw [very thick, red] (10,0) to (10,-6);
\end{tikzpicture} 
};
\endxy 
\, \, = \, \, 
\xy
(0,0)*{
\begin{tikzpicture}[scale=.3]
\draw [very thick, blue] (0,0) to [out=-90, looseness=1, in=180] (1.25,-4);
\draw [very thick, ->] (2,0) to (2,-3.25);
\draw [very thick, ->] (2,-3.25) to (2,-6);
\draw [very thick, ->] (4,0) to (4,-3.25);
\draw [very thick, ->] (4,-4.75) to (4,-6);
\draw [very thick, ->] (8,0) to (8, -3.25);
\draw [very thick, ->] (8,-3.25) to (8,-6);
\node at (6,-1) {\tiny $n$};
\node at (6,-2) {$\dots$};
\draw [very thick, red] (10,0) to (10,-6);
\draw [fill=white] (1.25,-3.25) rectangle (10.75,-4.75);
\node at (6, -4) {\tiny $\eta_{\ttX}\vert_{\ttV^{n+1}}$};
\end{tikzpicture} 
};
\endxy 
\, \, - \, \, 
\xy
(0,0)*{
\begin{tikzpicture}[scale=.3]
\draw [very thick, blue] (0,0) to [out=-90, looseness=1, in=180] (4,-4);
\draw [very thick, ->] (2,0) to [out=-90, looseness=1, in=90] (0,-6);
\draw [very thick, ->] (4,0) to (4,-3.25);
\draw [very thick, ->] (4,-4.75) to (4,-6);
\draw [very thick, ->] (8,0) to (8, -3.25);
\draw [very thick, ->] (8,-3.25) to (8,-6);
\node at (6,-1) {\tiny $n$};
\node at (6,-2) {$\dots$};
\draw [very thick, red] (10,0) to (10,-6);
\draw [fill=white] (3.25,-3.25) rectangle (10.75,-4.75);
\node at (7, -4) {\tiny $\eta_{\ttX}\vert_{\ttV^{n}}$};
\end{tikzpicture} 
};
\endxy \, ,
\end{equation*}
which follows directly since $\eta_{\ttX}$ is defined by extending $\rho \colon \ttX \ot \ttV \to \ttV$ by derivations.

Now we verify recursively on $n\geq 0$ that \eqref{eq:Kac-module-compatibility} holds in $\ttX\ot\ttV^*\ot \ttV^{\ot n}\ot \ttW$.
\begin{itemize}[leftmargin=*]
\item For $n=0$, the three terms of \eqref{eq:Kac-module-compatibility} vanish. Indeed, $\eta_{-}(\rho\ot \id_{\ttV^{\ot 0}\ot \ttW})=\eta_{\ttX}(\id_{\ttX}\ot \eta_{-})=0$ by definition of $\eta_{-}$. 
Similarly, the remaining term is $0$ since $\eta_{\ttX}(\ttX\ot\ttV^{\ot 0}\ot \ttW)\subseteq \ttV^{\ot 0}\ot \ttW$.
\item We also verify the case $n=1$ to illustrate the diagrammatic computations involved in the proof of the inductive step. Here \eqref{eq:Kac-module-compatibility} becomes
\begin{equation*}
- \, \, 
\xy(0,0)*{
\begin{tikzpicture}[scale=.25]
\draw [very thick, blue] (-2,2) to [out=-90, looseness=1, in=180] (0,0);
\draw [very thick, <-] (0,2) to (0,0);
\draw [very thick] (0,0) to  [out=-90, looseness=1, in=90] (4,-4);
\draw [very thick] (4,2) to (4,0);
\draw [very thick, ->] (4,0) to  [out=-90, looseness=1, in=90] (0,-4);
\draw [very thick] (0,-4) to [out=-90, looseness=1, in=180] (2,-5);
\draw [very thick] (4,-4) to  [out=-90, looseness=1, in=0] (2,-5);
\draw [very thick, blue] (2,-5) to [out=-90, looseness=1, in=180] (6,-7);
\draw [very thick, red] (6,2) to (6,-9);
\draw [fill=white] (5.25,-6.5) rectangle (6.75,-7.5);
\node at (6,-7) {\tiny $\lambda$};
\draw [blue, fill=blue] (0,0) circle (0.3cm and 0.3cm);
\end{tikzpicture}
}; 
\endxy
\, \,= \, \, - \, \, 
\xy(0,0)*{
\begin{tikzpicture}[scale=.25]
\draw [very thick, blue] (-2,2) to [out=-90, looseness=1.5, in=180] (6,-8);
\draw [very thick, <-] (0,2) to (0,0);
\draw [very thick] (0,0) to  [out=-90, looseness=1, in=90] (4,-4);
\draw [very thick] (4,2) to (4,0);
\draw [very thick, ->] (4,0) to  [out=-90, looseness=1, in=90] (0,-4);
\draw [very thick] (0,-4) to [out=-90, looseness=1, in=180](2,-5);
\draw [very thick] (4,-4) to [out=-90, looseness=1, in=0] (2,-5);
\draw [very thick, blue] (2,-5) to [out=-90, looseness=1, in=180] (6,-6.5);
\draw [very thick, red] (6,2) to (6,-9);
\draw [fill=white] (5.25,-6) rectangle (6.75,-7);
\node at (6,-6.5) {\tiny $\lambda$};
\draw [fill=white] (5.25,-7.5) rectangle (6.75,-8.5);
\node at (6,-8) {\tiny $\lambda$};
\end{tikzpicture}
}; 
\endxy
\, \, +\, \, 
\xy(0,0)*{
\begin{tikzpicture}[scale=.25]
\draw [very thick, blue] (-2,2) to [out=-90, looseness=1, in=180] (4,0);
\draw [very thick, <-] (0,2) to (0,0);
\draw [very thick] (0,0) to  [out=-90, looseness=1, in=90] (4,-4);
\draw [very thick] (4,2) to (4,0);
\draw [very thick, ->] (4,0) to  [out=-90, looseness=1, in=90] (0,-4);
\draw [very thick] (0,-4) to  [out=-90, looseness=1, in=180] (2,-5);
\draw [very thick] (4,-4) to  [out=-90, looseness=1, in=0] (2,-5);
\draw [very thick, blue] (2,-5) to [out=-90, looseness=1, in=180] (6,-7);
\draw [very thick, red] (6,2) to (6,-9);
\draw [fill=white] (5.25,-6.5) rectangle (6.75,-7.5);
\node at (6,-7) {\tiny $\lambda$};
\draw [blue, fill=blue] (4,0) circle (0.3cm and 0.3cm);
\end{tikzpicture}
}; 
\endxy
\, \, +\, \, 
\xy(0,0)*{
\begin{tikzpicture}[scale=.25]
\draw [very thick, blue] (-2,2) to [out=-90, looseness=.75, in=180] (6,0);
\draw [very thick, <-] (0,2) to (0,0);
\draw [very thick] (0,0) to  [out=-90, looseness=1, in=90] (4,-4);
\draw [very thick] (4,2) to (4,0);
\draw [very thick, ->] (4,0) to  [out=-90, looseness=1, in=90] (0,-4);
\draw [very thick] (0,-4) to  [out=-90, looseness=1, in=180] (2,-5);
\draw [very thick] (4,-4) to  [out=-90, looseness=1, in=0] (2,-5);
\draw [very thick, blue] (2,-5) to [out=-90, looseness=1, in=180] (6,-7);
\draw [very thick, red] (6,2) to (6,-9);
\draw [fill=white] (5.25,-.5) rectangle (6.75,.5);
\node at (6,0) {\tiny $\lambda$};
\draw [fill=white] (5.25,-6.5) rectangle (6.75,-7.5);
\node at (6,-7) {\tiny $\lambda$};
\end{tikzpicture}
}; 
\endxy,
\end{equation*}
which follows from \eqref{eq:contragradient-compatibility-module} by rearranging the terms and  pre-composing them with $c_{\ttV^*, \ttV}$.

\item For $n>1$ we express diagrammatically all the terms of \eqref{eq:Kac-module-compatibility} restricted to $\ttX\ot\ttV^*\ot \ttV^{\ot n}\ot \ttW$:
\begin{align*}
\eta^n_{-}(\rho \ot \id) &=-
\xy(0,0)*{
\begin{tikzpicture}[scale=.25]
\draw [very thick, blue] (-2,0) to [out=-90, looseness=1, in=180] (0,-3);
\draw [very thick, <-] (0,0) to (0,-3);
\draw [very thick] (0,-3) to  [out=-90, looseness=1, in=90] (4,-6);
\draw [very thick] (4,0) to (4,-3);
\draw [very thick, ->] (4,-3) to  [out=-90, looseness=1, in=90] (0,-6);
\draw [very thick] (0,-6) to [out=-90, looseness=1, in=180] (2,-7);
\draw [very thick] (4,-6) to [out=-90, looseness=1, in=0] (2,-7);
\draw [very thick, blue] (2,-7) to [out=-90, looseness=1, in=180] (5.25,-8.75);
\draw [very thick, ->] (6,0) to (6,-8);
\draw [very thick, ->] (6,-9.5) to (6,-10.5);
\draw [very thick, ->] (10,0) to (10,-8);
\draw [very thick, ->] (10,-9.5) to (10,-10.5);
\draw [very thick, red] (12,0) to (12,-10.5);
\node at (8,-.5) {\tiny $n-1$};
\node at (8,-1.5) {$\dots$};
\draw [fill=white] (5.25,-8) rectangle (12.75,-9.5);
\node at (9,-8.75) {$\eta_{\ttX}$};
\draw [blue, fill=blue] (0,-3) circle (0.3cm and 0.3cm);
\end{tikzpicture}
};
\endxy 
+ \xy(0,0)*{
\begin{tikzpicture}[scale=.25]
\draw [very thick, blue] (-2,0) to [out=-90, looseness=1, in=180] (0,-3);
\draw [very thick, <-] (0,0) to (0,-3);
\draw [very thick] (0,-3) to  [out=-90, looseness=1, in=90] (4,-6);
\draw [very thick] (4,-6) to  [out=-90, looseness=1, in=180] (5.25,-8);
\draw [very thick] (4,0) to (4,-3);
\draw [very thick] (4,-3) to  [out=-90, looseness=1, in=90] (0,-6);
\draw [very thick, ->] (0,-6) to (0,-10.5);
\draw [very thick, ->] (6,0) to (6,-7.25);
\draw [very thick, ->] (6.5,-8.75) to (6.5,-10.5);
\draw [very thick, ->] (10,0) to (10,-7.25);
\draw [very thick, ->] (9.5,-8.75) to (9.5,-10.5);
\draw [very thick, red] (12,0) to (12,-10.5);
\node at (8,-.5) {\tiny $n-1$};
\node at (8,-1.5) {$\dots$};
\node at (8,-9.25) {\tiny $n-2$};
\draw [fill=white] (5.25,-7.25) rectangle (12.75,-8.75);
\node at (9,-8) {\tiny$\eta_{-}^{n-1}$};
\draw [blue, fill=blue] (0,-3) circle (0.3cm and 0.3cm);
\end{tikzpicture}
}; 
\endxy; \\ \, \\
\eta_{\ttX}(\id\ot \eta^n_{-})&=- 
\xy(0,0)*{
\begin{tikzpicture}[scale=.25]
\draw [very thick, blue] (-2,0) to [out=-90, looseness=1.5, in=180]  (5.25,-9.75);
\draw [very thick, <-] (0,0) to (0,-1);
\draw [very thick] (0,-1) to  [out=-90, looseness=1, in=90] (4,-4);
\draw [very thick] (4,0) to (4,-1);
\draw [very thick, ->] (4,-1) to  [out=-90, looseness=1, in=90] (0,-4);
\draw [very thick] (0,-4) to [out=-90, looseness=1, in=180] (2,-5);
\draw [very thick] (4,-4) to [out=-90, looseness=1, in=0] (2,-5);
\draw [very thick, blue] (2,-5) to [out=-90, looseness=1, in=180] (6,-7);
\draw [very thick, ->] (6,0) to (6,-6.25);
\draw [very thick, ->] (6,-6 .25) to (6,-11.5);
\draw [very thick, ->] (10,0) to (10,-6.25);
\draw [very thick, ->] (10,-6.25) to (10,-11.5);
\draw [very thick, red] (12,0) to (12,-11.5);
\draw [fill=white] (5.25,-6.25) rectangle (12.75,-7.75);
\node at (9,-7) {$\eta_{\ttX}$};
\node at (8,-.5) {\tiny $n-1$};
\node at (8,-1.5) {$\dots$};
\draw [fill=white] (5.25,-9) rectangle (12.75,-10.5);
\node at (9,-9.75) {$\eta_{\ttX}$};
\end{tikzpicture}
}; 
\endxy
+ 
\xy(0,0)*{
\begin{tikzpicture}[scale=.25]
\draw [very thick, blue] (-2,0) to [out=-90, looseness=.75, in=180]  (0,-7);
\draw [very thick, <-] (0,0) to  [out=-90, looseness=1, in=90] (4,-4);
\draw [very thick] (4,0) to  [out=-90, looseness=1, in=90] (0,-4);
\draw [very thick, ->] (0,-4) to (0,-11.5);
\draw [blue, fill=blue] (0,-7) circle (0.3cm and 0.3cm);
\draw [very thick] (4,-4) to [out=-90, looseness=1, in=180] (5.25,-7);
\draw [very thick, ->] (6,0) to (6,-6.25);
\draw [very thick, ->] (10,0) to (10,-6.25);
\draw [very thick, red] (12,0) to (12,-11.5);
\draw [fill=white] (5.25,-6.25) rectangle (12.75,-7.75);
\draw [very thick, ->] (6.5,-7.75) to (6.5,-11.5);
\draw [very thick, ->] (9.5,-7.75) to (9.5,-11.5);
\node at (9,-7) {\tiny$\eta_{-}^{n-1}$};
\node at (8,-.5) {\tiny $n-1$};
\node at (8,-1.5) {$\dots$};
\node at (8,-8.25) {\tiny $n-2$};
\end{tikzpicture}
}; 
\endxy 
+ 
\xy(0,0)*{
\begin{tikzpicture}[scale=.25]
\draw [very thick, blue] (-2,0) to [out=-90, looseness=1.5, in=180]  (5.25,-9.75);
\draw [very thick, <-] (0,0) to  [out=-90, looseness=1, in=90] (4,-4);
\draw [very thick] (4,0) to  [out=-90, looseness=1, in=90] (0,-4);
\draw [very thick, ->] (0,-4) to (0,-11.5);
\draw [very thick] (4,-4) to [out=-90, looseness=1, in=180] (5.25,-7);
\draw [very thick, ->] (6,0) to (6,-6.25);
\draw [very thick, ->] (10,0) to (10,-6.25);
\draw [very thick, red] (12,0) to (12,-11.5);
\draw [fill=white] (5.25,-6.25) rectangle (12.75,-7.75);
\draw [very thick, ->] (6.5,-7.75) to (6.5,-11.5);
\draw [very thick, ->] (9.5,-7.75) to (9.5,-11.5);
\node at (9,-7) {\tiny$\eta_{-}^{n-1}$};
\node at (8,-.5) {\tiny $n-1$};
\node at (8,-1.5) {$\dots$};
\node at (8,-8.25) {\tiny $n-2$};
\draw [fill=white] (5.25,-9) rectangle (12.75,-10.5);
\node at (9,-9.75) {$\eta_{\ttX}$};
\end{tikzpicture}
}; 
\endxy;
\end{align*}
\begin{align*}
-\eta_{-}^n(\id\ot \eta_{\ttX})(c\ot \id)=& 
\xy(0,0)*{
\begin{tikzpicture}[scale=.25]
\draw [very thick, blue] (-2,0) to [out=-90, looseness=1, in=180] (4,-3);
\draw [very thick, <-] (0,0) to (0,-3);
\draw [very thick] (0,-3) to  [out=-90, looseness=1, in=90] (4,-6);
\draw [very thick] (4,0) to (4,-3);
\draw [very thick, ->] (4,-3) to  [out=-90, looseness=1, in=90] (0,-6);
\draw [very thick] (0,-6) to [out=-90, looseness=1, in=180] (2,-7);
\draw [very thick] (4,-6) to [out=-90, looseness=1, in=0] (2,-7);
\draw [very thick, blue] (2,-7) to [out=-90, looseness=1, in=180] (5.25,-8.75);
\draw [very thick, ->] (6,0) to (6,-8);
\draw [very thick, ->] (6,-9.5) to (6,-10.5);
\draw [very thick, ->] (10,0) to (10,-8);
\draw [very thick, ->] (10,-9.5) to (10,-10.5);
\draw [very thick, red] (12,0) to (12,-10.5);
\node at (8,-.5) {\tiny $n-1$};
\node at (8,-1.5) {$\dots$};
\draw [fill=white] (5.25,-8) rectangle (12.75,-9.5);
\node at (9,-8.75) {$\eta_{\ttX}$};
\draw [blue, fill=blue] (4,-3) circle (0.3cm and 0.3cm);
\end{tikzpicture}
}; 
\endxy
- \xy(0,0)*{
\begin{tikzpicture}[scale=.25]
\draw [very thick, blue] (-2,0) to [out=-90, looseness=1, in=180] (4,-3);
\draw [very thick, <-] (0,0) to (0,-3);
\draw [very thick] (0,-3) to  [out=-90, looseness=1, in=90] (4,-6);
\draw [very thick] (4,-6) to  [out=-90, looseness=1, in=180] (5.25,-8);
\draw [very thick] (4,0) to (4,-3);
\draw [very thick] (4,-3) to  [out=-90, looseness=1, in=90] (0,-6);
\draw [very thick, ->] (0,-6) to (0,-10.5);
\draw [very thick, ->] (6,0) to (6,-7.25);
\draw [very thick, ->] (6.5,-8.75) to (6.5,-10.5);
\draw [very thick, ->] (10,0) to (10,-7.25);
\draw [very thick, ->] (9.5,-8.75) to (9.5,-10.5);
\draw [very thick, red] (12,0) to (12,-10.5);
\node at (8,-.5) {\tiny $n-1$};
\node at (8,-1.5) {$\dots$};
\node at (8,-9.25) {\tiny $n-2$};
\draw [fill=white] (5.25,-7.25) rectangle (12.75,-8.75);
\node at (9,-8) {\tiny$\eta_{-}^{n-1}$};
\draw [blue, fill=blue] (4,-3) circle (0.3cm and 0.3cm);
\end{tikzpicture}
}; 
\endxy
\\ \, \\
&\qquad \qquad \qquad \quad+
\xy(0,0)*{
\begin{tikzpicture}[scale=.25]
\draw [very thick, blue] (-2,0) to [out=-90, looseness=1, in=180] (5.25,-3);
\draw [very thick, <-] (0,0) to (0,-3);
\draw [very thick] (0,-3) to  [out=-90, looseness=1, in=90] (4,-6);
\draw [very thick] (4,0) to (4,-3);
\draw [very thick, ->] (4,-3) to  [out=-90, looseness=1, in=90] (0,-6);
\draw [very thick] (0,-6) to [out=-90, looseness=1, in=180] (2,-7);
\draw [very thick] (4,-6) to [out=-90, looseness=1, in=0] (2,-7);
\draw [very thick, blue] (2,-7) to [out=-90, looseness=1, in=180] (6,-9);
\draw [very thick, ->] (6,0) to (6,-2.25);
\draw [very thick, ->] (6,0) to (6,-8);
\draw [very thick, ->] (6,0) to (6,-10.5);
\draw [very thick, ->] (10,0) to (10,-2.25);
\draw [very thick, ->] (10,0) to (10,-8);
\draw [very thick, ->] (10,0) to (10,-10.5);
\draw [very thick, red] (12,0) to (12,-10.5);
\draw [fill=white] (5.25,-2.25) rectangle (12.75,-3.75);
\node at (9,-3) {$\eta_{\ttX}$};
\node at (8,-.5) {\tiny $n-1$};
\node at (8,-1.5) {$\dots$};
\draw [fill=white] (5.25,-8) rectangle (12.75,-9.5);
\node at (9,-8.75) {$\eta_{\ttX}$};
\end{tikzpicture}
}; 
\endxy
- \xy(0,0)*{
\begin{tikzpicture}[scale=.25]
\draw [very thick, blue] (-2,0) to [out=-90, looseness=1, in=180] (5.25,-3);
\draw [very thick, <-] (0,0) to (0,-3);
\draw [very thick] (0,-3) to  [out=-90, looseness=1, in=90] (4,-6);
\draw [very thick] (4,-6) to  [out=-90, looseness=1, in=180] (5.25,-8);
\draw [very thick] (4,0) to (4,-3);
\draw [very thick] (4,-3) to  [out=-90, looseness=1, in=90] (0,-6);
\draw [very thick, ->] (0,-6) to (0,-10.5);
\draw [very thick, ->] (6,0) to (6,-2.25);
\draw [very thick, ->] (6,0) to (6,-7.25);
\draw [very thick, ->] (10,0) to (10,-2.25);
\draw [very thick, ->] (10,0) to (10,-7.25);
\draw [very thick, ->] (6.5,-8.75) to (6.5,-10.5);
\draw [very thick, ->] (9.5,-8.75) to (9.5,-10.5);
\draw [very thick, red] (12,0) to (12,-10.5);
\draw [fill=white] (5.25,-2.25) rectangle (12.75,-3.75);
\node at (9,-3) {$\eta_{\ttX}$};
\node at (8,-.5) {\tiny $n-1$};
\node at (8,-1.5) {$\dots$};
\node at (8,-9.25) {\tiny $n-2$};
\draw [fill=white] (5.25,-7.25) rectangle (12.75,-8.75);
\node at (9,-8) {\tiny$\eta_{-}^{n-1}$};
\end{tikzpicture}
}; 
\endxy.
\end{align*}
Now we compute the sum $\eta_{\ttX}(\id\ot \eta^n_{-})-\eta_{-}^n(\id\ot \eta_{\ttX})(c\ot \id)$. Notice that the second summands in each of these two terms cancel out mutually. Also, as illustrated in the case $n=1$, applying \eqref{eq:contragradient-compatibility-module} to the module $\eta_{\ttX}$, the first summand of $\eta_{\ttX}(\id\ot \eta^n_{-})$ together with the first and third summands of $ -\eta_{-}^n(\id\ot \eta_{\ttX})(c\ot \id)$ add up exactly to the first summand of $\eta_{-}^{n}$. Lastly, by induction, the third summand of $\eta_{\ttX}(\id\ot \eta^n_{-})$ together with the fourth summand of $-\eta_{-}^n(\id\ot \eta_{\ttX})(c\ot \id)$ give the second summand of $\eta_{-}^{n}$. Thus \eqref{eq:Kac-module-compatibility} holds also over $\ttX\ot\ttV^*\ot \ttV^{\ot n}\ot \ttW$.
\end{itemize}
This finishes the proof of \Cref{claim:Kac-module-semidirect}.

\begin{claim}\label{claim:Kac-module-gtilde} 
The map $\eta_\rtimes$ from  \Cref{claim:Kac-module-semidirect} factors through the quotient, thus defining an action of $\tttg(\rho, \ttd)$ on  $\op{T}(\ttV)\ot \ttW$.
\end{claim}

Indeed, by definition of $\tttg(\ttX, \ttV)$, we just need to verify that
$$ \eta_\rtimes \big(\ima(\ttb_{\ttV \ot \ttV^*},-\ttd)\ot \op{T}(\ttV)\ot \ttW \big)=0 .$$
To do so, it is enough to verify that, for each $n\ge0$, the following diagram commutes 
\[\begin{tikzcd}
\ttV \ot \ttV^*\ot \ttV^{\ot n}\ot \ttW
\arrow{rrr}{\ttb_{\ttV \ot \ttV^*}\ot \ttV^{\ot n}\ot\ttW} 
\arrow{d}[swap]{\ttd\ot \ttV^{\ot n}\ot \ttW} 
&&&
\FLie(\ttV\oplus\ttV^*) \ot  \ttV^{\ot n}\ot \ttW
\arrow{d}{\eta_{\op F}}
\\
\ttX \ot \ttV^{\ot n}\ot \ttW
\arrow{rrr}{\eta_{\ttX}}
&&& \ttV^{\ot n} \ot \ttW.
\end{tikzcd}\] 
The proof uses that $\eta_{\op F}$ gives a module structure together with the recursive definition of $\eta_{-}^n$:
\begin{align*}
\eta_{\op F} & \big(\ttb_{\ttV \ot \ttV^*}\ot\id_{\ttV^{\ot n}\ot\ttW}\big) 
= \eta_{+}(\id_{\ttV}\ot \eta_{-}^n)
- \eta_{-}^{n+1}(\id_{\ttV^*}\ot \eta_{+})(c_{\ttV,\ttV^*}\ot\id_{\ttW}) 
\\
& = (\id_{\ttV}\ot \eta_{-}^n)  -\big( -\eta_{\ttX}(\ttd\circ c_{\ttV^*,\ttV}\ot\id_{\ttV^{\ot n}\ot\ttW})
+ (\id_{\ttV}\ot\eta_{-}^n)(c_{\ttV^*,\ttV}\ot\id_{\ttV^{\ot n}\ot\ttW})\big) 
(c_{\ttV,\ttV^*}\ot\id_{\ttV^{\ot n}\ot\ttW})
\\
& = \eta_{\ttX}\big( \ttd\ot\id_{\ttV^{\ot n}\ot\ttW} \big),
\end{align*}
and the statement follows.
\end{proof}

In order to obtain a triangular decomposition as in the classical case, we need to impose a restriction on the Lie algebra $\ttX$. As will become clear later, this algebra plays a role similar to that of a Cartan subalgebra. Thus, for the examples of interest at the moment, it will make no harm to assume that it belong to some well-behaved class of Lie algebras.

\begin{definition}\label{def:enough-modules}
We say that $\ttX\in \Lie(\cC)$ has \emph{enough modules} if for every simple subobject $S\subseteq \ttX$ there exists an $\ttX$-module $\lambda\colon\ttX\ot \ttW\to\ttW$ such that $\restr{\lambda}{S\ot\ttW}\ne 0$.
\end{definition}

\begin{example}\label{exa:enough-modules-classical}
In the classical case over $\cC=\sVec$, see \Cref{exa:classical-semidirect}, one takes $\ttX=\fh$, a purely even abelian Lie algebra, which clearly has enough modules since each element of $\fh^*$ defines one.
\end{example}

\begin{example}\label{exa:enough-modules-simple}
If $\ttX$ is a simple Lie algebra (i.e., non-abelian and contains no proper ideals), then it has enough modules. In fact, for any simple object $S\subseteq \ttX$ the adjoint action $\ad \colon \ttX\ot\ttX\to\ttX$ satisfies $\restr{\ad}{\ttX\ot S}\ne 0$.
\end{example}

\begin{example}\label{exa:enough-modules-gl-verp}
More interesting examples can be constructed over $\cC=\Verp$ for $p\ge 5$. Namely, take a simple $\ttL_i$ in $\Verp$ with  $i\ne 1, p-1$, and consider the general Lie algebra $\fgl(\ttL_i)$, see \Cref{rem:module-representation}. By \cite{Ven-GLVerp}*{Proposition 3.8, Theorem 4.4}, we have  a Lie algebra decomposition $\fgl(\ttL_i)=\one \oplus \fsl(\ttL_i)$ where $\one$ is central and $\fsl(\ttL_i)$ is a simple Lie algebra. Thus, by the previous examples, $\fgl(\ttL_i)$ has enough modules. 
\end{example}

\begin{example}\label{exa:enough-modules-gl-semsimplification}
Since $\Verp$ contains $\sVec$, we can merge \Cref{exa:enough-modules-classical,exa:enough-modules-gl-verp} and take the direct sum Lie algebra $\ttX=\fh\oplus \fg(\ttL_i) \in \Lie(\Verp)$, which also has enough modules.
\end{example}

Now we obtain some fundamental structural results for the auxiliary algebra $\tttg$. Recall the subalgebras $\tttn_{\pm}\subset \tttg$ and the grading $\tttg= \mathop{\oplus}\limits_{k\in \bZ}\tttg_k$ from \Cref{rem:contragredient-gtilde}.

\begin{theorem}\label{thm:gtilde-verp}
Let $(\ttX, \rho, \ttd)$ as in \Cref{def:contragredient-data}. Assume that $\ttX$ has enough modules in the sense of \Cref{def:enough-modules}. Then, for the operadic Lie algebra $\tttg\coloneq\tttg(\rho, \ttd)$, the following hold:
\begin{enumerate}[leftmargin=*,label=\rm{(\alph*)}]
\item \label{item:gtilde-cartan-inclusion} The natural map $\iota_0\colon \ttX\to \tttg$ is an inclusion.
\item \label{item:gtilde-verp-free} The Lie subalgebras $\tttn_+$ and $\tttn_-$ are free on $\ttV$ and $\ttV^*$, respectively. Moreover, the natural maps $\ttV, \ttV^* \to \tttg$ induce Lie algebras isomorphisms $\FLie(\ttV)\cong \ttn_+$ and $\FLie(\ttV^*)\cong\ttn_-$.\footnote{For this part we do not need to assume that $\ttX$ has enough modules.}
\item \label{item:gtilde-verp-triangular} We have $\tttg=\tttn_-\oplus \ttX \oplus\tttn_+$ as objects of $\indcat{\cC}$. Moreover,
\begin{align}\label{eq:gtilde-verp-triangular}
\tttn_- &= \bigoplus_{k<0} \tttg_k, & \ttX &= \tttg_0, & \tttn_+ &= \bigoplus_{k>0} \tttg_k. 
\end{align} 
\item \label{item:gtilde-verp-maximal} Among all $\bZ$-homogeneous ideals of $\tttg$ trivially intersecting $\ttX$, there exists a unique maximal one $\ttm$. Moreover, $\ttm=\ttm_+\oplus \ttm_-$, where $\ttm_{\pm}=\ttm\cap\tttn_{\pm}$.
\end{enumerate}
\end{theorem}

\begin{proof}
We begin with \ref{item:gtilde-cartan-inclusion}. Assume for a contradiction that $\ker \iota_0 \ne 0$, and take a simple subobject $S$ in this kernel. Since $\ttX$ has enough modules, there is some module $\lambda\colon \ttX\ot\ttW \to \ttW$ such that $\restr{\lambda}{S\ot\ttW}\ne 0$. By  \Cref{lem:gtilde-representation-on-tensor-algebra} this $\lambda$ induces an action $\eta\colon \tttg \ot \op{T}(\ttV)\ot\ttW \to  \op{T}(\ttV)\ot\ttW$. By construction of $\eta$, if $u\colon \one \to \op{T}(\ttV)$ denotes the unit map, we get a commutative diagram
\[\begin{tikzcd}
\ttX \ot \ttW
\arrow{rrr}{\iota_0\ot u \ot \ttW} 
&&&
 \tttg\ot \op{T}(\ttV)\ot\ttW
\arrow{d}{\eta}
\\
S\ot\ttW
\arrow[hook]{u}
\arrow{rrr}{u \ot \lambda_{\vert S\ot\ttW}}
&&& \op{T}(\ttV)\ot\ttW.
\end{tikzcd}\] 
The up-right-down path of this diagram is zero, but the bottom arrow is not. Thus $\iota_0$ is injective.
\medbreak
For \ref{item:gtilde-verp-free}, we show that the Lie algebra map $\iota_+\colon\FLie(\ttV)\to\tttg$ extending  $\ttV\to\tttg$ (see \Cref{rem:contragredient-gtilde}) is injective. Denote by $\nu_\one$ the $\tttg$-action on $\op{T}(\ttV)$ obtained by  taking the trivial $\ttX$-module $\ttW=\one$ in Lemma \ref{lem:gtilde-representation-on-tensor-algebra}. The pull back of $\nu_\one$ along $\iota_+$ produces an $\FLie(\ttV)$-action on $\op{T}(\ttV)$, which happens to coincide with the adjoint action (because they agree on $\ttV$, which generates). Thus the composition
\begin{equation*}
\begin{tikzcd}
\FLie(\ttV) \arrow{r}{\iota_+} & \tttg \simeq \tttg\otimes \one \arrow[hookrightarrow]{r} &  \tttg\otimes \op{T}(\ttV) \arrow{r}{\nu_\one} & \op{T}(\ttV)
\end{tikzcd}
\end{equation*}
is just the inclusion $\FLie(\ttV) \hookrightarrow\op{T}(\ttV)$. Thence $\iota_+$ is injective. Since the image of $\iota_+$ is clearly $\tttn_+$, it induces a canonical isomorphism $\FLie(\ttV)\cong \tttn_+$.

The same argument applied to $\tttg(\rho^{\vee}, \ttd^{\vee})$ shows that the subalgebra generated by $\ttV^*$ (the \emph{positive part} of $\tttg(\rho^{\vee}, \ttd^{\vee})$) is free on $\ttV^*$. The isomorphism $\widetilde{\omega}\colon\tttg(\rho, \ttd)\to \tttg(\rho^{\vee}, \ttd^{\vee})$ from \Cref{lem:Cartan-involution} maps $\tttn_-$ to the subalgebra generated by $\ttV^*$ in $\tttg(\rho^{\vee}, \ttd^{\vee})$, so $\tttn_-$ is free on $\ttV^*$. Moreover, by naturality, the isomorphism $\FLie(\ttV^*)\cong \tttn_-$ is provided by the natural map $\iota_-\colon \FLie(\ttV^*)\to\tttg$ from \Cref{rem:contragredient-gtilde}.

\medbreak

Next we deal with \ref{item:gtilde-verp-triangular}. We have to prove that the map 
\begin{align}\label{eq:gtilde-triangular-map}
\iota\coloneq\iota_-\sqcup\iota_0\sqcup\iota_+\colon\FLie(\ttV^*)\oplus\ttX\oplus\FLie(\ttV)\to \tttg
\end{align}
is an isomorphism of ind-objects. We show that it is injective and surjective. 
\begin{claim}\label{claim:gtilde-triangular-injective}
The map $\iota$ from \eqref{eq:gtilde-triangular-map} is injective. 
\end{claim}

Indeed, take $N_-\oplus N_0\oplus N_+\subseteq \ker\iota$, where $N_-\subseteq \FLie(\ttV^*)$, $N_0\subseteq \ttX$ and $N_+\subseteq \FLie(\ttV)$. As before, consider the action $\nu_\one\colon \tttg \ot \op{T}(\ttV) \to \op{T}(\ttV)$, and let $\mathtt{F}$ denote the composition  
\begin{equation*}
\begin{tikzcd}
\mathtt{F}\colon\FLie(\ttV^*)\oplus\ttX\oplus\FLie(\ttV) \arrow{r}{\iota} & \tttg \simeq \tttg\otimes \one \arrow[hookrightarrow]{r} &  \tttg\otimes \op{T}(\ttV) \arrow{r}{\nu_\one} & \op{T}(\ttV).
\end{tikzcd}
\end{equation*}
By construction of $\nu_\one$ (see \Cref{lem:gtilde-representation-on-tensor-algebra}), we know that, when restricted to $\one\subset\op{T}(\ttV)$, the action of $\ttV^*$ is trivial while $\ttV$ acts by multiplication; thus
\[0=\mathtt{F}(N_-\oplus N_0\oplus N_+)=\nu_\one(\iota_0(N_0)\ot\one)\oplus \iota_+(N_+).\]
So we get $N_+=0$. Consider now an arbitrary $\ttX$-module $(\ttW,\lambda)$, and let $\eta$ denote the associated $\tttg$-action on $\op{T}(\ttV)\ot\ttW$ provided by \Cref{lem:gtilde-representation-on-tensor-algebra}. Then, by construction, 
\begin{align*}
\eta\big(\iota(N_- \oplus N_0) \ot\ttV^0\ot\ttW \big)=\eta_{\ttX}(\iota_0(N_0)\ot\one \ot\ttW)=\one \ot \lambda(N_0\otimes \ttW).
\end{align*}
As $\iota(N_-\oplus N_0)=0$, we get $\lambda(N_0\otimes \ttW)=0$ for an arbitrary $\ttX$-module $(\lambda,\ttW)$. But we are assuming that $\ttX$ has enough modules, so $N_0=0$. Finally $0=\iota(N_-)=\iota_-(N_-)$, so $N_-=0$. This finishes the proof of \Cref{claim:gtilde-triangular-injective}.

\begin{claim}\label{claim:gtilde-triangular-surjective}
The image of the map $\iota$ from \eqref{eq:gtilde-triangular-map} is a Lie subalgebra. Thus $\iota$ is surjective.
\end{claim}
Let $\ttb\colon \tttg \ot \tttg \to \tttg$ denote the Lie bracket. We show that, when restricted to a tensor product of any two terms among $\ima \iota_{-}$, $\ima \iota_{0}$, and $\ima \iota_{+}$, the image of $\ttb$ lies on $\ima \iota$.
\begin{itemize}[leftmargin=*]\renewcommand{\labelitemi}{$\diamond$}
\item For any $\star\in\{-,0,+\}$, since $\iota_{\star}$ is a Lie algebra map, we have $\ttb(\ima \iota_{\star}\ot \ima \iota_{\star})\subseteq \ima \iota_{\star}$.
\item Also, $\ttb(\ima \iota_{0}\ot \ima \iota_{\pm})\subseteq \ima \iota_{\pm}$, since $\tttg$ is a quotient of $\FLie(\ttV\oplus \ttV^* )\rtimes_\rho \ttX$.
\item The only delicate case is $\ttb(\ima\iota_{-}\otimes\ima\iota_{+}) \subseteq\ima\iota$. We know from part \ref{item:gtilde-verp-free} that $\iota_\pm$ canonically identify $\tttn_\pm$ with the corresponding free Lie algebras, thus we get decompositions $\ima\iota_{\pm}=\oplus_{n\ge 1}\ima\iota_{\pm,n}$. More precisely, we get such decomposition where 
\begin{align*}
\ima\iota_{-,1}&=\ttV^*, & \ima\iota_{+,1}&=\ttV, & \ima\iota_{-,n+1}&=\ttb(\ttV^*\otimes\ima\iota_{+,n}), & \ima\iota_{+,n+1}&=\ttb(\ttV\otimes\ima\iota_{+,n}).
\end{align*}
We check recursively on $m,n\ge 1$ that $\ttb(\ima\iota_{-,m}\otimes\ima\iota_{+,n})\subseteq \ima \iota$. For $m=n=1$,
\begin{align*}
\ttb(\ima\iota_{-,1}\otimes\ima\iota_{+,1}) &=\ttb(\ttV^*\otimes\ttV)=\ttd(\ttV^*\otimes\ttV)\subseteq\ttX\subset \ima\iota.
\end{align*}
Then, using the Jacobi identity we check that
\begin{align}\label{eq:induction-bracket-contragredient}
\ttb(\ima\iota_{-,n+1}\otimes\ima\iota_{+,1}) &\subseteq\ima\iota_{-,n}, &
\ttb(\ima\iota_{-,1}\otimes\ima\iota_{+,n+1}) &\subseteq\ima\iota_{+,n}, & &\text{for all }n\ge 0.
\end{align}
A similar argument shows the general case. Indeed, assume that $\ttb(\ima\iota_{-,r}\otimes\ima\iota_{+,s}) \subseteq\ima\iota$ when $r+s\le k$ and fix $m,n\ge 1$ such that $m+n=k$. Using the Jacobi identity,
\begin{align*}
\ttb (\iota_{-,m+1} &\otimes\iota_{+,n}) = \ttb (\ttb \otimes \id_{\ima\iota_{+,n}})\big(\id_{\ttV^*}\otimes\iota_{-,m}\otimes\iota_{+,n}\big)
\\
& = \ttb ( \id_{\ttV^*}\otimes\ttb)\big(\id_{\ttV^*}\otimes\iota_{-,m}\otimes\iota_{+,n}\big)
+\ttb ( \id_{\ttV^*}\otimes\ttb)(c\otimes\id_{\ima\iota_{+,n}})\big(\id_{\ttV^*}\otimes\iota_{-,m}\otimes\iota_{+,n}\big).
\end{align*}
From \eqref{eq:induction-bracket-contragredient} and the inductive hypothesis we conclude that $\ttb(\ima\iota_{-,m+1}\otimes\ima\iota_{+,n}) \subseteq\ima\iota$. Analogously, $\ttb(\ima\iota_{-,m}\otimes\ima\iota_{+,n+1}) \subseteq\ima\iota$.
\end{itemize}
Thus, $\ttb(\ima \iota\ot \ima \iota)\subseteq \ima \iota$, i.e. $\ima \iota$ is a subalgebra. Since $\ttV$, $\ttV^*$ and $\ttX$ are contained in $\ima\iota$ and they generate $\tttg$ as Lie algebra, we conclude that $\ima\iota=\tttg$, which proves \Cref{claim:gtilde-triangular-surjective}.

\smallbreak
To verify \eqref{eq:gtilde-verp-triangular}, note that the inclusions $\tttn_- \subseteq \oplus_{k<0} \tttg_k$ and $\tttn_+ \subseteq \oplus_{k>0} \tttg_k$ follow since $\ttV^*\subseteq \tttg_{-1}$, $\ttV\subseteq \tttg_{1}$ and both $\oplus_{k<0} \tttg_k$ and $ \oplus_{k>0} \tttg_k$ are Lie subalgebras of $\tttg$, while $\ttX \subseteq \tttg_0$ by definition.
Thus \eqref{eq:gtilde-verp-triangular} holds since we have proved that  $\tttg=\tttn_-\oplus \ttX \oplus\tttn_+$.

\medbreak

Finally we prove \ref{item:gtilde-verp-maximal}. Let $\ttI$ be a $\bZ$-graded ideal such that $\ttI\cap\ttX=0$. Thus $\ttI=\oplus_{k\ne 0}\ttI_k$, where $\ttI_k=\ttI\cap\tttg_k$.
Take $\ttm$ as the sum of all $\bZ$-graded ideals such that $\ttI\cap\ttX=0$. Then $\ttm=\oplus_{k\ne 0}\ttm_k$, where $\ttm_k=\ttm\cap\tttg_k$ is the sum of the degree $k$ components of these $\bZ$-graded ideals, thus $\ttm$ is a $\bZ$-graded ideal such that $\ttm\cap\ttX=0$, and $\ttm=\ttm_-\oplus\ttm_+$.
\end{proof}

\begin{remark}\label{rem:gtilde-inclusions}
By parts \ref{item:gtilde-cartan-inclusion} and \ref{item:gtilde-verp-free}, if $\ttX$ has enough modules, then the canonical maps $ \ttX, \ttV, \ttV^* \to \tttg$ are inclusions. As a consequence of \Cref{lem:Cartan-involution}, the Cartan involution 
\begin{align}\label{eq:gtilde-Cartan-involution}
\widetilde{\omega} \colon \tttg(\rho,\ttd)\to\tttg(\rho^\vee,\ttd^\vee)& &\text{satisfies} 
&&\restr{\widetilde{\omega}}{\ttX}=\id_{\ttX}, &&\restr{\widetilde{\omega}}{\ttV}=\vartheta_{\ttV}, &&\restr{\widetilde{\omega}}{\ttV^*}=\id_{\ttV^*}.
\end{align}
Similarly, by \Cref{lem:gtilde-representation-on-tensor-algebra}, each $\ttX$-module $\lambda\colon \ttX\ot\ttW\to\ttW$ induces a $\tttg$-action $\eta$ on $\op{T}(\ttV)\ot\ttW$ whose restrictions $\eta_\ttX$, $\eta_+$ and $\eta_-$ to $\ttX$, $\ttV$ and $\ttV^*$ are given by \ref{item:Kac-module-0}, \ref{item:Kac-module-positive} and \ref{item:Kac-module-negative}, respectively.
\end{remark}

\begin{example}\label{exa:classical-revisited}
For $\cC=\sVec$, take a classical contragredient datum $(A,\bp)$ as in \ref{item:contragredient-1}--\ref{item:contragredient-4}. As explained in \Cref{exa:classical-semidirect}, from this we extract a contragredient datum $(\rho,d)\in \opD(\fh)$, and we have $\tfg(A,\bp)=\tttg(\rho,d)$ by definition. 
\end{example}
\subsection{Definition and first properties of the contragredient Lie algebra \texorpdfstring{$\ttg(\rho, \ttd)$}{g}}\label{subsec:contragredient-def}
Let us start with a somewhat pathological case.

\begin{example}\label{exa:contragredient-d-trivial}
Assume $\ttX\in\Lie(\cC)$ has enough modules, and take any module $\rho\colon \ttX\ot\ttV\to \ttV$. Then the zero map $0\colon \ttV\ot\ttV^* \to \ttX$ obviously satisfies \eqref{eq:contragredient-compatibility} thus $(\rho,0)$ constitutes a contragredient datum. In this case, $\ttV$ and $\ttV^*$ commute in the auxiliary algebra $\tttg(\rho,0)$, and we have a decomposition $\tttg(\rho,0)=\tttn_- \oplus \ttX \oplus \tttn_-$ where $\tttn_+$ and $\tttn_-$ are actually ideals. Hence the quotient $\tttg(\rho,0)/\ttm$ is just $\ttX$.
\end{example}

The previous example shows that, even though one can run through the construction of $\tttg(\rho,\ttd)/\ttm$ with no assumption on the map $\ttd$ apart from the compatibility \eqref{eq:contragredient-compatibility}, it is not reasonable to expect nice properties on the algebra $\tttg(\rho,\ttd)/\ttm$ at this level of generality. To avoid such undesired behaviors, we impose a first restriction on the map $\ttd$. 

\begin{definition}\label{def:reduced-data}
Fix a Lie algebra $\ttX$, and take $(\rho\colon\ttX\ot\ttV\to \ttV, \ttd\colon \ttV\ot\ttV^*\to\ttX)\in \opD(\ttX)$ as in \Cref{def:contragredient-data}. We say that the datum $(\rho,\ttd)$ is \emph{reduced} if for any simple subobjects $S\subseteq\ttV$ and $T\subseteq\ttV^*$, we have $\restr{\ttd}{S\ot\ttV^*}\ne 0$ and $\restr{\ttd}{\ttV\ot T}\ne 0$. 

We denote the set of reduced data over $\ttX$ by $\opRD(\ttX)$.
\end{definition}

\begin{example}\label{exa:reduced-datum-classical}
In the classical case over $\cC=\sVec$, see \Cref{exa:classical-revisited,exa:classical-semidirect}, we extract $d\colon V\ot V^*\to \fh$ given by $e_i\ot f_j \mapsto \delta_{ij} h_i$, which is clearly reduced. 
\end{example}
Our next result shows how reducibility  of $\ttd$ replicates on $\tttg(\rho,\ttd)$.
\begin{lemma}\label{lem:ideals-reduced-datum}
Let $(\ttX,\rho,\ttd)$ as in \Cref{def:reduced-data} and assume $\ttX$ has enough modules. If $\ttI$ is an ideal of $\tttg(\rho,\ttd)$ with $\ttI\cap\ttX=0$, then also $\ttI\cap\ttV=0$ and $\ttI\cap\ttV^*=0$.
\end{lemma}

\begin{proof}
Let $\ttb$ denote the bracket of $\tttg$. Assume for a contradiction that $\ttI\cap\ttV\ne0$, and take a simple $S\subseteq \ttI\cap\ttV$. Since $\ttI$ is an ideal, we have $\ttb(S\ot \ttV^*)\subseteq \ttI$, and by definition $\ttd$ takes values on $\ttX$. Thus using that $\restr{\ttb}{\ttV\ot\ttV^*}$ coincides with $\ttd$, we get
\[0\ne \ttd(S\ot \ttV^*)= \ttb(S\ot \ttV^*) \subseteq \ttI\cap \ttX,\]
a contradiction. Thus $\ttI\cap\ttV=0$, and similarly $\ttI\cap\ttV^*=0$. 
\end{proof}

\begin{remark}\label{rem:gtilde-bracket-grading}
Consider $(\ttX, \rho, \ttd)$ as above. Then the Lie bracket of $\op{F}\coloneq\FLie(\ttV\oplus \ttV^* )\rtimes_\rho \ttX$ clearly satisfies $\ttb_{\op{F}} (\op{F}_k \ot \ttV^*)= \op{F}_{k-1}$ for each integer $k$. Thus, if $\ttq$ is a quotient of $\tttg(\rho,\ttd)$ by a $\bZ$-homogeneous ideal which  intersects $\ttX$ trivially, then we also have $\ttb_{\ttq} (\ttq_k \ot \ttV^*)= \ttq_{k-1}$.
\end{remark}

Now we are ready to introduce the main object of interest.

\begin{definition}\label{def:contragredient}
Fix $\ttX\in\Lie(\cC)$ with enough modules and take a reduced datum $(\rho,\ttd)$ over $\ttX$ in the sense of \Cref{def:reduced-data}. The \emph{contragredient Lie algebra} associated to the triple $(\ttX, \rho, \ttd)$ is the operadic ind-Lie algebra $\ttg(\rho, \ttd)\coloneqq \tttg(\rho, \ttd)/\ttm$, where $\ttm$ is the maximal ideal from \Cref{thm:gtilde-verp} \ref{item:gtilde-verp-maximal}.
\end{definition}

All the observations from \Cref{rem:contragredient-gtilde} factor through:

\begin{remark}\label{rem:contragredient-inclusions}
\begin{itemize}[leftmargin=*]
\item If $\cC$ is a Frobenius exact pre-Tannakian category of moderate growth, then $\ttg(\rho, \ttd)$ is a Lie algebra in $\indcat{\cC}$, rather than an operadic one.
\item By definition of $\ttm$, and thanks to \Cref{lem:ideals-reduced-datum}, the inclusions $\ttX, \ttV, \ttV^*\hookrightarrow \tttg(\rho,\ttd)$ from \Cref{thm:gtilde-verp} \ref{item:gtilde-cartan-inclusion}, \ref{item:gtilde-verp-free} induce inclusions $\ttX, \ttV, \ttV^*\hookrightarrow \ttg(\rho,\ttd)$.
\item 
By \Cref{thm:gtilde-verp} \ref{item:gtilde-verp-maximal}, the defining ideal of $\ttg(\rho, \ttd)$ is homogeneous, which gives a grading
\begin{align}\label{eq:contragredient-Z-graded}
\ttg(\rho, \ttd)&=\bigoplus_{k\in\bZ} \ttg_k(\rho, \ttd),
\end{align}
where $\ttg_k(\rho, \ttd)$ is the image of $\tttg_k(\rho, \ttd)$. 
\item We denote by $\ttn_+$ and $\ttn_-$ the subalgebras of $\ttg(\rho, \ttd)$ generated by $\ttV$ and $\ttV^*$, respectively. 
\item The subset $\opRD(\ttX)\subseteq \opD(\ttX)$ is closed under the duality introduced in \Cref{rem:contragredient-dual-pair}.
\end{itemize}
\end{remark}

From \Cref{thm:gtilde-verp} we obtain the first structural results for contragredient Lie algebras.

\begin{theorem}\label{thm:g-contragredient-verp}
Let $(\ttX, \rho, \ttd)$ as in \Cref{def:contragredient}, and assume that $\ttX$ has enough modules. Then, for the operadic Lie algebra $\ttg\coloneq\ttg(\rho, \ttd)$, the following hold:
\begin{enumerate}[leftmargin=*,label=\rm{(\alph*)}]
\item \label{item:g-contragredient-involution}  There exists a unique operadic Lie algebra isomorphism 
\begin{align}\label{eq:Cartan-involution-contragredient}
\omega& \colon \ttg(\rho, \ttd)\to \ttg(\rho^{\vee}, \ttd^{\vee}) & &\text{such that}&&\restr{\omega}{\ttX}=\id_{\ttX},  
&&\restr{\omega}{\ttV}=\vartheta_{\ttV}, &&\restr{\omega}{\ttV^*}=\id_{\ttV^*}.
\end{align}


\item \label{item:g-contragredient-verp-triangular} 
We have $\ttg=\ttn_-\oplus \ttX \oplus\ttn_+$ as objects in $\indcat{\cC}$. Moreover,
\begin{align}\label{eq:g-contragredient-verp-triangular}
\ttn_- &= \bigoplus_{k<0} \ttg_k, & \ttX &= \ttg_0, & \ttn_+ &= \bigoplus_{k>0} \ttg_k. 
\end{align}
\item\label{item:g-contragredient-verp-trivial-bracket} If $Y$ is a $\bZ$-homogeneous subobject of $\ttn_+$ (respectively, $\ttn_-$) such that $\ttb(Y\otimes\ttV^*)=0$ (respectively, $\ttb(Y\otimes\ttV)=0$), then $Y=0$.
\end{enumerate}
\end{theorem}
\begin{proof}
For \ref{item:g-contragredient-involution} we check that the isomorphism $\widetilde{\omega}$ from \Cref{lem:Cartan-involution} satisfies $\widetilde{\omega}(\ttm)=\ttm^{\vee}$, where $\ttm^{\vee}$ is the maximal $\bZ$-graded ideal of $\tttg(\rho^{\vee}, \ttd^{\vee})$ among those that meet $\ttX$ trivially. Notice that $\widetilde{\omega}(\tttg_k)=\tttg(\rho^{\vee}, \ttd^{\vee})_{-k}$ for all $k\in\bZ$, so $\widetilde{\omega}(\ttm)$ is a $\bZ$-graded ideal, and $\widetilde{\omega}(\ttm)\cap\ttX=0$ since $\ttX=\widetilde{\omega}(\ttX)$. Thus $\widetilde{\omega}(\ttm)\subseteq \ttm^{\vee}$ by maximality of $\ttm^{\vee}$. Analogously, $\widetilde{\omega}^{-1}(\ttm^{\vee})\subseteq\ttm^{\vee}$, so $\widetilde{\omega}(\ttm)=\ttm^{\vee}$.
\smallbreak

Now \ref{item:g-contragredient-verp-triangular} follows from the definition of the grading \eqref{eq:contragredient-Z-graded}, just by applying the canonical projection $\pi\colon\tttg\twoheadrightarrow\ttg$ to the correspondent decompositions \eqref{eq:gtilde-verp-triangular} for $\tttg$:
\begin{align*}
\ttn_{\pm} &=\pi(\tttn_{\pm })=\pi\left( \oplus_{\pm k>0} \tttg_k \right)=\oplus_{\pm k>0} \ttg_k, &
\ttX &=\pi(\ttX)=\pi\left(\tttg_0 \right)=\ttg_0.
\end{align*}

\smallbreak

Finally we prove \ref{item:g-contragredient-verp-trivial-bracket}, for which suffices to deal with the case $Y \subseteq \ttn_+$ since the other one follows by applying $\omega$. Set $n\in\bN$ as the degree of $Y$, and 
\begin{align*}
\ttI&=Y+\left(\sum_{k \ge 0} \ttb(\ttg_k\ot Y)\right).
\end{align*}
We claim that $\ttI$ is an ideal. As $\ttg$ is generated by $\ttV$, $\ttX$ and $\ttV^*$, it suffices to check that 
\begin{align*}
\ttb(\ttV\ot\ttI) & \subseteq \ttI, & \ttb(\ttX\ot\ttI) & \subseteq \ttI, & \ttb(\ttV^*\ot\ttI) & \subseteq \ttI.
\end{align*}
The first inclusion follows since $\ttb(\ttV\ot Y)=\ttb(\ttg_1\ot Y)$ is one of the summands defining $\ttI$, and because for each $k\ge0$ we get $\ttb(\ttV \ot \ttb(\ttg_k\ot Y)) \subseteq \ttb(\ttg_{k+1}\ot Y)$ by using the Jacobi identity and that $\ttg$ is $\bZ$-graded. Similarly, $\ttb(\ttX\ot\ttI)\subseteq \ttI$ since $\ttb(\ttX\ot Y)=\ttb(\ttg_0\ot Y)$ is one of the summands defining $\ttI$, and $\ttb(\ttX \ot \ttb(\ttg_k\ot Y)) \subseteq \ttb(\ttg_{k}\ot Y)$ for all $k\ge 0$, using the Jacobi identity and that each $\ttg_k$ is an $\ttX$-submodule. For the last inclusion we use that $\ttb(\ttV\ot Y)=0$ by hypothesis and the Jacobi identity again to check that $\ttb(\ttV^* \ot \ttb(\ttX\ot Y))=0$ and
\begin{align*}
\ttb(\ttV^* \ot \ttb(\ttg_k\ot Y)) & =\ttb(\ttb(\ttV^* \ot \ttg_k)\ot Y) \subseteq\ttb(\ttg_{k-1}\ot Y) &&\text{for all }k\ge 1.
\end{align*}
Hence $\ttI$ is an ideal. As $\ttg$ is a $\bZ$-graded operadic Lie algebra, we have that
\begin{align*}
Y+\ttb(\ttg_0\ot Y) &= Y+\ttb(\ttX\ot Y)\subseteq \ttg_{n},  & \ttb(\ttg_k\ot Y) &\subseteq \ttg_{n+k}, & &k>0.
\end{align*}
Thus $\ttI$ is a $\bZ$-graded ideal intersecting $\ttg_0=\ttX$ trivially. By maximality of $\ttm$ we get $\ttI=0$ and then also $Y=0$, as desired.
\end{proof}

\begin{remark}\label{rem:derived-subalgebra}
Consider $(\ttX, \rho, \ttd)$ as above and assume that $\rho$ is surjective. Let $\ttX'_\ttd\subseteq \ttX$ denote the subalgebra of $\ttX$ generated by the derived subalgebra $\ttX'$ and the image of $\ttd$. Then the derived subalgebra $\ttg'\subseteq\ttg$ coincides with the subalgebra generated by $\ttV^*$, $\ttX'_\ttd$ and $\ttV$, and it satisfies $\ttg'=\ttn_-\oplus\ttX'_\ttd\oplus\ttn_+$.
\end{remark}

\begin{remark}\label{rem:contr-data-semi-direct}
Consider $(\ttX, \rho, \ttd)$ as above. Assume that $\ttX=\ttX_1\oplus\ttX_2$ as Lie algebras, with $\ima\ttd\subseteq \ttX_1$, and set $\rho_i=\restr{\rho}{\ttX_i}$ for $i=1,2$. 
We view $\ttd$ as a map $\ttd\colon\ttV\ot \ttV^*\to\ttX_1$, so $(\ttX_1,\rho_1,\ttd)$ is a reduced contragredient datum.

The $\ttX_2$-actions on $\ttV$ and $\ttV^*$ extend to an action $\widehat{\rho_2}$ on $\FLie(\ttV\oplus \ttV^* )\rtimes_{\rho_1} \ttX_1$ (with $\ttX_2$ acting trivially on $\ttX_1$), which descends to $\ttg(\ttX_1,\rho_1,\ttd)$. Moreover, it is easy to see that
\begin{align*}
\ttg(\ttX,\rho,\ttd) \cong \ttg(\ttX_1,\rho_1,\ttd)\rtimes_{\widehat{\rho_2}} \ttX_2.
\end{align*}
Such decomposition will be used in \Cref{sec:examples}.
\end{remark}

One may wonder when a contragredient Lie algebra belongs to $\cC$ rather than  $\indcat{\cC}$. There is a general, easy case where this happens: when the action of $\ttX$ is trivial.

\begin{corollary}\label{cor:trivial-action}
Let $(\ttX,\rho,\ttd)$ as in \Cref{thm:g-contragredient-verp}. If $\rho=0$, then $\ttg=\ttV^*\oplus\ttX\oplus\ttV$ as $\bZ$-graded objects in $\cC$ and as $\ttX$-modules in $\cC$. The bracket $\ttb$ of $\ttg$ is given by 
\begin{align*}
\restr{\ttb}{\ttV \ot \ttV^*}&=\ttd, & \restr{\ttb}{\ttX \ot \ttV}=\restr{\ttb}{\ttX \ot \ttV^*}= \restr{\ttb}{\ttV \ot \ttV}&=\restr{\ttb}{\ttV^* \ot \ttV^*}=0. 
\end{align*}
\end{corollary}

\begin{proof}
In $\tttg=\tttg(\rho,\ttd)$, consider $Y=\ttb(\ttV\ot\ttV)=\tttg_{2}$ and denote by $\ov{Y}$ its image in $\ttg$. Using the Jacobi identity and that $\rho=0$ we check that $\ttb(\ov{Y}\ot\ttV^*)=0$ in $\ttg$. Hence $\ov{Y}=0$ by \Cref{thm:g-contragredient-verp} \ref{item:g-contragredient-verp-trivial-bracket}. In other words $Y\subset\ttm_+$, and as $\ttm$ is an ideal we have that $\ttm_+=\oplus_{n\ge 2}\tttg_{n}$ since $\tttg_{n+1}=\ttb(\tttg_{n}\otimes\ttV)$. Analogously, $\ttm_-=\oplus_{n\le 2}\tttg_{n}$, and the statement follows.
\end{proof}

Next we show that, as in the classical case, if a quotient of $\tttg(\rho,\ttd)$  admits a symmetric bilinear form with certain non-degeneracy, then it is the minimal quotient $\ttg(\rho,\ttd)$.

\begin{proposition}\label{prop:contragredient-bilinear-form}
Fix $\ttX\in\Lie(\cC)$ with enough modules and take a reduced datum $(\rho,\ttd)\in\opRD(\ttX)$ in the sense of \Cref{def:reduced-data}.

Let $\ttq=\oplus_{k\in\bZ} \ttq_k$ be a quotient of $\tttg(\rho,\ttd)$ by a $\bZ$-homogeneous ideal $\ttI$ such that 
$ \ttI\cap\ttX=0$. 
If $\ttq$ admits a symmetric invariant bilinear form $\ttB\colon\ttq\ot\ttq\to\one$ such that $\restr{\ttB}{\ttq_k\ot\ttq_{-k}}$ is non-degenerate for all $k\in\bZ$,
then $\ttI$ is the maximal ideal $\ttm$ from \Cref{thm:gtilde-verp} \ref{item:gtilde-verp-maximal}, i.e. $\ttq=\ttg(\rho,\ttd)$.
\end{proposition}
\begin{proof}
By \Cref{lem:ideals-reduced-datum}, we have $\ttI=\oplus_{k\ne 0, \pm 1} \ttI_k$, where $\ttI_k=\ttI\cap\tttg_k$. Suppose that there exists $k>1$ such that $\ttI_k\subsetneq \ttm_k$, and take $k$ minimal with this condition. Now $Y\coloneq\ttm_k/\ttI_k$ is a non-zero subobject of $\ttq_k$ such that $\ttb_{\ttq}(Y\ot \ttV^*)=0$; indeed $\ttb_{\tttg}(\ttm_k\ot \ttV^*)\subseteq \ttm_{k-1}$, and $\ttm_{k-1}=\ttI_{k-1}$ by minimality of $k$. Using this fact and that $\ttB$ is invariant, we get
\begin{align}\label{eq:gtilde-quotient-bilinear}
0 &= \ttB\left( \ttb_{\ttq}(Y\ot \ttV^*)\ot \ttq_{1-k} \right) =  \ttB\left( Y\ot \ttb_\ttq(\ttV^*\ot \ttq_{1-k}) \right).
\end{align}
On the other hand, by \Cref{rem:gtilde-bracket-grading}, we have $\ttb_{\ttq}(\ttV^*\ot \ttq_{1-k})=\ttq_{-k}$. Thus \eqref{eq:gtilde-quotient-bilinear} implies that $0= \ttB\left( Y\ot \ttq_{-k} \right)$, and by assumption on $\ttB$ we get $Y=0$, a contradiction. Hence $\ttI_k=\ttm_k$ for all $k>1$. Analogously, one shows that $\ttI_k=\ttm_k$ for any $k<-1$. So $\ttI=\ttm$, as desired.
\end{proof}

\begin{remark}
After the first version of this work was made public, we became aware of the recent preprint \cite{Sherman-Silberberg}, which develops a similar construction in the category $\sVec$ for $\op{char} \bk = 0$. They define a Lie superalgebra $\fg(\cA)$ starting from a \emph{Cartan datum} $\cA$, which consists on the following: 
\begin{itemize}[leftmargin=*]
\item a quasi-toral Lie algebra $\fh$ in $\sVec$ (this means that the even part $\ft=\fh_{0}$ is central in $\fh$);
\item a finite set of independent weights $\Pi\subset \ft^*$;
\item for each  $\alpha\in\pm \Pi$ an $\fh$-module $\fg_\alpha$ where $\ft$ acts via $\alpha$; 
\item a collection of $\fh$-equivariant maps $[-,-]_{\alpha}\colon \fg_{\alpha}\ot\fg_{-\alpha}\to \fh$.
\end{itemize}
This is called the \emph{queer Kac-Moody} construction, since in particular the so-called queer Lie superalgebras $Q(n)$ can be obtained in this way.

Under the additional assumption that $\fg_{-\alpha}=\fg_{\alpha}^{\vee}$ for all $\alpha \in \Pi$, it follows by definition that $\fg(\cA)$ is contragredient in our sense. Indeed, $\fg(\cA)=\ttg(\ttX,\rho,\ttd)$ where $\ttX=\fh$, $\rho$ is the direct sum of the $\fg_{\alpha}$'s and $\ttd$ is the coproduct arrow of the $[-,-]_{\alpha}$'s.
\end{remark}

\subsection{Symmetrizable data}
In this Section we study contragredient data over a Lie algebra $\ttX$ which is assumed to come equipped with a symmetric non-degenerate invariant form. It turns out that for any module $\rho\colon \ttX\ot\ttV \to \ttV$ one can produce a map $\ttd\colon \ttV \ot \ttV^*\to\ttX$ satisfying \eqref{eq:contragredient-compatibility}, which happens to be reduced under a further assumption.

\begin{definition}\label{def:symmetrizable-data}
Let $\ttX \in \Lie(\cC)$.  
A \emph{symmetrizable contragredient datum} over $\ttX$ is a pair $(\rho, \ttK)$ where
\begin{enumerate}[leftmargin=*,label=\rm{(SD\arabic*)}]
\item\label{item:symmetrizable-def-1} $\rho\colon \ttX \ot \ttV \to \ttV$ is an $\ttX$-module in $\cC$ such that $\rho(\ttX\ot S)\ne 0$ for all simple $S\subseteq \ttV$ and $\pi\rho\ne 0$ for every projection $\pi\colon \ttV\to T$ onto a simple object $T\subseteq \ttV$.
\item\label{item:symmetrizable-def-2} $\ttK\colon \ttX \ot \ttX \to \one$ is a symmetric non-degenerate invariant form.
\end{enumerate}
We denote the set of symmetrizable contragredient data over $\ttX$ by $\opSD(\ttX)$.

Given $(\rho,\ttK)\in \opSD(\ttX)$, we denote by $\ttd_{\rho,\ttK}$ the image of $\rho$ under the isomorphism
\begin{equation}\label{eq:symmetrizable-dmap}
\Hom_{\cC}(\ttX\ot\ttV, \ttV) \xrightarrow{*} \Hom_{\cC}(\ttV^*, \ttV^*\ot \ttX^*) \cong \Hom_{\cC}(\ttV\ot\ttV^*, \ttX^*) \cong\Hom_{\cC}(\ttV\ot\ttV^*, \ttX),  
\end{equation}
where the first map is the dualization functor, the second one is the dualization adjunction and the third one denotes composition with the isomorphism $\ttX^*\cong\ttX$ induced by $\ttK$.
\end{definition}

\begin{example}\label{exa:classical-revisited-symmetrizable}
For $\cC=\sVec$, take a classical contragredient datum $(A,\bp)$ as in \ref{item:contragredient-1}--\ref{item:contragredient-4}, and assume $A$ is symmetrizable, i.e. $A=DB$ for some invertible $D=\text{diag}(\varepsilon_1,\dots,\varepsilon_r)$ and symmetric $B$. In \Cref{exa:classical-semidirect}, we extracted from $(A,\bp)$ a contragredient datum $(\rho,d)\in \opD(\fh)$, where $d\colon V\ot V^* \to \fh$ is defined \emph{ad-hoc} by $d(e_i\ot f_j)=\delta_{ij}h_i$ and the matrix $A$ encodes the action of the $h_i$'s. For an alternative, more organic approach, consider the symmetric non-degenerate form $K$ on $\fh$ such that $K(h_i, h_j)=\varepsilon_i \varepsilon_j b_{ij}$, as defined in \cite{Kac-book}*{\S 2.1}. Now it is easy to see that the image of $\rho$ under \eqref{eq:symmetrizable-dmap} is given by $\ttd_{\rho,K}(e_i\ot f_j)=\delta_{ij}\varepsilon_i^{-1}h_i=\varepsilon_i^{-1}d(e_i\ot f_j)$, thus $\ttg(\rho,\ttd_{\rho,K})\simeq \fg(A,\bp)$.
\end{example}

We record for later use a compatibility fulfilled by any symmetrizable datum.

\begin{lemma}\label{lem:symmetrizable-datum-compatibility-K-d}
Given $\ttX\in \Lie(\cC)$ and $(\rho,\ttK)\in \opSD(\ttX)$, the following equality holds:
\begin{align}\label{eq:symmetrizable-datum-compatibility-K-d}
\ttK (\id_{\ttX}\otimes \ttd_{\rho,\ttK}) & = \ev_{\ttV}(\rho\otimes\id_{\ttV^*}) \in \Hom_{\cC}(\ttX\otimes\ttV\otimes\ttV^*,\one).
\end{align}
Moreover, $\ttd_{\rho,\ttK}$ is the unique element in $\Hom_{\cC}(\ttV\ot\ttV^*, \ttX)$ satisfying \eqref{eq:symmetrizable-datum-compatibility-K-d}.
\end{lemma}

\begin{proof}
Let $\psi\colon \ttX\to\ttX^*$ denote the isomorphism induced by $\ttK$. Let us proceed diagrammatically; we stick to \Cref{notation:contragredient-setup}, but since now $\ttd_{\rho, \ttK}$ is obtained from $\ttK$ and $\rho$, we also need finer vocabulary. Instead of introducing a symbol for $\ttK\colon \ttX\ot\ttX \to \one$, we chose to do so for the morphism $\psi\colon \ttX\to\ttX^*$ it induces. Consider
\begin{equation*}
\id_{\ttX}=\xy
(0,0)*{
\begin{tikzpicture}[scale=.3]
\draw [very thick,  blue, ->] (0,0) to (0,-4);
\end{tikzpicture} 
};
\endxy, 
\qquad
\id_{\ttX^*}=\xy
(0,0)*{
\begin{tikzpicture}[scale=.3]
\draw [very thick, blue, ->] (0,-4) to (0,0);
\end{tikzpicture} 
};
\endxy, 
\qquad 
\psi=\xy
(0,0)*{
\begin{tikzpicture}[scale=.3]
\draw [very thick,blue, ->] (0,0) to (0,-1.5);
\draw [very thick,blue, <-] (0,-2.5) to (0,-4);
\draw [fill=white] (1,-1.5) rectangle (-1,-2.5);
\node at (0,-2) {\tiny $\psi$};
\end{tikzpicture} 
};
\endxy,
\qquad 
\psi^{-1}=\xy
(0,0)*{
\begin{tikzpicture}[scale=.3]
\draw [very thick,blue, <-] (0,0) to (0,-1.5);
\draw [very thick,blue, ->] (0,-2.5) to (0,-4);
\draw [fill=white] (1,-1.5) rectangle (-1,-2.5);
\node at (0,-2) {\tiny $\psi^{-1}$};
\end{tikzpicture} 
};
\endxy.
\end{equation*}
So now $\ttK$ and $\ttd$ become 
\begin{equation*}
\ttK=\xy
(0,0)*{
\begin{tikzpicture}[scale=.3]
\draw [very thick,blue, ->] (0,0) to (0,-1.5);
\draw [very thick,blue, <-] (0,-2.5) to (0,-4);
\draw [fill=white] (1,-1.5) rectangle (-1,-2.5);
\node at (0,-2) {\tiny $\psi$};
\draw [very thick, blue] (0,-4) to [out=-90, looseness=2, in=-90] (2,-4);
\draw [very thick, blue, ->] (2,0) to (2, -4);
\end{tikzpicture} 
};
\endxy, \qquad
\ttd=\xy
(0,0)*{
\begin{tikzpicture}[scale=.3]
\draw [very thick, ->] (0,0) to (0,-5);
\draw [very thick] (0,-5) to [out=-90, looseness=2, in=-90] (2,-5);
\draw [very thick, ->] (2,-5) to (2,-2.5);
\draw [very thick, ->] (3,-1.5) to (3, 0);
\draw [fill=white] (1,-1.5) rectangle (5,-2.5);
\node at (3,-2) {\tiny {$\rho^*$}};
\draw [very thick,blue, <-] (4,-2.5) to (4,-4);
\draw [very thick,blue, ->] (4,-5) to (4,-6.5);
\draw [fill=white] (5,-4) rectangle (3,-5);
\node at (4,-4.5) {\tiny $\psi^{-1}$};
\end{tikzpicture} 
};
\endxy,
\qquad
\rho^*=\xy
(0,0)*{
\begin{tikzpicture}[scale=.3]
\draw [very thick, <-] (0,0) to (0,-5);
\draw [very thick] (0,-5) to [out=-90, looseness=1, in=-90] (4,-5);
\draw [very thick, blue] (2,-2) to [out=90, looseness=1, in=90] (8,-2);
\draw [very thick, ->] (4,-2) to (4, -5);
\draw [very thick] (4,-2) to [out=90, looseness=1, in=90] (6,-2);
\draw [very thick, <-] (6,-2) to (6,-6);
\draw [very thick,blue,  <-] (8,-2) to (8,-6);
\draw [very thick, blue, bend right=40] (2,-2) to (4,-4);
\draw [blue, fill=blue] (4,-4) circle (0.3cm and 0.3cm);
\end{tikzpicture} 
};
\endxy,
\end{equation*}
where $\rho^*\in \Hom_\cC(\ttV^*, \ttV^*\ot\ttX^*)$ denotes the left dual of $\rho\in\Hom(\ttX\ot\ttV, \ttV)$. Now we compute
\begin{align*}
\ttK(\id_{\ttX}\ot\ttd)= 
\xy
(0,0)*{
\begin{tikzpicture}[scale=.3]
\draw [very thick, ->] (0,0) to (0,-5);
\draw [very thick] (0,-5) to [out=-90, looseness=2, in=-90] (2,-5);
\draw [very thick, ->] (2,-5) to (2,-2.5);
\draw [very thick, ->] (3,-1.5) to (3, 0);
\draw [fill=white] (1,-1.5) rectangle (5,-2.5);
\node at (3,-2) {\tiny {$\rho^*$}};
\draw [very thick,blue, <-] (4,-2.5) to (4,-4);
\draw [very thick,blue, ->] (4,-5) to (4,-6.5);
\draw [fill=white] (5,-4) rectangle (3,-5);
\node at (4,-4.5) {\tiny $\psi^{-1}$};
\draw [very thick,blue, ->] (-2,0) to (-2,-1.5);
\draw [very thick,blue, <-] (-2,-2.5) to (-2,-6);
\draw [fill=white] (-1,-1.5) rectangle (-3,-2.5);
\node at (-2,-2) {\tiny $\psi$};
\draw [very thick, blue] (-2,-6) to [out=-90, looseness=1, in=-90] (4,-6.5);
\end{tikzpicture} 
};
\endxy=
\xy
(0,0)*{
\begin{tikzpicture}[scale=.3]
\draw [very thick, ->] (-6,0) to (-6,-5);
\draw [very thick] (-6,-5) to [out=-90, looseness=1, in=-90] (2,-5);
\draw [very thick, ->] (2,-5) to (2,-2);
\draw [very thick,blue, <-] (4,-2) to (4,-4);
\draw [very thick,blue, ->] (4,-5) to (4,-6.5);
\draw [fill=white] (5,-4) rectangle (3,-5);
\node at (4,-4.5) {\tiny $\psi^{-1}$};
\draw [very thick,blue, ->] (-8,0) to (-8,-1.5);
\draw [very thick,blue, <-] (-8,-2.5) to (-8,-6);
\draw [fill=white] (-7,-1.5) rectangle (-9,-2.5);
\node at (-8,-2) {\tiny $\psi$};
\draw [very thick, blue] (-8,-6) to [out=-90, looseness=.7, in=-90] (4,-6.5);
\draw [very thick, <-] (-4,0) to (-4,-5);
\draw [very thick] (-4,-5) to [out=-90, looseness=1, in=-90] (0,-5);
\draw [very thick, blue] (-2,-2) to [out=90, looseness=1, in=90] (4,-2);
\draw [very thick, ->] (0,-2) to (0, -5);
\draw [very thick] (0,-2) to [out=90, looseness=1, in=90] (2,-2);
\draw [very thick, blue, bend right=40] (-2,-2) to (0,-4);
\draw [blue, fill=blue] (0,-4) circle (0.3cm and 0.3cm);
\end{tikzpicture} 
};
\endxy
=
\xy
(0,0)*{
\begin{tikzpicture}[scale=.3]
\draw [very thick, blue, ->] (-2,2) to (-2,0);
\draw [very thick, ->] (0,2) to (0,0);
\draw [very thick, blue, bend right=40] (-2,0) to (0,-2);
\draw [very thick, ->] (0,2) to (0,-4);
\draw [blue, fill=blue] (0,-2) circle (0.3cm and 0.3cm);
\draw [very thick, <-] (2,2) to (2,-4);
\draw [very thick] (0,-4) to [out=-90, looseness=2, in=-90] (2,-4);
\end{tikzpicture} 
};
\endxy
=
\ev_{\ttV}(\rho \ot \id_{\ttV^*}), 
\end{align*}
as claimed. If $\widetilde{\ttd} \in \Hom_{\cC}(\ttV\ot\ttV^*, \ttX)$ also satisfies \eqref{eq:symmetrizable-datum-compatibility-K-d}, then
$\ttK \left(\id_{\ttX}\otimes (\ttd - \widetilde{\ttd}) \right)=0$, thus the last statement follows directly by non-degeneracy of $\ttK$.
\end{proof}

Now we verify that symmetrizable data provide contragredient data.

\begin{lemma}\label{lem:symmetrizable-data-gives-data}
Let $\ttX\in \Lie(\cC)$. If $(\rho,\ttK)$ is a symmetrizable datum (\Cref{def:symmetrizable-data}), then $(\rho,\ttd_{\rho,\ttK})$ is a reduced contragredient datum (\Cref{def:contragredient-data,def:reduced-data}).
\end{lemma}

\begin{proof}
Write $\ttd=\ttd_{\rho,\ttK}$ and consider $\rho^\vee$ as in \eqref{eq:rho-vee}. To show that $(\rho,\ttd)$ is a contragredient datum, we have to verify the equality \eqref{eq:contragredient-compatibility}, or equivalently, that $\ttd$ is a morphism of $\ttX$-modules. Since $\rho$ is a morphism of $\ttX$-modules, it is enough to verify that all the maps in \eqref{eq:symmetrizable-dmap} preserve such morphisms. Both the dualization functor and the dualization adjunction preserve such morphisms because evaluation and coevaluation of $\ttX$-modules are morphisms in $\Mod_{\cC}(\ttX)$. Thus, the image of $\rho$ in $\Hom_{\cC}(\ttV\ot\ttV^*, \ttX^*)$ is a morphism of $\ttX$-modules. In addition, since the non-degenerate form $\ttK$ is invariant, the isomorphism $\ttX^*\cong\ttX$ it induces lies actually in $\Mod_{\cC}(\ttX)$. Thus $\ttd$ is a composition of two morphisms of $\ttX$-modules.

Next we check that $(\rho,\ttd_{\rho,\ttK})$ is reduced. For each simple subobject $S\subseteq \ttV$, the dual map $(\restr{\rho}{\ttX\otimes S})^*\colon \ttV^*\to S^*\otimes\ttX^*$ is not zero by \ref{item:symmetrizable-def-1}, so its adjoint $S\otimes\ttV^*\to\ttX$ is nonzero and then $\restr{\ttd}{S\otimes\ttV^*}\ne 0$. Similarly, the quotient condition from \ref{item:symmetrizable-def-1} implies that $\restr{\ttd}{\ttV\otimes T}\ne 0$ for any simple subobject $T\subseteq\ttV^*$.
\end{proof}

\subsection{Contragredient Lie algebras arising from symmetrizable data}
We show that, if we start with symmetrizable data, then any decomposition of the module $\ttV$ as a direct sum $\ttX$-submodules yields a lattice grading for the contragredient Lie algebra, and an invariant form with certain refined non-degeneracy.
First, we establish the set-up for this Section.

\begin{notation}\label{not:symmetrizable-datum}
\begin{itemize}[leftmargin=*]
\item Take $\ttX \in \Lie(\cC)$ with enough modules and a symmetrizable datum $(\rho,\ttK)$.
\item Fix a decomposition of $\rho\colon \ttX \ot \ttV \to \ttV$ as a sum of $\ttX$-submodules, say $\ttV=\mathop{\oplus}\limits_{1\leq i \leq r} \ttV_i$.
\item Consider the reduced datum $(\rho, \ttd)$ from \Cref{lem:symmetrizable-data-gives-data}.  Since each $\ttV_i$ is $\ttX$-stable, we have
\begin{align}\label{eq:symmetrizable-datum-ortogonal-summands}
\ttd(\ttV_i\otimes \ttV_j^*)&=0 \text{ for all }i\ne j, & \text{so }\ttd(\ttV_i\otimes \ttV_i^*) &\ne 0 \text{  for all }1\le i\le r.
\end{align}
\item Take a lattice  $Q\coloneq\mathop{\oplus}\limits_{1\leq i \leq r}\alpha_i\bZ$, and put $Q^\pm\coloneq\pm\mathop{\oplus}\limits_{1\leq i \leq r}\alpha_i\bN_0$. Consider the \emph{height function}
\begin{align*}
\he &\colon Q\to \bZ, &  \sum\limits_{1\leq i \leq r} b_i \alpha_i &\mapsto \sum\limits_{1\leq i \leq r} b_i.
\end{align*}
\end{itemize}
\end{notation}

\begin{theorem}\label{thm:contragredient-symmetrizable-grading}
Consider the set-up from \Cref{not:symmetrizable-datum}. Then $\ttg\coloneq \ttg(\rho,\ttd)$ admits a $Q$-grading such that
\begin{align}\label{eq:contragradient-symmetrizable-grading-def} 
\ttg&=\bigoplus_{\alpha \in Q} \ttg_\alpha, &\deg \ttX&=0, &\deg \ttV_i&=\alpha_i, &\deg \ttV_i^*&=-\alpha_i, & 1&\leq i\leq r.
\end{align}
Moreover, under this grading we have
\begin{align}\label{eq:contragradient-symmetrizable-grading-parts} 
\ttg_0&=\ttX, &\ttn_+&=\bigoplus_{\alpha \in Q^+, \alpha \neq 0} \ttg_\alpha,  &\ttn_-&=\bigoplus_{\alpha \in Q^-, \alpha \neq 0} \ttg_\alpha,
\end{align}
and the Cartan involution satisfies $\omega(\ttg(\rho,\ttd)_\alpha)=\ttg(\rho^\vee,\ttd^\vee)_{-\alpha}$ for all $\alpha \in Q$.
\end{theorem}
\begin{proof}
First, it is clear that the operadic Lie algebra $\FLie(\ttV\oplus \ttV^* )\rtimes_\rho \ttX$ has a $Q$-grading as in \eqref{eq:contragradient-symmetrizable-grading-def}.
Now, by \eqref{eq:symmetrizable-datum-ortogonal-summands}, the defining relations of $\tttg\coloneq \tttg(\rho,\ttd)$ are $Q$-homogeneous, so the $Q$-grading on $\FLie(\ttV\oplus \ttV^* )\rtimes_\rho \ttX$ descends to $\tttg$. By \Cref{thm:gtilde-verp} \ref{item:gtilde-verp-free}, the subalgebras $\ttn^+$ and $\ttn_{-}$ are free on $\ttV$ and $\ttV^*$, respectively; thus the triangular decomposition \eqref{eq:gtilde-verp-triangular} implies that 
\begin{align}\label{eq:gtilde-symmetrizable-grading-parts} 
\tttg_0&=\ttX, &\tttn_{\pm}&=\bigoplus_{\alpha \in Q^{\pm}, \alpha \ne 0} \tttg_\alpha,  &\tttg_{k}&=\bigoplus_{\he \alpha=k} \tttg_\alpha \, \, \, (k \in \bZ).
\end{align}

Now we show that the $Q$-grading further descends to $\ttg=\tttg/\ttm$. As $\ttm=\oplus_{|k|>1} \ttm_k$, it suffices to check that $\ttm_k=\oplus_{\he \alpha=k}\ttm \cap \tttg_\alpha$ whenever $|k|>1$. Thanks to the Cartan involution from \eqref{eq:g-contragredient-verp-triangular}, we can restrict to values $k>1$, for which we argue inductively.

For $k=2$, using the $Q$-grading on $\tttg$, we write $\ttm_2=\bigoplus_{\he \alpha=2}\ttm_{\alpha}$ for some $\ttm_{\alpha} \subseteq \tttg_\alpha$. Since $\ttm\subset \tttg$ is a $\bZ$-graded ideal with $\ttm_1=0$, and recalling that $\ttV^*$ has $\bZ$-degree $-1$ in $\tttg$, we obtain
\begin{align*}
0 &= \ttb\left(\ttm_2\ot \ttV_j^*\right) = \ttb\left(\bigoplus_{\he \alpha=2}\ttm_{\alpha}\ot \ttV_j^*\right), &\text{for all } &1\le j\le r.
\end{align*}
But $\ttb(\ttm_{\alpha}\ot \ttV_j^*)\subseteq \tttg_{\alpha-\alpha_j}$, so $\ttb(\ttm_{\alpha}\ot \ttV_j^*)=0$ for all $j$. Thus $\ttb(\ttm_{\alpha}\ot \ttV^*)=0$ and, by \Cref{thm:g-contragredient-verp} \ref{item:g-contragredient-verp-trivial-bracket}, it follows that $\ttm_\alpha$ vanishes in the quotient $\ttg=\tttg/\ttm$. In other words, each $Q$-homogeneous component of $\ttm_2$ is contained in $\ttm_2$, as desired.

For the inductive step, assume that $\ttm_k=\oplus_{\he \alpha=k}\ttm \cap \tttg_\alpha$ and write $\ttm_{k+1}=\bigoplus_{\he \beta=k+1}\ttm_{\beta}$, with $\ttm_{\beta} \subseteq \tttg_{\beta}$. Then $\ttb\left(\ttm_{k+1}\ot \ttV_j^*\right)\subseteq \ttm_k$  also $\ttb(\ttm_{\beta}\ot \ttV_j^*)\subseteq \tttg_{\beta-\alpha_j}$ for all $1\le j\le r$ and $\beta$ of height $k+1$. In case $\beta-\alpha_j\notin\bN_0^r$, we have $\tttg_{\beta-\alpha_j}=0$ by \eqref{eq:gtilde-symmetrizable-grading-parts}, thus $\ttb(\ttm_{\beta}\ot \ttV_j^*)=0$. For the case $\beta-\alpha_j\in\bN_0^r$ we have $\ttb(\ttm_{\beta}\ot \ttV_j^*)\subseteq \ttm \cap \tttg_{\beta-\alpha_j}$ by inductive hypothesis, because the latter is the component in degree $\beta-\alpha_j$ of $\ttb\left(\ttm_{k+1}\ot \ttV_j^*\right)\subseteq \ttm_k$. In any case, for any $\beta$ of height $k+1$, we obtain $\ttb(\ttm_{\beta}\ot \ttV^*)\subseteq \ttm_k$, so $\ttm_{\beta}\subseteq \ttm_{k+1}$. Thus $\ttm_{k+1}=\oplus_{\he \beta=k+1}\ttm \cap \tttg_\beta$ and the inductive step follows.
\end{proof}

Over $\cC=\sVec$, take a classical contragredient datum $(A,\bp)$ as in \ref{item:contragredient-1}--\ref{item:contragredient-4}. It is well known, see e.g. \cite{Kac-book}*{Exercise 2.3} that $\fg(A,\bp)$ admits a non-degenerate invariant super-symmetric form if and only if $A$ is symmetrizable. Our next result clarifies the choice of the word \emph{symmetrizable} in symmetrizable datum.

\begin{theorem}\label{thm:contragredient-bilinear}
Under the set-up from \Cref{not:symmetrizable-datum}, consider on $\ttg\coloneq \ttg(\rho,\ttd)$ the $Q$-grading from \Cref{thm:contragredient-symmetrizable-grading}. Then $\ttg$ admits an invariant symmetric form $\ttB\colon\ttg\ot\ttg\to\one$ such that
\begin{enumerate}[leftmargin=*,label=\rm{(\alph*)}]
\item\label{item:contragradien-bilinear-X} $\restr{\ttB}{\ttX\ot\ttX}=\ttK$,
\item\label{item:contragradien-bilinear-ga-gb} $\restr{\ttB}{\ttg_\alpha\ot\ttg_\beta}=0$ if $\alpha+\beta\ne0$, and
\item\label{item:contragradien-bilinear-ga-ga} $\restr{\ttB}{\ttg_\alpha\ot\ttg_{-\alpha}}$ is non-degenerate for all $\alpha \in Q$.
\end{enumerate}
\end{theorem}
\begin{proof}
We follow the arguments of \cite{Kac-book}*{Theorem 2.2} and define, in first place, an invariant symmetric form $\tttB\colon\tttg\ot\tttg\to\one$ satisfying \ref{item:contragradien-bilinear-X} and \ref{item:contragradien-bilinear-ga-gb}. Recall that $\tttB$ is invariant if
\begin{align}\label{eq:invariant-bilinear-form}
\tttB(\ttb\ot \id_{\tttg}) &= \tttB(\id_{\tttg}\ot\ttb) \in\Hom_{\cC}(\tttg\ot\tttg\ot\tttg,\one),
\end{align}
where $\ttb$ is the bracket of $\tttg$. Set $\tttg[n]=\oplus_{-n\le k\le n}\tttg_k$ for each non-negative integer $n$. We define recursively on $n$ a symmetric form $\tttB\colon\tttg[n]\otimes\tttg[n]\to\one$ satisfying \ref{item:contragradien-bilinear-X}, \ref{item:contragradien-bilinear-ga-gb}, and
\begin{align}\label{eq:invariant-bilinear-form-pieces}
\tttB(\ttb\ot \id_{\tttg}) &= \tttB(\id_{\tttg}\ot\ttb) \in\Hom_{\cC}(\tttg_i\ot\tttg_j\ot\tttg_k,\one), &\text{ for all } |i+j|, |j+k| &\le n.
\end{align}
For $n=1$ we define
\begin{align*}
\restr{\tttB}{\tttg_0\otimes\tttg_0} &=\ttK, & \restr{\tttB}{\tttg_{\pm1}\otimes\tttg_{\mp1}} &=\ev_{\tttg_{\pm1}}, & 
\restr{\tttB}{\tttg_i\otimes\tttg_j} &=0 \text{ otherwise}.
\end{align*}
By definition, $\tttB=\tttB\circ c$, that is $\tttB$ is symmetric, and \ref{item:contragradien-bilinear-X} holds. 
As $\restr{\ev}{\ttV_i\ot\ttV_j^*}=0$ if $i\ne j$, \ref{item:contragradien-bilinear-ga-gb} also holds. 
Hence it remains to check  \eqref{eq:invariant-bilinear-form-pieces}:
\begin{itemize}[leftmargin=*]
\item For $i=j=k=0$ the equality holds since the restriction to $\ttX\otimes\ttX$ is $\ttK$, which is invariant.
\item $\restr{\tttB(\ttb\ot \id)}{\tttg_{1}\ot\tttg_0\ot\tttg_{-1}}=-\ev_{\ttV}(\rho c_{\ttV,\ttX}\ot\id)=\ev_{\ttV}(\id\otimes\rho^{\vee})=\restr{\tttB(\id\ot \ttb)}{\tttg_{1}\ot\tttg_0\ot\tttg_{-1}}$, where the first and the third identities follow by definition and the second follows since $\ev_{\ttV}$ is a map of $\ttX$-modules. Similarly we have that
$\restr{\tttB(\ttb\ot \id)}{\tttg_{-1}\ot\tttg_0\ot\tttg_{1}}=\restr{\tttB(\id\ot \ttb)}{\tttg_{-1}\ot\tttg_0\ot\tttg_{1}}$.
\item $\restr{\tttB(\ttb\ot \id)}{\tttg_0\ot\tttg_{1}\ot\tttg_{-1}}=\ev_{\ttV}(\rho\ot\id) \overset{\eqref{eq:symmetrizable-datum-compatibility-K-d}}{=} \ttK(\id\ot\ttd) = \restr{\tttB(\id\ot \ttb)}{\tttg_0\ot\tttg_{1}\ot\tttg_{-1}}$.
\item Over $\tttg_0\ot\tttg_{-1}\ot\tttg_{1}$ we have
\begin{align*}
\restr{\tttB(\ttb\ot \id)}{\tttg_0\ot\tttg_{-1}\ot\tttg_{1}}& =\ev_{\ttV^*}(\rho^{\vee}\ot\id)= \ev_{\ttV}c_{\ttV^*,\ttV}(\rho^{\vee}\ot\id)
=\ev_{\ttV}(\id\otimes\rho^{\vee})c_{\ttX\otimes\ttV^*,\ttV}
\\
&=\ev_{\ttV}(\id\otimes\rho^{\vee})(c_{\ttX,\ttV}\otimes\id)(\id\otimes c_{\ttV^*,\ttV})
=-\ev_{\ttV}(\rho\otimes\id)(\id\otimes c_{\ttV^*,\ttV})
\\ 
&=-\ttK(\id\otimes\ttd c_{\ttV^*,\ttV})=\restr{\tttB(\id\ot \ttb)}{\tttg_0\ot\tttg_{-1}\ot\tttg_{1}}.
\end{align*}
\item The identities $\restr{\tttB(\ttb\ot \id)}{\tttg_{\pm1}\ot\tttg_{\mp1}\ot\tttg_0}=\restr{\tttB(\id\ot \ttb)}{\tttg_{\pm1}\ot\tttg_{\mp1}\ot\tttg_0}$ follow from the previous two cases, up to composing with the braiding on the right.
\item In all the other cases, $\restr{\tttB(\ttb\ot \id)}{\tttg_i\ot\tttg_j\ot\tttg_k}=0=\restr{\tttB(\id\ot \ttb)}{\tttg_i\ot\tttg_j\ot\tttg_k}$.
\end{itemize}

Assume now that we have a symmetric form $\tttB\colon\tttg[n-1]\otimes\tttg[n-1]\to\one$ satisfying \ref{item:contragradien-bilinear-X}, \ref{item:contragradien-bilinear-ga-gb}, and  \eqref{eq:invariant-bilinear-form-pieces}, and let us extend it to  $\tttg[n]$.  We define $\restr{\tttB}{\tttg_{\pm n}\otimes\tttg[n-1]}= \restr{\tttB}{\tttg[n-1]\otimes\tttg_{\pm n}}=0$ and $\restr{\tttB}{\tttg_{\pm n}\otimes\tttg_{\pm n}}=0$. By symmetry, it only remains to define $\restr{\tttB}{\tttg_{n}\otimes\tttg_{-n}}$. We proceed in steps.
\begin{claim}\label{claim:invariant-form-well-def}
Given integers $i,j> 0 >s,t$ with $i+j=n$ and $s+t=-n$, we have
\begin{align}\label{eq:invariant-form-well-def}
\tttB(\ttb(\ttb\ot \id) \ot \id)&= \tttB(\id\ot\ttb(\id\ot \ttb)) \in\Hom_{\cC}(\tttg_i\ot\tttg_j\ot\tttg_s\ot \tttg_t,\one).
\end{align}
\end{claim}
This follows as in the proof of \cite{Kac-book}*{Theorem 2.2} using that  $\tttB$ is invariant over $\tttg[n-1]$. 
\begin{claim}\label{claim:invariant-form-kernels}
For any integers $i,j,s,t$ as in \Cref{claim:invariant-form-well-def}, we have  
\begin{align*}
\ker\left( \restr{\ttb}{\tttg_i \ot \tttg_j}\ot \restr{\ttb}{\tttg_s \ot \tttg_t}\right) \subseteq \ker\tttB\left(\restr{\ttb}{\tttg_n \ot \tttg_s}\left(\restr{\ttb}{\tttg_i \ot \tttg_j}\ot \id_{\tttg_s} \right) \ot \id_{\tttg_t} \right).
\end{align*}
\end{claim}
In fact, 
\begin{align*}
\ker&\left( \restr{\ttb}{\tttg_i \ot \tttg_j}\ot \restr{\ttb}{\tttg_s \ot \tttg_t}\right)
 = \tttg_i\ot\tttg_j \ot  \ker\restr{\ttb}{\tttg_s \ot \tttg_t}\oplus  \ker \restr{\ttb}{\tttg_i \ot \tttg_j} \ot \tttg_s\ot \tttg_t \\
& \subseteq  \ker\tttB\left(\id_{\tttg_i} \ot \restr{\ttb}{\tttg_n \ot \tttg_s}\left( \id_{\tttg_j}\ot \restr{\ttb}{\tttg_s \ot \tttg_t} \right)\right)
\oplus \ker\tttB\left(\restr{\ttb}{\tttg_j \ot \tttg_{-n}}\left(\restr{\ttb}{\tttg_i \ot \tttg_j}\ot \id_{\tttg_s} \right) \ot \id_{\tttg_t} \right),
\end{align*}
and these last two kernels coincide by \Cref{claim:invariant-form-well-def}.

\begin{claim}\label{claim:invariant-form-diagrams}
Given integers $i,j,s,t$ as in \Cref{claim:invariant-form-well-def}, there is a unique map $\psi\colon \tttg_n\ot\tttg_{-n}\to \one$ such that the following diagram commutes
\begin{equation}\label{eq:invariant-form-diagrams}
\begin{tikzcd}
&&\tttg_n \ot \tttg_s \ot \tttg_t 
\arrow[twoheadrightarrow]{d}{\id \ot \ttb} 
\arrow{rr}{\ttb\ot\id} 
&& \tttg_{n+s} \ot \tttg_t \arrow{d}{\tttB}
\\
\tttg_i \ot \tttg_j \ot \tttg_s \ot \tttg_t 
\arrow[twoheadrightarrow]{rr}{\ttb\ot\ttb} \arrow[twoheadrightarrow]{urr}{\ttb\ot\id\ot \id}
\arrow[twoheadrightarrow, swap]{drr}{\id\ot \id\ot \ttb}
&& \tttg_n \ot \tttg_{-n} \arrow[dashed]{rr}{\psi} && \one 
\\
&&\tttg_i \ot \tttg_j \ot \tttg_{-n} 
\arrow[twoheadrightarrow, swap]{u}{\ttb\ot\id}
\arrow{rr}{\id\ot \ttb}
&& \tttg_{i} \ot \tttg_{j-n} \arrow{u}{\tttB}
\end{tikzcd}
\end{equation}
\end{claim}
Note that it is enough to show there is a map $\psi$ such that diagram obtained from \eqref{eq:invariant-form-diagrams} by disregarding the two double-headed vertical arrows is commutative. By  \Cref{claim:invariant-form-kernels}, there exist a unique map $\psi_1\colon  \tttg_n\ot \tttg_{-n} \to \one$ such that the top half diagram commutes. Using also  \Cref{claim:invariant-form-well-def}, there is a map $\psi_2\colon \tttg_n\ot \tttg_{-n} \to \one$ such that the bottom half diagram commutes. Finally, using \Cref{claim:invariant-form-well-def} again, one can see as in  \cite{Kac-book}*{Theorem 2.2} that actually $\psi_1=\psi_2$. 

\begin{claim}\label{claim:invariant-form-independence}
The map $\psi$ from \Cref{claim:invariant-form-diagrams} does not depend on $i,j,s,t$. 
\end{claim}

Again, this follows from \Cref{claim:invariant-form-well-def} as shown in  \cite{Kac-book}*{Theorem 2.2}. 

\medbreak
These Claims show the existence of an extended map $\tttB\colon\tttg[n]\otimes\tttg[n]\to\one$ which satisfies \eqref{eq:invariant-bilinear-form-pieces} thanks to \Cref{claim:invariant-form-diagrams}.
This map is symmetric and satisfies \ref{item:contragradien-bilinear-X} by definition. For a verification of \ref{item:contragradien-bilinear-ga-gb}, if $|\he\alpha|,|\he\beta|\le n-1$ then $\restr{\tttB}{\ttg_\alpha\ot\ttg_\beta}=0$ by inductive hypothesis while for 
$|\he\alpha|=n>|\he\beta|$ or $|\he\alpha|< n=|\he\beta|$ the equality follows by definition. It remains to check \ref{item:contragradien-bilinear-ga-gb} when $|\he\alpha|=|\he\beta|=n$ and $\alpha+\beta\ne0$, which follows by definition of $\restr{\tttB}{\tttg_{\pm n}\otimes\tttg_{\mp n}}$ and the inductive hypothesis. This conclude the inductive step.

\medbreak
Finally, we show that $\tttB$ factors through the quotient $\ttg=\tttg/\ttm$. Let $\ttI$ be the radical of $\tttB$. By \ref{item:contragradien-bilinear-X} and \ref{item:contragradien-bilinear-ga-gb} the ideal $\ttI$ is $Q$-graded, with trivial component in degree $0$ as $\ttK$ is non-degenerate. Accordingly, $\ttI$ is $\bZ$-graded and $\ttI\cap\ttX=0$. Thus $\tttg/\ttI$ projects onto $\ttg$ and admits an invariant symmetric form (the induced one) whose restrictions to $(\tttg/\ttI)_k\otimes(\tttg/\ttI)_{-k}$ is non-degenerate for all $k\in\bN$. By \Cref{prop:contragredient-bilinear-form}, it must be $\ttI=\ttm$, so $\tttB$ induces a non-degenerate invariant symmetric form $\ttB\colon\ttg\ot\ttg\to\one$ satisfying \ref{item:contragradien-bilinear-X} and \ref{item:contragradien-bilinear-ga-gb}. Now \ref{item:contragradien-bilinear-ga-ga} holds by non-degeneracy of $\ttB$ thanks to part \ref{item:contragradien-bilinear-ga-gb}.
\end{proof}

\begin{remark}\label{rem:tori}
One of the fundamental features of contragredient Lie superalgebras is that the zero part corresponds to an algebraic torus in $\Vec$. At the moment, we do not have a notion of algebraic tori in a general symmetric tensor category $\cC$, not even in $\Verp$. Anyhow, the rank-one general linear groups $\op{GL}(L)$ where $L$ is simple should definitely be considered a torus. Already in $\cC=\Verp$, the representation theory of these tori can be rather involved, see \cite{Ven-GLVerp}. From previous experience we desire to construct Lie algebras with a zero part that is non-abelian but still has a well-behaved action on the entire algebra. This explains why we do not impose many restrictions on the algebra $\ttX$, but on the $\ttX$-module $\ttV$ instead.
\end{remark}

\begin{remark}
The lattice grading from \Cref{thm:contragredient-symmetrizable-grading} should be considered as a first step toward a root system theory for contragredient Lie algebras in symmetric tensor categories. In that direction, and following the observations in \Cref{rem:tori}, one should start with a finer decomposition $\ttV=\oplus \ttV_i$ of $\ttV$ into a direct sum of submodules. We expect to address this and related questions in future work.
\end{remark}

\section{Examples}\label{sec:examples}

We conclude this work with examples of contragredient Lie algebras in symmetric categories, with an emphasis in $\Ver_p$. In  \Cref{subsec:examples-general-linear,subsection:semisimplification-examples} we exhibit  contragredient realizations of known Lie algebras, and in \Cref{subsec:examples-over-gl(L2)} we obtain new examples.

\subsection{General and special Lie algebras}\label{subsec:examples-general-linear}
Following \cite{Ven-GLVerp}, given an object $L$ in a symmetric tensor category $\cC$ one can construct the general Lie algebra $\fgl(L)\in \Lie(\cC)$, with underlying object $L \ot L^*$ and associative multiplication given by $\ev_L$, see  \Cref{rem:module-representation}.  Furthermore, $\fgl(L)$ contains a Lie subalgebra $\fsl(L)$ defined by  \cite{Ven-GLVerp}*{Proposition 3.5} as the kernel of the map $\ev_{L^*}\colon \fgl(L)\to\one$. As for Lie superalgebras, in case $\dim L\ne 0$ we obtain a decomposition $\fgl(L)= \one \oplus \fsl(L)$; here $\one$ is a central Lie subalgebra obtained as the image of $\coev_{L}$, see \cite{Ven-GLVerp}*{Proposition 3.8}. In this Section we show that, if $L$ is semisimple, then $\fgl(L)$ and $\fsl(L)$ are contragredient Lie algebras.

Assume that $L\in \cC$ is semisimple, and fix a decomposition $L=\oplus_{1\leq i \leq r} X_i$ as sum of simple objects. 
As described in \cite{Ven-GLVerp}*{\S 5}, for the general Lie algebra $\fgl(L)$, we take 
\begin{itemize}[leftmargin=*,label=$\circ$]
\item a maximal torus $\ttX=\oplus_{1\leq i \leq r} X_i\otimes X_i^*=\oplus_{1\leq i \leq r} \fgl(X_i)$, the \emph{diagonal matrices}, and
\item a generating subobject $\ttV=\oplus_{1\leq i <r} X_i\otimes X_{i+1}^*$ of the Borel subalgebra $\oplus_{1\leq i<j \leq r} X_i\otimes X_j^*$ of \emph{upper triangular matrices}.
\end{itemize}

Now we define contragredient data over $\ttX$. Notice that $\ttV^*=\oplus_{1\leq i <r} X_{i+1}\otimes X_{i}^*$. 
Let $\rho\colon\ttX\otimes\ttV\to\ttV$ and $\ttd\colon \ttV\ot\ttV^*\to\ttX$ denote the restrictions of the bracket $\ttb$ of $\fgl(L)$; explicitly,
\begin{align*}
\restr{\rho}{\left(X_i\ot X_i^*\right) \otimes \left(X_j\ot X_{j+1}^*\right)} &= 
\begin{cases}
0 & i\ne j,j+1, \\
\id\ot \ev_{X_i}\ot\id, & i=j, \\
-(\id\ot \ev_{X_i}\ot\id)c_{\left(X_i\ot X_i^*\right),\left(X_{i-1}\ot X_i^*\right)}, & i=j+1;
\end{cases}
\\
\restr{\ttd}{\left(X_i\ot X_{i+1}^*\right) \otimes \left(X_{j+1}\ot X_j^*\right)} &= 
\begin{cases}
0 & i\ne j, \\
\id\ot \ev_{X_{j+1}}\ot\id -(\id\ot \ev_{X_{j}}\ot\id)c_{\left(X_{j}\ot X_{j+1}^*\right),\left(X_{j+1}\ot X_j^*\right)}, & i=j.
\end{cases}
\end{align*}
We also define $\ttK\colon \ttX \ot \ttX \to \one$ as the form such that
\begin{align*}
\restr{\ttK}{\left(X_i\ot X_i^*\right) \otimes \left(X_j\ot X_{j}^*\right)} &= 
\begin{cases}
0 & i\ne j, \\
\ev_{X_i^*}\left(\id \ot \ev_{X_i}\ot\id\right)=\ev_{X_i\ot X_i^*}, & i=j,
\end{cases}
\end{align*}
which is clearly non-degenerate.

\begin{lemma}\label{lemma:gl-contragredient-symmetric}
The pair $(\rho, \ttK)$ is a symmetrizable contragredient datum over $\ttX$, and $\ttd_{\rho, \ttK}=\ttd$.
\end{lemma}
\begin{proof}
We have to check that $(\rho, \ttK)$ satisfies \ref{item:symmetrizable-def-1} and \ref{item:symmetrizable-def-2} in \Cref{def:symmetrizable-data}. 

For \ref{item:symmetrizable-def-1}, notice that 
$\restr{\rho}{\left(X_i\ot X_i^*\right) \otimes \left(X_i\ot X_{i+1}^*\right)}=\id\ot \ev_{X_i}\ot\id\ne 0$, so
$\rho(\ttX\ot S)\ne 0$ for all simple $S\subseteq X_i\ot X_{i+1}^*$ and $\pi\rho\ne 0$ for every projection $\pi\colon X_i\ot X_{i+1}^*\to T$ onto a simple object $T\subseteq X_i\ot X_{i+1}^*$. For \ref{item:symmetrizable-def-2}, we have to check that $\ttK$ is invariant. Consider
\begin{align*}
\mathtt{M}_{ijk} &\coloneq\left(X_i\ot X_i^*\right) \otimes\left(X_j\ot X_j^*\right) \otimes\left(X_k\ot X_k^*\right), && \text{for each }1\le i,j,k\le r,
\end{align*}
so $\ttX\otimes\ttX\otimes\ttX=\oplus_{i,j,k}\mathtt{M}_{ijk}$ and we can just verify the invariance condition $\ttK(\ttb\ot\id)=\ttK(\id\ot\ttb)$ by restricting to  each $\mathtt{M}_{ijk}$. Note that $\restr{\ttK(\ttb\ot\id)}{\mathtt{M}_{ijk}}=0=\restr{\ttK(\id\ot\ttb)}{\mathtt{M}_{ijk}}$ if $\#\{i,j,k\}\ge 2$, and
\begin{align*}
\restr{\ttK(\ttb\ot\id)}{\mathtt{M}_{iii}} &=\ev_{X_i\ot X_i^*}
\left(\id_{X_i} \ot \ev_{X_i}\ot\id_{X_i^{*}\ot X_i\ot X_i^{*}}\right)
\\ & -\ev_{X_i\ot X_i^*}\left(\id_{X_i} \ot \ev_{X_i}\ot\id_{X_i^{*}\ot X_i\ot X_i^{*}}\right)(c_{X_i\ot X_i^{*},X_i\ot X_i^{*}}\ot\id_{X_i\ot X_i^{*}})
\\
\restr{\ttK(\id\ot\ttb)}{\mathtt{M}_{iii}} &=\ev_{X_i\ot X_i^*}
\left(\id_{X_i\ot X_i^{*}\ot X_i} \ot \ev_{X_i}\ot\id_{X_i^{*}}\right)
\\ & - \ev_{X_i\ot X_i^*}\left(\id_{X_i\ot X_i^{*}\ot X_i} \ot \ev_{X_i}\ot\id_{X_i^{*}}\right) (\id_{X_i\ot X_i^{*}} \ot c_{X_i\ot X_i^{*},X_i\ot X_i^{*}}).
\end{align*}
Now, 
\begin{align}\label{eq:evaluations}
\ev_{X_i\ot X_i^*} & \left(\id_{X_i} \ot \ev_{X_i}\ot\id_{X_i^{*}\ot X_i\ot X_i^{*}}\right)
=\ev_{X_i\ot X_i^*} \left(\id_{X_i\ot X_i^{*}\ot X_i} \ot \ev_{X_i}\ot\id_{X_i^{*}}\right),
\end{align}
and using diagramatic calculus we check that
\begin{align}\label{eq:evaluations+c}
\begin{aligned}
\ev_{X_i\ot X_i^*} &
\left(\id_{X_i} \ot \ev_{X_i}\ot\id_{X_i^{*}\ot X_i\ot X_i^{*}}\right)\left(c_{X_i\ot X_i^{*},X_i\ot X_i^{*}}\ot\id_{X_i\ot X_i^{*}}\right)
\\
&=\ev_{X_i\ot X_i^*}\left(\id_{X_i\ot X_i^{*}\ot X_i} \ot \ev_{X_i}\ot\id_{X_i^{*}}\right)
\left(\id_{X_i\ot X_i^{*}}c_{X_i\ot X_i^{*},X_i\ot X_i^{*}}\right).
\end{aligned}
\end{align}
Hence, $\restr{\ttK(\ttb\ot\id)}{\mathtt{M}_{iii}} =\restr{\ttK(\id\ot\ttb)}{\mathtt{M}_{iii}}$. Therefore the form is invariant.

To verify the last statement, by \Cref{lem:symmetrizable-datum-compatibility-K-d}, we have to check that $\ttd$ satisfies \eqref{eq:symmetrizable-datum-compatibility-K-d}. Let
\begin{align*}
\mathtt{N}_{ijk} &:= \left(X_i\ot X_i^*\right) \otimes\left(X_j\ot X_{j+1}^*\right) \otimes\left(X_{k+1}\ot X_k^*\right), && 1\le i\le r, \ 1\le j,k<r,
\end{align*}
so $\ttX\otimes\ttV\otimes\ttV^*=\oplus_{i=1}^r\oplus_{j,k=1}^{r-1}\mathtt{N}_{ijk}$. Note that the equality $\restr{\ttK (\id_{\ttX}\otimes \ttd)}{\mathtt{N}_{iii}} = \restr{\ev_{\ttV}(\rho\otimes\id_{\ttV^*})}{\mathtt{N}_{iii}}$ follows by \eqref{eq:evaluations}, while \eqref{eq:evaluations+c} implies that $\restr{\ttK (\id_{\ttX}\otimes \ttd)}{\mathtt{N}_{i\, i+1\, i+1}} = \restr{\ev_{\ttV}(\rho\otimes\id_{\ttV^*})}{\mathtt{N}_{i\, i+1\, i+1}}$.
In all the other cases, it is clear that $\restr{\ttK (\id_{\ttX}\otimes \ttd)}{\mathtt{N}_{ijk}} =0= \restr{\ev_{\ttV}(\rho\otimes\id_{\ttV^*})}{\mathtt{N}_{ijk}}$.

Thus $\ttK (\id_{\ttX}\otimes \ttd)= \ev_{\ttV}(\rho\otimes\id_{\ttV^*})$ and the statement follows.
\end{proof}

To construct a contragredient datum for $\fsl(L)$, consider the  Lie subalgebra $\ttX'=\ttX\cap\fsl(L)$ and take the $\ttX'$-action $\rho'\colon\ttX'\otimes\ttV\to\ttV$ obtained by restriction of $\rho$.
As $\ev_{L^*}\circ\ttd=\ev_{L^*}\circ\ttb=0$, we have that $\ima\ttd \subseteq \ttX'$, so we regard $\ttd$ as a map $\ttd\colon\ttV\ot\ttV^*\to\ttX'$.

\begin{proposition}\label{prop:gl-sl-contragredient}
\begin{enumerate}[leftmargin=*, label=(\arabic*)]
\item\label{item:gl-sl-contragredient-1} The pairs $(\rho,\ttd)$ and $(\rho',\ttd)$ are contragredient data.
\item\label{item:gl-sl-contragredient-2} There exist Lie algebra isomorphisms $\fgl(L)\simeq \ttg(\rho,\ttd)$ and $\fsl(L)\simeq \ttg(\rho',\ttd)$.
\end{enumerate}
\end{proposition}
\begin{proof}
\ref{item:gl-sl-contragredient-1} Notice that $(\rho,\ttd)$ is a contragredient datum by Lemmas \ref{lem:symmetrizable-data-gives-data} and \ref{lemma:gl-contragredient-symmetric}. As a consequence so is $(\rho',\ttd)$.

\smallbreak

\ref{item:gl-sl-contragredient-2} We apply \Cref{prop:contragredient-bilinear-form}. The canonical Lie algebra map $\FLie(\ttV\oplus\ttV^*)\rtimes_\rho \ttX\to\fgl(L)$ identifying $\ttV$, $\ttV^*$ and $\ttX$ inside $\fgl(L)$ descends to a Lie algebra map $\widetilde{\Psi}\colon \tttg(\rho,\ttd)\to\fgl(L)$, which is surjective. We decompose
\begin{align*}
\fgl(L) &=\bigoplus_{-r<k<r} \fgl(X)_k, & \text{ where } \fgl(L)_k &\coloneq \bigoplus_{j-i=k} X_i\ot X_j^*.
\end{align*}
Notice that this decomposition satisfies
\begin{align*}
\ttb\left(\fgl(L)_{\pm i}\otimes \fgl(L)_{\pm j}\right) &=\fgl(L)_{\pm(i+j)}, & i,j & \ge 0.
\end{align*}
Hence $\fgl(L)$ is a $\bZ$-graded quotient Lie algebra of $\tttg(\rho,\ttd)$ such that $\fgl(L)_1=\ttV$, $\fgl(L)_0=\ttX$ and $\fgl(L)_{-1}=\ttV^*$.

Now we define a form $\ttB\colon \fgl(L)\ot\fgl(L)\to\one$ by
\begin{align*}
\ttB\left( (X_i\otimes X_j^*)\otimes (X_r\otimes X_s^*) \right) &=\begin{cases}
0, & i\ne s \text{ or }j\ne r,
\\
\ev_{X_i^*}(\id\otimes\ev_{X_{j}}\otimes\id), & i=s, j=r.
\end{cases}
\end{align*}
This $\ttB$ is a symmetric invariant form such that $\restr{\ttB}{\fgl(L)_k\ot\fgl(L)_{-k}}$ is non-degenerate for all $1\le k<r$. Thus, by Proposition \ref{prop:contragredient-bilinear-form}, $\widetilde{\Psi}$ induces the desired isomorphism $\Psi\colon\ttg(\rho,\ttd)\to\fgl(L)$.

Now the subalgebra of $\ttg(\rho,\ttd)$ generated by $\ttX'$, $\ttV$ and $\ttV^*$ is (isomorphic to) $\ttg(\rho',\ttd)$, and $\Psi(\ttX'), \Psi(\ttV), \Psi(\ttV^*) \subseteq \fsl(L)$, so the image of $\restr{\Psi}{\ttg(\rho',\ttd)}$ is contained in $\fsl(L)$. Using again Proposition \ref{prop:contragredient-bilinear-form} we deduce that $\restr{\Psi}{\ttg(\rho',\ttd)}\colon\ttg(\rho',\ttd)\to\fsl(L)$ is a Lie algebras isomorphism.
\end{proof}

\subsection{Semisimplification of ordinary contragredient Lie algebras}\label{subsection:semisimplification-examples}

As in \Cref{subsec:semisimplification}, consider a finite dimensional contragredient Lie algebra $\fg(A)$ in $\Vec$, and fix an index $i$ with $a_{ii}=2$. 
We have
$(\ad e_i)^p=0$ so $\fg$ becomes  a Lie algebra in $\Rep(\balp_{p})$ by letting the generator $t$ of the coordinate algebra $\bk[t]/(t^p)$ act via $\ad e_i$. 

Let $\sfS (\fg, e_i)$ be the operadic Lie algebra in $\Ver_p$ obtained as the semisimplification of $\fg$ by the action of $\ad e_i$. Next we summarize the main result from \cite{APS-ssLieAlg}, which describes the structure of this algebra. We refer to \emph{loc. cit.} for unexplained terminology.

\begin{theorem}{\cite{APS-ssLieAlg}}\label{thm:semisimplification-contragredient-main-thm}
The Lie algebra $\sfS(\fg, e_i)$ in $\Ver_p$ has a triangular decomposition 
\begin{align*}
\sfS(\fg, e_i)=\sfS(\fg, e_i)_+\oplus\sfS(\fg, e_i)_0\oplus\sfS(\fg, e_i)_-.
\end{align*} 

It  also admits grading by $\bZ^{r-1}$
whose non-trivial homogeneous components are simple.

In addition, $\sfS(\fg, e_i)$ has an invariant non-degenerate symmetric form $\mathtt{B}$ such that
\begin{align*}
\restr{\mathtt{B}}{\sfS(\fg, e_i)_{\alpha} \ot \sfS(\fg, e_i)_{\beta}} &=0 && \text{ if }\alpha\ne -\beta\in \bZ^{r-1}.
\end{align*}

Under a mild assumption on the root system of $A$ and the adjoint action of $e_i$, $\sfS(\fg, e_i)$ is generated by its components in $\bZ$-degrees $-1$, $0$ and $1$. \qed

\end{theorem}

Now it is easy to see that $\sfS(\fg,e_i)\simeq\ttg(\ttX,\rho,\ttd)$, where:
\begin{enumerate}[leftmargin=*]
\item $\ttX=\sfS(\fg,i)_0=\overline{S}\oplus\overline{\fh}$. Here  $\overline{S}\cong\mathfrak{sl}(\ttL_2)$ and $\overline{\fh}\in \Vec\subset\Verp$ is an abelian Lie algebra of dimension $r-1$.
\item $\ttV$ is the degree $1$ component of $\sfS(\fg,i)$ and $\ttV^*$ is (identified with) the degree $-1$ component:
\begin{align*}
\ttV&=\bigoplus_{j\in\overline{\bI}} \tte_j, & \ttV^*&=\bigoplus_{j\in\overline{\bI}} \ttf_j.
\end{align*}
\item The action $\rho\colon\ttX\otimes\ttV\to\ttV$ and the map $\ttd\colon \ttV\otimes\ttV^*\to\ttX$ are the restrictions of the bracket of $\sfS(\fg,i)$ to $\sfS(\fg,i)_0\otimes\sfS(\fg,i)_1$ and $\sfS(\fg,i)_1\otimes\sfS(\fg,i)_{-1}$, respectively. More explicitly, 
\begin{itemize}[leftmargin=*,label=$\circ$]
\item each $\tte_j$ and is a stable under the adjoint $\ttX$-action,
\item $\restr{\rho}{\overline{\fh}\ot\ttV}$ is given by the matrix $(2a_{jk}-a_{ji}a_{ik})_{j\ne i,k\in\overline{\bI}}$,
\item $\restr{\ttd}{\tte_j\otimes\ttf_k}=0$ if $j\ne k$,
\item the action of $\mathfrak{sl}(\ttL_2)$ and the diagonal map $\restr{\ttd}{\tte_j\ot\ttf_j}\colon\tte_j\ot\ttf_j \to \mathfrak{sl}(\ttL_2)\oplus \sfS(\Bbbk \widetilde{h}_j) \subseteq \ttX$ are given explicitly in \cite{APS-ssLieAlg}*{Lemmas 4.16 \& 4.18}.
\end{itemize}
\end{enumerate}
To show that this contragredient datum yields an isomorphism $\sfS(\fg, e_i)\simeq\ttg(\ttX,\rho,\ttd)$, we just apply \Cref{thm:semisimplification-contragredient-main-thm} together with \Cref{prop:contragredient-bilinear-form}.

\subsection{On contragredient Lie algebras over \texorpdfstring{$\fgl(\ttL_2)$}{}}\label{subsec:examples-over-gl(L2)}
In $\Verp$, the rank-one general Lie algebra over the simple object $\ttL_2$ decomposes as $\fgl(\ttL_2)=\one\oplus \fsl(\ttL_2)$ where $\one$ is central and $\fsl(\ttL_2) \simeq \ttL_3$ as an object. Recall from \Cref{exa:enough-modules-gl-verp,exa:enough-modules-simple} that $\fsl(\ttL_2)$ and $\fgl(\ttL_2)$ have enough modules. In this Section we study contragredient Lie algebras over the three non-zero subalgebras  $\ttX \subseteq\fgl(\ttL_2)$. We restrict to contragredient algebras of rank $1$, meaning that the positive part is generated by a single simple object in $\Verp$.

So fix a simple object $\ttV=\ttL_k$ in $\Verp$. Recall for later use that actions of a Lie algebra $\ttX$ on $\ttL_k$ correspond to Lie algebra maps $\ttX\to \fgl(\ttL_k)$, and that the underlying object of $\fgl(\ttL_k)$ is $\ttL_k\ot \ttL_k^*=\ttL_k\ot \ttL_k$, which by the fusion rules \eqref{eq:verp-fusionrules} decomposes as $\ttL_k\ot \ttL_k = \bigoplus_{i=1}^{\min\{k, p-k\}} \ttL_{2i-1}$.

We look for $\fgl(\ttL_2)$-contragredient structures on a simple object $\ttV=\ttL_k$. 

\begin{example}
In case  $\ttV = \ttL_k$ where $k =1$ or $k=p-1$, such structures are rather classical and we recover the rank one contragredient Lie superalgebras  with an enlarged zero part. Indeed, by the fusion rules, an action $\rho$ of $\fgl(\ttL_2)=\one \oplus \fsl(\ttL_2)$ on $\ttL_k$ is trivial on the $\fsl(\ttL_2)$ and is thus determined by a scalar $a$ encoding the action of $\one$. Also, a map $\ttd\colon \ttL_k\ot \ttL_k^* \to \fgl(\ttL_2)$ only lands in $\one$, and it is reduced precisely when it is nonzero. Any such map $\ttd$ is a morphism of $\fgl(\ttL_2)$-modules regardless of the action $\rho$, and the contragredient Lie algebra $\ttg(\fgl(\ttL_2), \rho, \ttd)$ decomposes as $\ttg(\fgl(\ttL_2), \rho, \ttd)=\ttf\oplus \fsl(\ttL_2)$ where $\ttf\in\Lie(\sVec)$. Furthermore,
\begin{itemize}[leftmargin=*,label=$\diamondsuit$]
\item for $k=1$, we get $\ttf \cong \fsl_2(\bk)$ when $a\ne0$, and $\ttf\cong \fH_3$ (the Heisenberg algebra) if $a=0$.
\item for $k=p-1$, we obtain $\ttf \cong \osp_(2\vert 1)$ in case $a\ne 0$, and $\ttf\cong\supersl{1}{1}$ when $a=0$.
\end{itemize}
\end{example}

More interesting examples are afforded by $\ttV=\ttL_k$ when $k\ne 1, p-1$.

\begin{remark}\label{rem:cont-data-Lk}
For $k\ne 1, p-1$, the set of reduced $\fgl(\ttL_2)$-contragredient structures on $\ttL_k$ is in bijection with 
$$ \left\{(a,\ati,b, \bti) | a,b,\bti \in\Bbbk, \ati\in\{0,1\}, \text{either }b\ne 0 \text{ or }\bti \ne 0\right\}.$$

Indeed, after fixing direct sum decompositions, a map $\rho\colon\fgl(\ttL_2)\to\fgl(\ttL_k)$ in $\Verp$ is just a $2\times2$ diagonal matrix $\rho=\op{diag}(a \id_\one, \ati\id_{\ttL_3})$. For this map to preserve Lie brackets, we need $\ati^2=\ati$, while there is no restriction on $a$ because $\one$ is abelian. 

Similarly, a map $\ttd\colon\ttL_k\otimes\ttL_k^*\to\fgl(\ttL_2)$ in $\Ver_p$ boils down to $\ttd=\op{diag}(b \id_\one, \bti\id_{\ttL_3})$, and since $\ttL_k$ is simple, $\ttd$ is reduced precisely when it is non-zero. 
Now one can check in $\Rep(\balp_{p})$ that, for any choices of $\rho$ and $\ttd$ as above, the compatibility \eqref{eq:contragredient-compatibility} always holds. 

Up to equivalence of the associated Lie algebras, we can normalize these scalars and thus consider fewer options. In the next table, we organize $\fgl(\ttL_2)$-contragredient structures $(a,\ati,b, \bti)$ on $\ttL_k$ according to the different possibilities for the image of $\ttd$ and for the annihilator of the $\fgl(\ttL_2)$-action, viewed as a map $\rho\colon\fgl(\ttL_2)\to\fgl(\ttL_k)$. Here $\rho=\op{diag}(a \id_\one, \ati\id_{\ttL_3})$ and  $\ttd=\op{diag}(b \id_\one, \bti\id_{\ttL_3})$.

\begin{center}
\begin{tabularx}{0.8 \textwidth}{>{\raggedright\arraybackslash}X | >{\centering\arraybackslash}X >{\centering\arraybackslash}X >{\centering\arraybackslash}X  }
& $\ttd$ lands in $\one$& $\ttd$ lands in $\fsl(\ttL_2)$ & Full rank $\ttd$\\
\hline
$\ker \rho=\fgl(\ttL_2)$&$(0,0,1,0)$ & $(0,0,0,1)$ & $(0,0,b,1)$ \\
$\ker \rho=\one$& $(0,1,1,0)$& $(0,1,0,1)$ & $(0,1,b,1)$ \\
$\ker \rho=\fsl(\ttL_2)$& $(1,0,1,0)$& $(1,0,0,1)$& $(1,0,b,1)$\\
$\ker \rho=0$ & $(1,1,1,0)$& $(1,1,0,1)$ &$(1,1,b,1)$
\end{tabularx}
\end{center}
\end{remark}

\begin{remark}\label{rem:contr-data-decomposition}
\begin{enumerate}[leftmargin=*]
\item The cases in the first row satisfy $\rho=0$ and were already considered in \Cref{cor:trivial-action}.
\item The cases in the first ($\bti=0$) and second ($b=0$) columns can be reduced to smaller tori. Indeed, these are covered by \Cref{rem:contr-data-semi-direct} applied to $\ttX=\fgl(\ttL_2)=\one\oplus \fsl(\ttL_2)$ as follows: 
\begin{itemize}[leftmargin=*]
\item For $\bti=0$, we take $\ttX_1=\one$, $\ttX_2=\fsl(\ttL_2)$, and get $\ttg(\ttX,\rho,\ttd) \simeq \ttg(\one,\rho_1,\ttd)\rtimes \fsl(\ttL_2)$.
We study the algebras $ \ttg(\one,\rho_1,\ttd)$ in \S \ref{subsubsec:one}.
Note moreover, that if also $\ati=0$, then $\ttg=\ttg'\oplus \fsl(\ttL_2)$ as Lie algebras.

\item For $b=0$, we choose $\ttX_1=\fsl(\ttL_2)$, $\ttX_2=\one$ and get $\ttg(\ttX,\rho,\ttd) \simeq \ttg(\fsl(\ttL_2),\rho_1,\ttd)\rtimes \one$.
The algebras $\ttg(\fsl(\ttL_2),\rho_1,\ttd)$ will be  considered in \S \ref{subsubsec:sl-L2}.
If in addition $a=0$, then $\ttg=\ttg'\oplus \one$ as Lie algebras.
\end{itemize}
\end{enumerate}
\end{remark}

\subsubsection{Lie algebras over \texorpdfstring{$\ttX=\one$}{}}\label{subsubsec:one}
Reduced $\one$-contragredient structures on $\ttL_k$ correspond to pairs of scalars $(a,b)$ with non-zero $b$. In fact, an action $\rho\colon\one\otimes\ttL_k\simeq \ttL_k\to\ttL_k$ is given by $a\id_{\ttL_k}$, and $\ttd\colon\ttL_k\ot\ttL_k^*\to\one$ is just $b\ev_{\ttL_k}$, with non-zero $b$ for reduced structures. By \Cref{cor:trivial-action} we may also assume $a\ne 0$. Thus up to normalization we may fix $a=1$ and $b=1$, and denote the corresponding contragredient  Lie algebra by $\ttg(\one,1,1)$

In what follows, the notation $\ttL_i^{(d)}$ means a copy of $\ttL_i$ siting in degree $d\in\bZ$ of $\ttg(\one,1,1)$. 

\begin{lemma}\label{lem:contragredient-Li-one}
The $\bZ$-grading of the Lie algebra $\ttg=\ttg(\one,1,1)$ satisfies
\begin{align*}
\ttg&=\left(\oplus_{j \geq 2}\ttg_j\right)\oplus\ttL_k^{(1)}\oplus\one\oplus\ttL_k^{(-1)}\oplus\left(\oplus_{j \geq 2}\ttg_{-j}\right), & 
\text{where }\ttg_{\pm2}&=\oplus_{\ell\ge0}\ttL_{2(k-\ell)-3}^{(\pm2)}.
\end{align*}
The bracket satisfies the following identities:
\begin{align*}
\restr{\ttb}{\one\otimes\ttL_{2i-1}^{(\pm 2)}} &= \pm 2\id_{\ttL_{2i-1}^{(\pm 2)}}, &
\restr{\ttb}{\ttL_{2i-1}^{(\pm 2)}\otimes \ttL_k^{(\mp1)}} &= \mp 2\iota_{\ttL_k^{(\pm1)}}, & 
\restr{\ttb}{\ttL_{2i-1}^{(2)}\otimes\ttL_{2j-1}^{(-2)}} &= \delta_{ij}\ev_{\ttL_{2i-1}^{(2)}}.
\end{align*}
\end{lemma}
\begin{proof}
By Theorem \ref{thm:gtilde-verp}, $\tttg=\tttn_-\oplus\one\oplus\tttn_+$, with $\tttn_{\pm}=\FLie(\ttL_k^{\pm 1})$. Thus, by Remark \ref{rem:FOLie-Li},
\begin{align*}
\tttg_{\pm2}&=\oplus_{\ell\ge0}\ttL_{2(k-\ell)-3}^{(\pm2)}=\ttb(\ttL_k^{(\pm 1)} \ot \ttL_k^{(\pm1)}). 
\end{align*}
The action of $\ttX=\one$ on $\op{T}(\ttL_k^{(\pm 1)})$ is by derivations, so $\restr{\ttb}{\one\otimes\ttL_{2i-1}^{(\pm 2)}}= \pm 2\id_{\ttL_{2i-1}^{(\pm 2)}}$.

We want to compute $\restr{\ttb}{\tttg_{2}\ot \ttL_k^{(-1)}}$.
As $\tttg_2=\ttb(\ttL_k^{(1)}\ot\ttL_k^{(1)})$ we look at 
$$\ttb(\ttb\ot\id)\colon\ttL_k^{(1)}\ot\ttL_k^{(1)}\ot\ttL_k^{(-1)}\to\ttL_k^{(1)}.$$
By the Jacobi identity and definition of $\ttb$, we obtain
\begin{align}\label{eq:bracket-Lk-Lk-Lkdual}
\restr{\ttb(\ttb\ot\id)}{\ttL_k^{(1)}\ot\ttL_k^{(1)}\ot\ttL_k^{(-1)}}=-\rho\, c_{\ttL_k^{(1)},\one}(\id\ot\ev)+\rho(\ev\ot\id)(\id\ot c_{\ttL_k^{(1)},\ttL_k^{(-1)}}).
\end{align}
Now, via de adjunction $\Hom(\ttL_k^{(1)}\ot\ttL_k^{(1)}\ot\ttL_k^{(-1)},\ttL_k^{(1)})\simeq \Hom(\ttL_k^{(1)}\ot\ttL_k^{(1)},\ttL_k^{(1)}\ot\ttL_k^{(1)})$, $f\mapsto (f\ot \id)(\id\ot\coev_{\ttL_k^{(1)}})$, the map in \eqref{eq:bracket-Lk-Lk-Lkdual} becomes $-\id+c=-2 \iota_{\tttg_2}$. Thus $\restr{\ttb}{\ttL_{2(k-\ell)-3}^{(2)} \ot \ttL_k^{(-1)}} \ne 0$, so $\ttL_{2(k-\ell)-3}^{(2)}\ne0$ in $\ttg$.
%
%
\end{proof}

\begin{remark}\label{rem:sl2-osp21-generalization}
The Lie algebra $\ttg$ partially described in \Cref{lem:contragredient-Li-one} generalizes the rank-one Lie superalgebras with non-trivial action of the one-dimensional torus (which is necessarily even). Indeed, in the case $k=1$ of Lie algebras, $\ttg$ recovers $\fsl_2$, which is trivial in degree $2$ since there is no antisymmetric part in the rank-one free Lie algebra. For $k=p-1$ we obtain $\osp(2|1)$, which only vanishes in degrees $\geq3$, since the free Lie superalgebra contains an antisymmetric component of dimension $1$ (generated by $[e,e]\ne 0$) in degree $2$, but after that $[e,[e,e]]=0$.
\end{remark}

\begin{example}
Let $p=5$. We explain a way to get the Lie algebra $\ttg$ in \Cref{lem:contragredient-Li-one} for $k=2$ by means of the semisimplification of an appropriate Lie algebra.\footnote{We thank Pavel Etingof for discussions about this computation.}

Let $A=\begin{pmatrix} 2 & 2 \\ 2 & 2\end{pmatrix}$, $\tfg:=\tfg(A)$, the Lie algebra presented by generators $e_i$, $h_i$, $f_i$, $i=1,2$, and relations
\begin{align*}
[h_i,e_j] &= 2e_j, & [h_i,f_j] &= -2f_j, & [e_i,f_j] &= \delta_{i,j|}h_i, & [h_i,h_j] &=0, & i,j &= 1,2.
\end{align*}
As usual, $\tfg$ has a triangular decomposition $\tfg=\tfn_+\oplus\fh\oplus\tfn_-$, where $\tfn_+$, respectively $\tfn_-$ is the Lie subalgebra generated by the $e_i$'s, respectively the $f_i$'s, and is free on these generators, and $\fh$ is the subalgebra generated by the $h_i$'s, which is abelian. In addition $\tfg$ is $\bZ^2$-graded, 
with $e_i$ in degree $\alpha_i$, $f_i$ in degree $-\alpha_i$ and $h_i$ in degree $0$.
The contragredient Lie algebra $\fg:=\fg(A)$ is the quotient $\fg=\tfg/\tau$, where $\tau$ is the maximal $\bZ$-graded ideal of $\tfg$ such that $\tau\cap\fh=0$. Then $\fg$ is also a $\bZ^2$-graded Lie algebra admitting a triangular decomposition $\fg=\fn_+\oplus\fh\oplus\fn_-$, both structures induced from $\tfg$. Notice that
\begin{align}\label{eq:ver5-serre-rels}
(\ad e_i)^4 e_j, (\ad f_i)^4 f_j \in \tau  \text{ for }i\ne j.
\end{align}
Also, $h_1-h_2$ is a central element in degree $0$, so we can take the quotient $\ufg:=\fg/(h_1-h_2)$ to get a Lie algebra with analogous properties, but a one-dimensional degree $0$ component $\underline{\fh}$ generated by $h$, the class of $h_i$.

\smallbreak

There exists a Lie algebra isomorphism $\phi:\tfg\to\tfg$ such that
\begin{align*}
\phi(e_1) &= e_1, & \phi(e_2) &= e_2+e_1, & \phi(f_1) &= f_1-f_2, & \phi(f_2) &= f_2, & \phi(h_i) &= h_i, & i&=1,2.
\end{align*}
As $\phi$ is $\bZ$-graded, $\phi(\tau)=\tau$, so it induces a Lie algebra isomorphism $\phi:\fg\to\fg$, which in turn induces a Lie algebra isomorphism $\phi:\ufg\to\ufg$.

The semisimplification of $\ufg$ is a graded Lie algebra in $\Ver_5$. Let $\ttg$ be the subalgebra generated by the components of degrees $1,0,-1$. Then $\ttg=\oplus_{n\in\bZ}\ttg_n$ is also graded, with $\ttg_0=\one$, $\ttg_{\pm1}\simeq\ttL_{2}$: As the non-degenerate symmetric invariant bilinear form on $\ufg$ induces an analogous bilinear form on its semisimplification which restricts to $\ttg$, we have that $\ttg \simeq \ttg(\one,1,1)$ (with $k=2$).

\smallbreak

Next we compute recursively a basis of $\fg_j\simeq \ufg_j$ for $j\in\{2,3,4,5,6\}$ as well as the bracket of these elements with $f_i$ and the images under $\phi$. We make use of the basis of $\tfn_+$ made by appropriate elements $[\ell]$ attached to Lyndon words $\ell$ in letters $e_1<e_2$ \cite{Reutenauer-book}, see \Cref{tab:char5-g(A)-basis}. In addition, and using that $\fg$ is $\bZ^2$-graded, we get the following relations:
\begin{align}\label{eq:ver5-g-new-rels}
[e_1^2e_2^3]&=e_{1212^2}
&
[e_1^3e_2^2] &=2e_{1^2212}.
\end{align}

\begin{table}[ht]
\caption{Basis of $\fg(A)$} \label{tab:char5-g(A)-basis}
\begin{center}
\begin{tabular}{|c| c | c c |c|}
\hline degree & element & $[-,f_1]$  & $[-,f_2]$ & $\phi$ 
\\ \hline
2 & $e_{12}$ & $2e_{2}$ & $3e_{1}$ & $e_{12}$
\\ \hline
3 & $e_{1^22}$ & $e_{12}$ & $0$ & $e_{1^22}$
\\ \cline{2-5}
& $e_{12^2}$ & $0$ & $-e_{12}$ & $e_{12^2}-e_{1^22}$
\\ \hline
4 & $e_{1^32}$ & $2e_{1^22}$ & $0$ & $e_{1^32}$
\\ \cline{2-5}
& $e_{1^22^2}$ & $e_{12^2}$ & $-e_{1^22}$ & $e_{1^22^2}-e_{1^32}$
\\ \cline{2-5}
& $e_{12^3}$ & $0$ & $3e_{12^2}$ & $e_{12^3}-2e_{1^22^2}+e_{1^32}$
\\ \hline
5 & $e_{1^2212}$ & $2 e_{1^22^2}$ & $2e_{1^32}$ & $e_{1^2212}$
\\ \cline{2-5}
& $e_{1212^2}$ & $3 e_{12^3}$ & $3 e_{1^22^2}$ & $e_{1212^2}+e_{1^2212}$
\\ \hline
6 & $e_{1^3212}$ & $-e_{1^2212}$ & $0$ & $e_{1^3212}$
\\ \cline{2-5}
& $e_{1^22^212}$ & $0$ & $e_{1^2212}$ & $e_{1^22^212}-e_{1^3212}$
\\ \cline{2-5}
& $e_{1^2212^2}$ & $3 e_{1212^2}$ & $e_{1^2212}$ & $e_{1^2212^2}+e_{1^3212}$
\\ \cline{2-5}
& $e_{1212^3}$ & $0$ & $-e_{1212^2}$ & $e_{1212^3}+2e_{1^22^212}-e_{1^3212}$
\\
\hline
\end{tabular}
\end{center}
\end{table}

A straightforward but long computation using the data from the table can be performed to get the following relations:
\begin{align*}
[e_{1^22^212},f_{12}] &= -[e_{1^2212^2},f_{12}] = 3e_{1^22^2},
&
[e_{\alpha},f_{\beta}] &=3h_1+3h_2, & \alpha,\beta &\in\{1^22^212,1^2212^2\}.
\end{align*}

We claim that $\ttg$ does not belong to $\Ver_5$ (i.e. it is a direct sum of infinitely many simple objects). 
To do so, we will check that the Lie subalgebra $\fk$ of $\ttg$ generated by 
\begin{align*}
\bx_1&\coloneqq e_{12}, & \bx_2 &\coloneqq e_{1^22^212}+e_{1^2212^2}, & \by_1&\coloneqq 2f_{12}, & \by_2&\coloneqq 3f_{1^22^212}+3f_{1^2212^2}
\end{align*}
is infinite dimensional. Notice that $\fk$ is a Lie algebra (in $\Vec$), which is $\bZ$-graded: $\fk=\oplus_{n\in\bZ} \fk_n$. 
Set $\bz\coloneqq 3h$. Using the relations above we check the following equalities hold:
\begin{align*}
[\bx_1,\by_1] &= \bz, & 
[\bx_1,\by_2] &= 0, & 
[\bz,\bx_1] &= 2\bx_1,
\\
[\bx_2,\by_1] &= 0, & 
[\bx_2,\by_2] &= 2\bz, &
[\bz,\bx_2] &= \bx_2,
\end{align*}
Here, $[\bx_2,\by_1]=0$ can be computed using the relations above while $[\bx_1,\by_2]$ is the image of $[\bx_2,\by_1]$ by the Cartan involution, so it is 0 as well.

Let $B=\begin{pmatrix} 2 & -4 \\ -1 & 2\end{pmatrix}$. The associated Lie algebra $\tfg(B)$ is presented by generators $x_i$, $y_i$, $z_i$, and relations
\begin{align*}
[z_i,x_j]&=b_{ij} x_j, & [z_i,y_j]&=-b_{ij} y_j, & [x_i,y_j]&=\delta_{ij} z_j, & [z_i,z_j]&=0.
\end{align*}
The element $z_1-3z_2$ is central and of degree $0$, so the quotient $\tfg\coloneqq \tfg(B)/\langle z_1-3z_2 \rangle$ is also a $\bZ$-graded Lie algebra with the \emph{same} homogeneous components of degree $n\ne 0$.

There exists an epimorphism $\tfg(B) \twoheadrightarrow \fk$ of $\bZ$-graded Lie algebras such that
\begin{align*}
z_1&\mapsto \bz, & z_2&\mapsto 2\bz, & 
x_j&\mapsto \bx_j, & y_j&\mapsto \by_j, & j &=1,2.
\end{align*}
which induces to an epimorphism $\pi:\tfg \twoheadrightarrow \fk$ of $\bZ$-graded Lie algebras. As the degree 0 component of both Lie algebras is one-dimensional, $\ker \pi$ is a $\bZ$-homogeneous ideal contained in $\oplus_{n\ne 0}\tfg_n$. 
Let $\varpi:\tfg(B)\twoheadrightarrow\tfg$ be the canonical projection: We can identify $\tfg(B)_n=\tfg_n$ for $n\ne 0$ under this map. Doing so, $\ker\pi$ is a $\bZ$-homogeneous ideal of $\tfg(B)$ which intersects trivially $\tfg(B)_0$, so $\ker\pi$ annihilates in $\fg(B)$, and then in $\fg\coloneqq\fg(B)/(z_1-3z_2)$. Thus the canonical projection $\tfg\twoheadrightarrow\fg$ induces an epimorphism $\fk \twoheadrightarrow \fg$ of $\bZ$-graded Lie algebras.

\begin{align*}
\xymatrix@C=30pt@R=5pt{
	\tfg(B)\ar@{->>}[rr] \ar@{->>}[rd]\ar@{->>}[dd]_{\varpi} && \fg(B) \ar@{->>}[dd] 
	\\
	& \fk \ar@{.>>}[rd] &
	\\
	\tfg \ar@{->>}[rr] \ar@{->>}[ru]^{\pi} &  & \fg. }
\end{align*}
The Lie algebra $\fg(B)$ is infinite dimensional since $B$ is an affine Cartan matrix. Thus $\dim\fg=\infty$ as well, and then $\dim\fk=\infty$.
\end{example}

We conclude our work with open problems about other two families of examples.

\subsubsection{Lie algebras over \texorpdfstring{$\ttX=\fsl(\ttL_2)$}{}}\label{subsubsec:sl-L2}
By the discussion above, if $k=1$ or $k=p-1$, then the unique $\fsl(\ttL_2)$-contragredient structure $(\rho,\ttd)$ on $\ttV_k$ is the trivial one $(\rho,\ttd)=0$. On the other hand, if $k\ne 1, p-1$, the set of $\fsl(\ttL_2)$-contragredient structures on $\ttL_k$ is in bijection with pairs of scalars $(a,b)$ with $a=0$ or $a=1$.


Now, up to isomorphism on the corresponding contragredient Lie algebras, we only need to consider four possibilities for $(a,b)$, namely $$(0,0), \qquad (1,0), \qquad (0,1), \qquad (1,1).$$ 
By \Cref{exa:contragredient-d-trivial}, the non-reduced data $(0,0)$ and $(1,0)$ give just $\ttg(0,0)=\ttg(1,0)=\fsl(\ttL_2)$.

Also, by \Cref{cor:trivial-action}, we have $\ttg(0,1)=\ttL_k^*\oplus\fsl(\ttL_2) \oplus \ttL_k$ where $\fsl(\ttL_2)$ commutes with $\ttL_k^*\oplus\ttL_k$ and the bracket between $\ttL_k^*$ and $\ttL_k$ is the canonical projection $\ttL_k^* \ot \ttL_k \to \ttL_3=\fsl(\ttL_2)$.

The only remaining case is $(1,1)$. A detailed study of the corresponding Lie algebra $\ttg(\fsl(\ttL_2), 1,1)$ is left for future work.

\subsubsection{Lie algebras over \texorpdfstring{$\ttX=\fgl(\ttL_2)$}{}}\label{subsubsec:gl-L2}

As explained in Remark \ref{rem:contr-data-decomposition}, it only remains to study contragredient data such that $\ttd$ has full rank, i.e. $\ima\ttd=\fgl(\ttL_2)$, and the actions of $\one$ and $\fsl(\ttL_2)$ are both non-trivial. 
In this case, the calculations become more involved, since the actions of the two subalgebras can combine to cancel each other out starting from different degrees, which depend on the scalar $\bti$ (following the notation from \Cref{rem:cont-data-Lk}). As an example of this phenomenon, note that suitable semisimplifications of Lie algebras of type $A_2$, $B_2$ and $G_2$ as in \Cref{subsection:semisimplification-examples} fall in this context for $\ttV=\ttL_2$. However, different scalars $\bti$ yield  Lie algebras that differ drastically, as their highest $\bZ$-degree are $1$, $2$, and $3$ respectively. We leave the determination of the structure of these Lie algebras for future work.

\bibliography{biblio}

\end{document}